\documentclass[11pt,namelimits,sumlimits]{amsart}
\usepackage{amssymb,amsmath}
\usepackage[mathscr]{eucal}
\usepackage{comment}
\usepackage{hyperref}
\usepackage{color}

\usepackage{enumitem}
 
\setlist{leftmargin=*}

\usepackage{psfrag}

\usepackage{graphicx}
\usepackage{amsmath,amsthm}
\usepackage{graphics}
\usepackage{color}
\usepackage{epsfig}
\usepackage{fullpage}
\usepackage{amssymb,amsmath}
\usepackage[mathscr]{eucal}

\usepackage[usenames,dvipsnames]{pstricks}
\usepackage{epsfig}
\usepackage{pst-grad} 
\usepackage{pst-plot} 

\newtheorem{theorem}[equation]{Theorem}
\newtheorem{lemma}[equation]{Lemma}
\newtheorem{prop}[equation]{Proposition}
\newtheorem{corollary}[equation]{Corollary}

\newtheorem{definition}[equation]{Definition}

\theoremstyle{remark}
\newtheorem{remark}[equation]{Remark}
\newtheorem{notation}[equation]{Notation}
\newtheorem{convention}[equation]{Convention}
\newtheorem{assumption}[equation]{Assumption}

\numberwithin{equation}{section}

\newcommand{\phie}{{\phi_{even}}}
\newcommand{\phieder}{{\phi'_{even}}}
\newcommand{\phio}{{\phi_{odd}}}

\newcommand{\R}{\mathbb{R}}
\newcommand{\N}{\mathbb{N}}
\newcommand{\B}{\mathbb{B}}
\newcommand{\K}{\mathbb{K}}
\newcommand{\Ktilde}{\underline{\mathbb{K}}}
\newcommand{\Khat}{\widehat{\mathbb{K}}}
\newcommand{\D}{\mathbb{D}}
\newcommand{\Z}{\mathbb{Z}}
\newcommand{\A}{\mathbb{A}}
\newcommand{\Hyp}{\mathbb{H}}
\newcommand{\Sph}{\mathbb{S}}
\newcommand{\ThetaSph}{\Theta_{\Sph^2}}

\newcommand{\Fcal}{\mathcal{F}}
\newcommand{\Wcal}{\mathcal{W}}
\newcommand{\Wcaltilde}{\underline{\mathcal{W}}}
\newcommand{\Circle}{\mathcal{C}}
\newcommand{\Munder}{\underline{M}}
\newcommand{\Ctheta}{{\Circle_\theta}}
\newcommand{\Sigmatilde}{\Munder}
\newcommand{\Lcal}{\mathcal{L}}
\newcommand{\Bcal}{\mathcal{B}}
\newcommand{\Hcal}{\mathcal{H}}
\newcommand{\Rcal}{\mathcal{R}}
\newcommand{\Zcal}{\mathcal{Z}}
\newcommand{\Zcalhat}{\widehat{\mathcal{Z}}} 
\newcommand{\Jcal}{\mathcal{J}}
\newcommand{\Tcal}{\mathcal{T}}
\newcommand{\Sch}{\mathcal{S}}
\newcommand{\Schhat}{\widehat{\mathcal{S}}}

\newcommand{\Zscr}{\mathscr{Z}}
\newcommand{\Bscr}{\mathscr{B}}

\newcommand{\Ehat}{\widehat{E}}
\newcommand{\uhat}{\widehat{u}}

\newcommand{\psihat}{{\widehat{\psi}}}

\newcommand{\PiW}{\Pi_{\Wcal}}
\newcommand{\PiWtheta}{\Pi_{\Wcal_\theta}}
\newcommand{\PiSch}{\Pi_{\Sch}}

\newcommand{\xx}{\ensuremath{\mathrm{x}}}
\newcommand{\yy}{\ensuremath{\mathrm{y}}}

\newcommand{\YYbar}{\underline{\mathsf{Y}}}

\newcommand{\YYhbar}{\widehat{\underline{\mathsf{Y}}}}

\newcommand{\Cbar}{\underline{C}}
\newcommand{\abar}{\underline{a}}

\newcommand{\dtau}{\delta_\tau}
\newcommand{\dtheta}{\delta_\theta}
\newcommand{\ds}{\delta_s}
\newcommand{\ep}{\epsilon}

\newcommand{\group}{\mathscr{G}}

\newcommand{\phitilde}{{\widetilde{\varphi}}}
\newcommand{\thetatilde}{{\widetilde{\theta}}}
\newcommand{\Mtilde}{\widetilde{M}}

\newcommand{\rhotilde}{\widetilde{\rho}}
\newcommand{\tM}{\widetilde{M}}
\newcommand{\tX}{\widetilde{X}}
\newcommand{\tSi}{\widetilde{\Sigma}}

\newcommand{\tF}{\widetilde{F}}

\newcommand{\tnu}{\widetilde{\nu}}

\newcommand{\arcsinh}{\operatorname{arcsinh}}

\newcommand{\inj}{\operatorname{inj}}

\newcommand{\id}{\operatorname{Id}}
\newcommand{\Graph}{\operatorname{Graph}}
\newcommand{\Ric}{\operatorname{Ric}}
\newcommand{\Immer}{\operatorname{Immer}}
\newcommand{\area}{\operatorname{area}}

\newcommand{\beq}{\begin{equation}}
\newcommand{\eeq}{\end{equation}}

\begin{document}

\title[Free Boundary Minimal Surfaces in the Unit Three-Ball via Desingularization]{Free Boundary Minimal Surfaces in the Unit Three-Ball via Desingularization of the Critical Catenoid and the Equatorial Disk}


\author[N. Kapouleas]{Nikolaos Kapouleas}
\address{Department of Mathematics, Brown University, Providence, RI 02912, USA}
\email{nicos@math.brown.edu}

\author[M. Li]{Martin Man-chun Li}
\address{Department of Mathematics, The Chinese University of Hong Kong, Shatin, Hong Kong} 
\email{martinli@math.cuhk.edu.hk}




\date{\today}


\begin{abstract}
We construct a new family of high genus examples of 
free boundary minimal surfaces in the Euclidean unit 3-ball 
by desingularizing the intersection of a coaxial pair of a critical catenoid and an equatorial disk. 
The surfaces are constructed by singular perturbation methods and have three boundary components. 
They are the free boundary analogue of the 
Costa-Hoffman-Meeks surfaces and the surfaces constructed by 
Kapouleas by desingularizing coaxial catenoids and planes. 
It is plausible that the minimal surfaces we constructed here are the same as the ones obtained recently by Ketover in \cite{Ketover16} using min-max method. 
\end{abstract}

\maketitle



\section{Introduction}
\label{S:intro}

Minimal surfaces have been a central object of study in differential geometry. 
They are defined as  critical points to the area functional in a Riemannian manifold. 
These minimal surfaces are interesting as they reveal important information about the geometry of the underlying spaces. 
For example, this idea has led to much success in the study of spaces with positive scalar or Ricci curvature 
(for instance, see \cite{Meeks-Simon-Yau82} \cite{Schoen-Yau79} \cite{Schoen-Yau82}). 
On the other hand, the theory of minimal surfaces is highly non-trivial even when the underlying space is homogeneous (e.g. $\R^n$, $\Sph^n$ and $\Hyp^n$). 
The solution to the classical Plateau problem guarantees the existence of (immersed) minimal disks with any prescribed Jordan curve in $\R^3$ as its boundary. 
For a long time, the only known embedded complete minimal surfaces of finite total curvature in $\R^3$ were the planes and the catenoids. 
It was a groundbreaking discovery when Costa \cite{Costa84} found new examples using Weierstrauss representation 
which were later proved to be embedded by Hoffman-Meeks \cite{Hoffman-Meeks85}. 
More examples were then found \cite{Hoffman-Meeks88} \cite{Hoffman-Meeks90a} \cite{Wohlgemuth91} \cite{Hoffman-Meeks90b},  
and Hoffman-Meeks effectively recognized them as desingularizations of a catenoid intersecting a plane through the waist \cite{Hoffman-Meeks90c}.  
Finally N.K. \cite{Kapouleas97} 
(see also \cite{Kapouleas05} and \cite{Kapouleas11} for a discussion of the approach and further developments),  
provided a more general construction for complete embedded minimal surfaces of finite total curvature 
in Euclidean 3-space by desingularizing intersecting coaxial catenoids and planes using the singular perturbation method. 

For the case of the round three-sphere $\Sph^3(1)$, 
Lawson \cite{Lawson70} constructed embedded closed minimal surfaces of arbitrary genus 
which were the first examples besides the round equatorial sphere and the Clifford torus. 
In retrospect the Lawson surfaces can be recognized as desingularizations 
(carried out by non-perturbative methods) 
of intersecting equatorial spheres symmetrically arranged around a great circle of intersection. 
Karcher-Pinkall-Sterling \cite{Karcher-Pinkall-Sterling88} employed Lawson's method to construct finitely many closed embedded minimal surfaces 
which can be interpreted as doublings of the equatorial $\Sph^2$ in $\Sph^3(1)$.
More recently further examples of embedded closed minimal surfaces in $\Sph^3(1)$ 
have been obtained by Lawson's method \cite{Choe-Soret16}  
and by singular perturbation methods 
\cite{Kapouleas-Yang10}, \cite{Kapouleas14}, \cite{Wiygul15} and \cite{Kapouleas-Wiygul}. 

If the ambient space has a boundary it is natural to search for critical points among 
the class of immersed surfaces whose boundary lies on the boundary of the ambient space. 
Such critical points are called free boundary minimal surfaces and they 
meet the boundary of the ambient space orthogonally along their boundary. 
The simplest example is the equatorial flat disk $\D$ in the Euclidean three-ball $\B^3$ or more generally 
the Euclidean $n$-ball $\B^n$. 
$\D$ is the unique (immersed) free boundary minimal disk in $\B^3$ by a result of Nitsche \cite{Nitsche85}, 
and by a surprising recent result of Fraser-Schoen \cite{Fraser-Schoen15a} 
the unique free boundary minimal disk in $\B^n$ for any $n\ge3$. 
The next non-trivial example is the so-called critical catenoid $\K$ (see \cite{Fraser-Schoen11}), 
which is a catenoid in $\R^3$ suitably rescaled so that it meets $\Sph^2(1)$ orthogonally. 
As we will later check in this article 
(see corollary \ref{C:unique-catenoid})
$\D$ and $\K$ are the only rotationally symmetric free boundary minimal surfaces in $\B^3$. 
In some sense they are analogous to the equatorial sphere and the Clifford torus in $\Sph^3(1)$. 

The first study of free boundary minimal surfaces was done by R. Courant \cite{Courant40}, 
and the existence and regularity theory was subsequently developed by Nitsche \cite{Nitsche76}, 
Taylor \cite{Taylor77}, Hildebrant-Nitsche \cite{Hildebrandt-Nitsche79}, 
Gr\"{u}ter-Jost \cite{Gruter-Jost86} and Jost \cite{Jost86}. 
A fundamental question is to classify the free boundary minimal surfaces, 
or at least understand the existence and uniqueness questions, as in the following:

\textbf{Question}: \textit{Given a smooth compact domain $\Omega$ in $\mathbb{R}^n$, 
or more generally a compact Riemannian manifold with boundary, 
what are the (immersed or embedded) free boundary minimal surfaces 
(that is meeting $\partial \Omega$ orthogonally along their boundary)  
contained in $\Omega$?}

Some general existence results along this direction have been established in the past decade. 
For immersed solutions, the most general existence result was obtained by A. Fraser \cite{Fraser00} for disk type solutions, 
and later by Chen-Fraser-Pang \cite{Chen-Fraser-Pang15} for incompressible surfaces. 
For embedded solutions in compact $3$-manifolds, a general existence result using min-max constructions was obtained by the M.L. \cite{Li15}. 
The min-max theory for free boundary minimal hypersurfaces in the Almgren-Pitts setting is recently developed by the M.L. with X. Zhou \cite{Li-Zhou}, 
completing Almgren's program in search for minimal hypersurfaces in Riemannian manifolds with possibly non-empty boundary (without any convexity assumption!).
Free boundary minimal surfaces are important tools in studying Riemannian manifolds with boundary since their properties are greatly affected by the ambient geometry. 
For example, Fraser \cite{Fraser07} \cite{Fraser02} used index estimates to study the topology of Euclidean domains with $k$-convex boundary. 
Fraser-Li \cite{Fraser-Li14} proved a smooth compactness result for embedded free boundary minimal surfaces 
when the ambient manifold has nonnegative Ricci curvature and convex boundary. 

In a recent breakthrough, Fraser-Schoen \cite{Fraser-Schoen11} discovered a deep connection between free boundary minimal surfaces 
in the Euclidean unit ball $\B^n\subset\R^n$ and extremal metrics on compact surfaces with boundary associated with the Steklov eigenvalue problem. 
This has led to much research activity on free boundary minimal surfaces in $\B^n$ (especially when $n=3$).
In a follow-up article \cite{Fraser-Schoen16} Fraser-Schoen constructed new examples of embedded free boundary minimal surfaces 
with genus zero and arbitrary number of boundary components.  
Recently, more examples were constructed by Ketover \cite{Ketover16} using min-max method. 
The main result of this article is Theorem \ref{T:finalthm}
which clearly implies the following. 

\begin{theorem}
\label{T:mainthm}
For any $g \in \N$ sufficiently large, there exists an embedded, orientable, smooth, compact surface $\Sigma_g\subset\B^3$ 
which is a free boundary minimal surface in $\B^3$ and satisfies: 
\\ 
(i). $\partial\Sigma_g=\Sigma_g\cap \partial\B^3$ has three connected components.\\
(ii). $\Sigma_g$ has genus $g$ and is symmetric under a dihedral group with $4g+4$ elements .\\
(iii). As $g \to \infty$, the sequence $\{\Sigma_g\}$ converges in the Hausdorff sense to $\D \cup \K$. 
Moreover, the convergence is smooth away from the circle of intersection $\D \cap \K$. Hence, 
\[ \lim_{g \to \infty} \area(\Sigma_g)=\area(\K)+\pi \approx 8.37898
.\]
\end{theorem}

The methodology we follow originates with a gluing construction of R. Schoen for constant scalar 
curvature metrics \cite{Schoen88} and a 
gluing construction of N.K. for constant mean curvature surfaces \cite{Kapouleas90a}.
The methodology was systematized and refined further in order to carry out
a challenging gluing construction for Wente tori
which provided the first genus two counterexamples to a celebrated question of Hopf 
\cite{kapouleas:wente:announce,kapouleas:wente}.
(The genus one case had been resolved by Wente \cite{Wente} and any genus at least three by N.K. \cite{Kapouleas:jdg}).

More directly related to the construction in this article is the desingularization construction 
in \cite{Kapouleas97} of coaxial catenoids and planes. 
The construction in \cite{Kapouleas97} is based on the methodology developed in 
\cite{kapouleas:wente:announce,kapouleas:wente} 
and utilizes the $O(2)$ symmetry of the given configuration of catenoids and planes. 
\cite{Kapouleas97} effectively settles desingularization constructions by gluing 
in the presence of $O(2)$ symmetry in any setting, 
except of course for the idiosyncratic aspects related to each setting (as in this article for example).  
Note that a desingularization construction for intersecting planes parallel to a given line 
carried out indepedently by Traizet \cite{Traizet96} 
is inadequate for our purposes, 
because in his case the intersection curves are straight lines, 
and therefore the main difficulties of our construction are not present. 

Free boundary minimal surfaces in the unit ball appear to be much more rigid 
than complete minimal surfaces in Euclidean space. 
For example in stark constrast to the Euclidean case (see \cite{Kapouleas97}),  
there is only one possible rotationally invariant 
configuration of free boundary minimal surfaces, 
because as we have already mentioned 
the equatorial disk $\D$ and the critical catenoid $\K$ are the only rotationally symmetric 
free boundary minimal surfaces in $\B^3$ 
(see \ref{C:unique-catenoid} for the proof). 
The $\K\cup\D$ configuration corresponds 
in \cite{Kapouleas97} 
to the special case of a catenoid intersecting a plane through its waist. 
The $\K\cup\D$ and the catenoid with plane through the waist configurations 
share extra symmetries (compared to the general case in \cite{Kapouleas97}) 
which can be used to substantially simplify the construction and proof (see \ref{r:seven}). 
This is the case also for some recent desingularization constructions of self-shrinkers of the 
Mean Curvature Flow \cite{nguyen} \cite{kapouleas:kleene:moller} which are also based on \cite{Kapouleas97}. 

The article is self-contained and we have carefully simpified the construction and proof in \cite{Kapouleas97} 
to take advantage of the extra symmetries available.
For the interested reader we remark also that the proof of the main linear estimate 
(proposition \ref{P:linear-global}) 
is closer to the one in \cite{Kapouleas14} rather than the one in \cite{Kapouleas97}.
The approach in \cite{Kapouleas97} is more robust because the comparison with the model 
standard regions is only at the level of the lower spectrum of the linearized equations. 
The approach in \cite{Kapouleas14} is more streanlined and more demanding computationally 
because it is based on a more detailed comparison with the model standard regions 
at the level of actual solutions.  
Note also that in section \ref{S:deformation}
we discuss for future reference the boundary conditions 
in more detail and generality than strictly needed in this article. 

Finally we mention that in an article under preparation, 
we construct free boundary minimal surfaces of arbitrary high genus with connected boundary,  
and also ones with two boundary components,  
by desingularizing two disks intersecting orthogonally along a diameter of the unit three-ball. 
The intersecting disks configuration is clearly not rotationally invariant,  
and the symmetry group is small and independent of the genus. 
These features make that construction much harder,   
but we can overcome the difficulties by following the approach in \cite{kapouleas:compact} with appropriate modifications 
(see also \cite{Kapouleas05} and \cite{Kapouleas11} for a detailed outline of the construction 
and proof of the theorem in \cite{kapouleas:compact}).
That construction can be extended also to apply to the case of more than two disks 
symmetrically arranged around a common diameter by using higher order Karcher-Scherk towers as models.

\subsection*{Organization of the presentation}
In section \ref{S:deformation} we study properly immersed hypersurfaces and their deformations in a Riemannian manifold with boundary. 
The boundary angle $\Theta$ is defined and we establish a uniform estimate \ref{P:local-estimate-Theta} 
on the change of $\Theta$ when one of the hypersurfaces is perturbed 
to the twisted graph of a small function over it. 
We also prove a strengthened version of the corresponding estimate on the mean curvature in \ref{P:local-estimate-H}. 
In section \ref{S:rotation-free} 
we study in detail the geometry of the initial configuration $\K\cup\D$ and its 
perturbations. 
We also check the uniqueness of the critical catenoid $\K$ 
as the only non-flat rotationally symmetric free boundary minimal surface in $\B^3$ 
in corollary \ref{C:unique-catenoid}. 
We finally establish the triviality of the rotationally symmetric kernels for the linearized equations on the standard pieces. 
In section \ref{S:desing} 
we first construct and study the geometry of the desingularizing surfaces which will be used 
to replace a neighborhood of the circle of intersection of $\K$ and $\D$, 
and then the one parameter families of initial surfaces $M_{\theta,m}$ (for each large $m$).  
In section \ref{S:linear} 
we study the linearized free boundary minimal surface equation on our model surfaces and then apply this information 
to solve the linearized equation on the initial surfaces with suitable decay estimates. 
Finally, in section \ref{S:main}, we estimate the nonlinear error terms and prove our main theorem \ref{T:mainthm} by the standard Schauder fixed point argument.

\subsection*{Notations and conventions}

Throughout this article, $\R^3$ will denote the Euclidean 3-space with Cartesian coordinates $(x,y,z)$ with standard orientation and orthonormal basis $\{ e_x,e_y,e_z\}$. 
We also have:  
\begin{notation} 
\label{E:BS+}
As in the introduction, we will use $\B^n$ to denote the closed unit ball in $\R^n$ whose boundary is the unit sphere $\Sph^{n-1}$. 
We will use $B^n(r) \subset \R^n$ to denote the Euclidean open $n$-ball of radius $r$ centered at the origin.  
We also define 
$\R^n_\pm:=\{(x^1,x^2,...,x^n)\in \R^n: \pm x^n\ge0\}$, 
$\B^n_\pm:=\B^n\cap\R^n_\pm$, 
$\Sph^{n-1}_\pm:=\Sph^{n-1}\cap\R^n_\pm$,  
and $B^n_+(r):=B^n(r) \cap \R^n_+ $. 
Note that $B^n_+(r)$ as a manifold with boundary has $\partial B^n_+(r) =B^{n-1}(r)=B^{n}(r)\cap\{x^n=0\}$.  
We may omit $r$ when $r=1$ and we assume them all equipped with the Euclidean metric which we will denote by $g_0$. 
(See Figure \ref{fig:half-ball})
\end{notation} 

\begin{figure}[h]
\includegraphics[height=5cm]{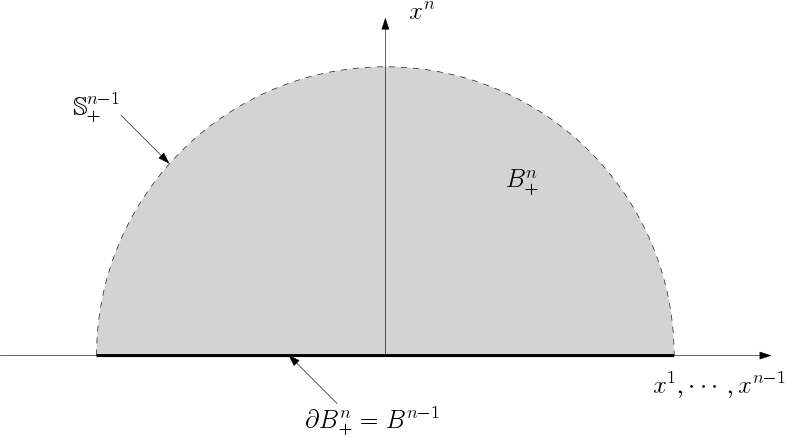}
\caption{The open unit half ball $B^n_+$ with its boundary $B^{n-1}$}
\label{fig:half-ball}
\end{figure}

Any surface $S \subset \R^3$ will be equipped with the induced metric $g$ (unless otherwise stated). 
We use the word ``surfaces'' to denote surfaces with or without boundary. 
We will often identify $\Sph^1$ with the quotient group $\R/2\pi \Z$. 
We always take the mean curvature $H$ of a surface in $\R^3$ to be the sum of its principal curvatures (so that the unit round sphere has mean curvature $2$).

\begin{notation} 
\label{n:exp}
If $(M,g)$ is a Riemannian manifold (without boundary) and $p\in M$ we will denote by $\exp_p^{M,g}$ the exponential map at $p$ 
(defined on the largest possible subset of $T_pM$) mapping to $M$.
We will denote by $\inj_p(M,g)$ the injectivity radius of $(M,g)$ at $p$.
In both cases we may omit $M$ or $g$ if clear from the context.
\end{notation}

Let $\Sigma$ be a smooth $n$-dimensional manifold (with or without boundary), and $(M,g)$ be a Riemannian manifold without boundary. Consider an immersion $X:\Sigma \to M$, we will use $X^*$ and $X_*$ to denote respectively the pullback of functions or tensors and the pushforward of vectors by the map $X$. For our purpose, we always assume that there exists a global unit normal $N: \Sigma \to TM$.

\begin{definition}
\label{d:imm}
Given an immersion $X:\Sigma \to M$ with unit normal $N:\Sigma \to TM$ as above,  
and a ``small enough'' function $\varphi$ defined on a domain $\Omega \subset \Sigma$. 
We define then the perturbation of $X$ by $\varphi$ over $\Omega$ 
(or of $\Omega$ by $\varphi$ when $X$ is an inclusion map) 
to be the map 
$\Immer[X,\varphi;\Omega] : \Omega \to M$ 
given by 
\beq
\Immer[X,\varphi;\Omega](p) :=\exp^{M,g}_{X(p)}(\varphi(p) N(p)) \quad \forall p\in\Omega. 
\eeq
We will also call the image of $\Immer[X,\varphi;\Omega]$ the graph of $\varphi$ over $X$ (or $\Omega$) 
and we will denote it by $\Graph[X,\varphi;\Omega]\subset M$.  
Finally we may omit $X$ when $X:\Sigma\to N$ is the inclusion map of an embedded $\Sigma\subset M$. 
\end{definition}

\begin{remark}
\label{R:graph} 
By ``small enough'' in the previous definition we mean any condition which ensures that 
$\exp^{M,g}_{X(p)}(\varphi(p) N(p))$ is well defined, 
as for example when $|\varphi(p)|<\inj_{X(p)}(M,g) $ (recall \ref{n:exp}) for each $p\in\Omega$.  
If $X$ is an inclusion of an embedded hypersurface $\Sigma\subset M$,  
and $\varphi$ is small---depending on $\Sigma$ this time---enough, 
then $\Immer[X,\varphi;\Omega]$ is the inverse of the nearest point projection to $\Omega$ 
restricted to $\Graph[X,\varphi;\Omega]$ . 
\end{remark}

We will be using extensively cut-off functions. To simplify the notation we introduce the following definition. 

\begin{definition}
\label{D:cutoff}
We fix a smooth function $\Psi:\R \to [0,1]$ with the following properties:\\
(i). $\Psi$ is weakly increasing.\\
(ii). $\Psi \equiv 0$ on $(-\infty,-1]$ and $\Psi \equiv 1$ on $[1,\infty)$.\\
(iii). $\Psi - \tfrac{1}{2}$ is an odd function.
\end{definition}

Given $a,b \in \R$ with $a \neq b$, we define the smooth function $\psi_{cut}[a,b]:\R \to [0,1]$ by
\beq
\label{E:cutab}
\psi_{cut}[a,b]:=\Psi \circ L_{a,b},
\eeq
where $L_{a,b}:\R \to \R$ is the unique linear function satisfying $L(a)=-3$ and $L(b)=3$. Clearly, the cutoff function $\psi_{cut}[a,b]$ satisfies the following properties (see Figure \ref{fig:cutoff}):\\
(i). $\psi_{cut}[a,b]$ is weakly monotone.\\
(ii). $\psi_{cut}[a,b]=0$ on a neighborhood of $a$ and $\psi_{cut}[a,b]=1$ on a neighborhood of $b$.\\
(iii). $\psi_{cut}[a,b]+\psi_{cut}[b,a]=1$ on $\R$.

\begin{figure}[h]
\includegraphics[height=6cm]{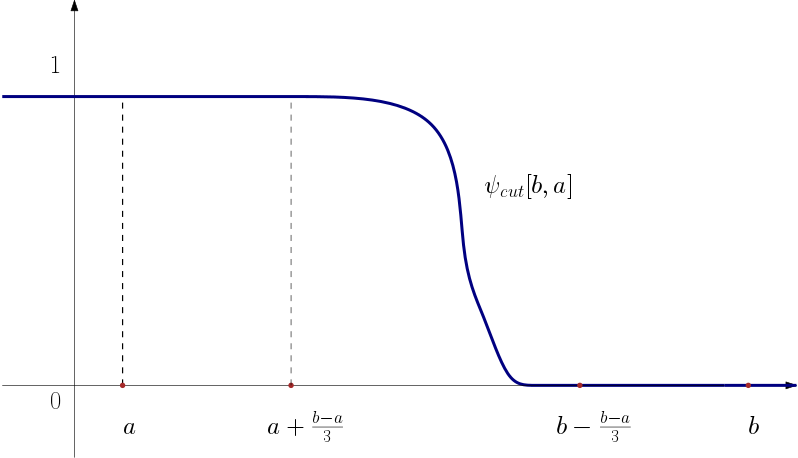}
\caption{A cutoff function $\psi_{cut}[b,a]$ for the case $a<b$}
\label{fig:cutoff}
\end{figure}

\begin{definition}
\label{D:cutoff-epsilon}
For each $\epsilon>0$, we define the symmetric cutoff function $\psi_{cut}^{\epsilon}: \R \to \R$ as
\[ \psi_{cut}^\epsilon := \psi_{cut}[\epsilon,0] \cdot \psi_{cut}[-\epsilon,0].\]
\end{definition}

Let $s:\Omega \to \R$ be a real-valued function defined on some domain $\Omega$. We will denote for any $c \in \R \cup \{ \pm \infty\}$,
\beq
\label{E:sub-super}
\Omega_{s \leq c} := \{ p \in \Omega : s(p) \leq c\} \qquad \text{ and } \qquad  \Omega_{s \geq c}:=\{p \in \Omega : s(p) \geq c\}.
\eeq

Suppose now we have two sections $f_0, f_1$ of some vector bundle over $\Omega$. We define a new section
\begin{equation}
\label{E:glue}
\Psi[a,b;s](f_0,f_1):=(\psi_{cut}[a,b] \circ s )\, f_1 + (\psi_{cut}[b,a] \circ s )\, f_0.
\end{equation}
Note that $\Psi[a,b;s](f_0,f_1)$ is a section of the same vector bundle, which is bilinear on the pair $(f_0,f_1)$ and transits from $f_0$ on a neighborhood of $\Omega_{s \leq a}$ to $f_1$ on a neighborhood of $\Omega_{s \geq b}$ when $a<b$. If $f_0, f_1, s$ are smooth, then $\Psi[a,b;s](f_0,f_1)$ is also smooth.

When comparing equivalent norms, it is handy to have the following definition.
\begin{definition}
\label{D:sim}
For real numbers (or metric tensors) $a,b >0$ and a real number $c>1$, 
we write $a \sim_c b$ to mean that the inequalities $a \leq cb$ and $b \leq c a$ simultaneously hold.
\end{definition}

In this article we will need the notion of weighted H\"{o}lder norms for functions on a domain $\Omega$ of a 
Riemannian $n$-manifold $(M,g)$ with possibly $\partial \Omega \neq \emptyset$.

\begin{definition}[Weighted H\"{o}lder norms]
\label{D:weighted-Holder}
Assuming $\Omega$ is a domain (possibly with boundary) inside a smooth Riemannian manifold $(M,g)$, 
$k \in \N_0$, $\beta \in [0,1)$, $u \in C^{k,\beta}_{loc}(\Omega)$ or more generally $u$ is a 
$C^{k,\beta}_{loc}$ tensor field (section of a vector bundle) on $\Omega$, $f:\Omega \to (0,\infty)$ is a given 
function, and that the injectivity radius of $(M, g)$ is larger than $1/10$ at each $p\in\Omega$, 
we define 
\[
\| u:C^{k,\beta}(\Omega, g, f)\| := \sup_{p \in \Omega} \frac{ \| \, u  \,:\, 
C^{k,\beta}(\Omega \cap B_p, g) \|}{f(p)}, 
\]
where $B_p$ is a geodesic ball centered at $p$ and of radius $1/100$ in the metric $g$. 
For simplicity we may omit either $\beta$ or $f$ when $\beta=0$ or $f \equiv 1$ respectively. 
We will also omit the metric $g$ if it is clear from the context.
\end{definition}

From the definition, one can easily verify a multiplicative, a scaling, and a monotonicity property as follows: 
\begin{alignat}{1} 
\label{E:norm-multi}
\| u_1 u_2 : C^{k,\beta}(\Omega,g,f_1f_2)\| \leq & \,\, C(k) \, \|u_1 :C^{k,\beta}(\Omega,g,f_1)\| \, \|u_2 :C^{k,\beta}(\Omega, g, f_2)\|,
\\
\label{E:norm-scale}
\|u:C^{k,\beta}(\Omega, \lambda^2 g,f)\| \leq & \,\, \lambda^{-(k+\beta)} \|u:C^{k,\beta}(\Omega,g,f)\| \quad (\forall\lambda \in (0,1)\,),
\\
\label{E:norm-f12}
\|u:C^{k,\beta}(\Omega, g,f_1)\| \leq & \,\, \|u:C^{k,\beta}(\Omega, g,f_2)\| \quad \text{ for } f_2\le f_1. 
\end{alignat} 

\subsection*{Acknowledgments}

The authors would like to thank Richard Schoen for his continuous support and interest in the results of this article. 
M. L. would like to thank the Croucher Foundation for the financial support 
and the Department of Mathematics at Massachusetts Institute of Technology, 
where part of the work in this paper was done. 
M. L. was partially supported by CUHK Direct Grant for Research C001-4053118 and a grant from the Research Grants Council of the Hong Kong SAR, China [Project No.: CUHK 24305115]. 
N. K. was partially supported by NSF grants DMS-1105371 and DMS-1405537.


\section{Deformations of properly immersed hypersurfaces}
\label{S:deformation}

In this section, we study the geometry of properly immersed hypersurfaces in a Riemannian manifold with boundary. 
In particular we describe the deformations of such hypersurfaces and the corresponding changes of the mean curvature and boundary angle. 

\begin{notation} 
\label{n:k-beta}
Throughout this section, $k \in \N$ and $\beta \in [0,1)$.
\end{notation}

\subsection*{Proper immersion and boundary angle}

Let $(M,g)$ be an $(n+1)$-dimensional Riemannian manifold with boundary $\partial M \neq \emptyset$. 
Without loss of generality, we can assume that $M$ is contained in a fixed $(n+1)$-dimensional Riemannian manifold $(\tM,g)$ without boundary 
\footnote{We can assume that $(\tM,g)$ is complete by \cite{Pigola-Veronelli}. 
For the purpose of this article, we just need the case that $M$ is a compact smooth domain of $\mathbb{R}^3$.}. 

\begin{definition}[Proper $C^{k,\beta}$-immersion]
\label{D:proper-immersion}
Let $\Sigma$ be a smooth $n$-dimensional manifold with (possibly empty) boundary $\partial \Sigma$.
A map $X:\Sigma \to M$ is said to be a \emph{proper $C^{k,\beta}$-immersion} if it satisfies both of the following:\\
(i). $X(\Sigma) \subset M$, $X(\partial \Sigma)=X(\Sigma) \cap \partial M$. 
\\
(ii). There exists an extension $X:\tSi \to \tM$ of $X:\Sigma \to M$ such that\\
(a). $\tSi$ is a smooth $n$-dimensional manifold without boundary where $\Sigma \subset \tSi$ and the closure of $\Sigma$ is a compact subset in $\tSi$.\\
(b). $X:\tSi \to \tM$ is a $C^{k,\beta}$-immersion which agrees with $X:\Sigma \to M$ on $\Sigma$.\\
(c). At each $p \in \tSi$ where $X(p) \in \partial M$, we have $X_*(T_p \tSi) + T_{X(p)} \partial M = T_{X(p)} M$.\\
\end{definition}

\begin{remark}
\label{R:metric}
For 
$X:\Sigma \to M$ as in \ref{D:proper-immersion}
we will always equip $\Sigma$ (and $\tSi$) with the induced metric $X^*g$ 
unless stated otherwise. 
Moreover when $X(\tSi)$ is embedded, we will usually take $X$ to be the inclusion map.  
\end{remark}

\begin{remark}
\label{R:extension}
Note that $X: \Sigma' \subset \tSi \to \tM$ is also an extension of $X:\Sigma \to M$ 
for any open subset $\Sigma' \subset \tSi$ whose closure is compact and contained inside $\tSi$. 
Therefore we can assume w.l.o.g. that the extension $X:\tSi\to\tM$ has  further extension satisfying 
(ii).(a)-(c) in \ref{D:proper-immersion}. 
We will assume this from now on. 
\end{remark}

\begin{remark}
Note that the transversality condition in \ref{D:proper-immersion}.ii.c and $X(\Sigma) \subset M$ imply that 
$X(\partial \Sigma)=X(\Sigma) \cap \partial M$ is equivalent to $X(\partial \Sigma) \subset \partial M$ 
since the immersed hypersurface cannot be tangent to $\partial M$ at an interior point of $\Sigma$ (see Figure \ref{fig:proper-immer}).
\end{remark}

\begin{figure}[h]
\includegraphics[height=7cm]{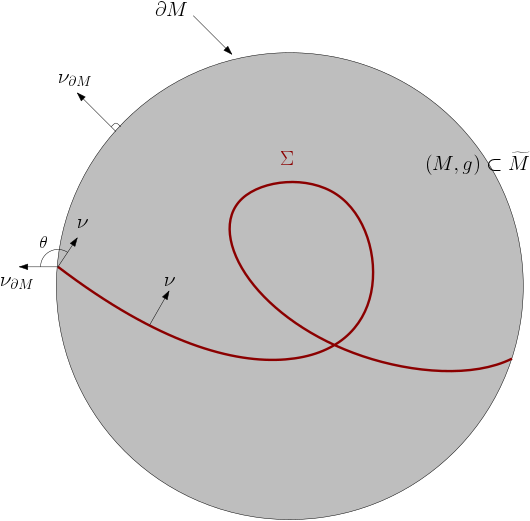}
\caption{A properly immersed hypersurface $\Sigma$ in $M$ with unit normal $\nu$ and boundary angle $\Theta=\cos \theta$}
\label{fig:proper-immer}
\end{figure}

Now we proceed to define the angle at which a proper immersion $X:\Sigma \to M$ makes with $\partial M$ along $\partial \Sigma$. Recall that an immersion $X: \Sigma \to M$ is \emph{2-sided} if there exists a continuous globally defined unit normal $\nu: \Sigma \to TM$ such that $\nu(p) \perp X_*(T_p \Sigma)$ for all $p \in \Sigma$. 

\begin{definition}[Boundary angle]
\label{D:boundary-angle}
Let $X:\Sigma \to M$ be a 2-sided proper $C^{k,\beta}$-immersion (recall \ref{D:proper-immersion}) with a chosen unit 
$C^{k-1,\beta}$ normal $\nu: \Sigma \to TM$. The \emph{boundary angle} $\Theta: \partial \Sigma \to \R$ is defined by
\[ \Theta(p):=g(\nu_{\partial M}(X(p)),\nu(p)),\]
where $\nu_{\partial M}:\partial M \to TM$ is the outward unit normal vector field of $\partial M$ relative to $(M,g)$.
\end{definition}

Note that the above definition makes sense since $X(p) \in \partial M$ for all $p \in \partial \Sigma$. Moreover, $\Theta$ is independent of the extension of $X$. 
The following lemma is clear from the definitions.

\begin{lemma}
\label{L:boundary-angle}
Let $X:\Sigma \to M$ be a 2-sided proper $C^{k,\beta}$-immersion with a chosen unit normal $\nu: \Sigma \to TM$. Then, the boundary angle $\Theta: \partial \Sigma \to \R$ defined in \ref{D:boundary-angle} is a $C^{k-1,\beta}$-function. Moreover, $\Theta \equiv 0$ if and only if $X:\Sigma \to M$ is a properly immersed free boundary hypersurface, i.e. $X(\Sigma)$ meets $\partial M$ orthogonally along $\partial \Sigma$.
\end{lemma}

\subsection*{Perturbations of proper immersions}

Let $X:\Sigma \to M$ be a proper $C^{k+1,\beta}$-immersion. We are interested in its deformation among the class of proper immersions, i.e. a family $X_t:\Sigma \to M$ of proper immersions such that $X_0=X$. 
We need to use a ``twisted exponential map'' to deform a properly immersed hypersurface in $M$ so that it remains properly immersed throughout the deformation. 
The basic idea is to modify the unit normal vector field of $\Sigma$ near its boundary $\partial \Sigma$ 
and extend it to a tubular neighborhood of $\Sigma$ so that the vector field is \emph{tangential to $\partial M$}. 
Then, we make use of this modified unit normal vector field to generate a flow which plays the role of the normal exponential map for hypersurfaces without boundary. 

Let $X:\Sigma \to M$ be a proper $C^{k+1,\beta}$-immersion, oriented by the global unit normal $\nu:\Sigma \to TM$ (which is of class $C^{k,\beta}$). 
Recall that $M$ is a smooth domain of $\tM$, which is a Riemannian manifold without boundary. 
To keep our discussion less technical, we will just focus on the case where the extension $X:\tSi \to \tM$ in \ref{D:proper-immersion} is an \emph{embedding}. 
This does not put any restriction to our applications as we are going to discuss local properties. 
Note that $X:\tSi \to \tM$ is no longer \emph{proper} (unless $\partial \Sigma = \emptyset$). 

\begin{definition}[Tubular neighborhoods]
\label{D:tube-nbd}
Let $S \subset \tM$ be a 2-sided embedded $C^{k+1,\beta}$-hypersurface without boundary 
and we assume we are given a global unit $C^{k,\beta}$ normal $\nu_S:S \to T\tM$. 
Furthermore we assume that there is an $\epsilon>0$ such that the map 
$E:S\times(-\epsilon,\epsilon)\to \tM$ defined by 
$E(p,t):=  \exp^{\tM}_p(t \nu_S(p))$ 
is (well defined and) a diffeomorphism onto an open subset of $\tM$ which we denote by 
$ V_\epsilon(S)$ 
and we call the \emph{tubular neighorbood of $S$} of size $\epsilon$.  
We denote the components of the inverse of $E$ by $\Pi_S$ and $\rho_S$ so that $\forall q\in V_\epsilon(S)$ 
we have $E^{-1}(q)=(\, \Pi_S(q)\, , \, \rho_S(q) \, )$.  
Note that $\Pi_S: V_\epsilon(S) \to S$ is the nearest point projection to $S$ which is $C^{k,\beta}$ 
and $\rho_S: V_\epsilon(S) \to \R$ is a (signed) distance function to $S$ which is $C^{k+1,\beta}$ by \cite{Foote84}. 
We finally extend the given $\nu_S$ to $\nu_S: V_\epsilon(S) \to T\tM$ by 
\[ \nu_S(p):= (\nabla^{\tM} \rho_{S})(p) \qquad \text{for all } p \in V_{\epsilon}(S). \]
\end{definition}

\begin{assumption}
\label{ass:e} 
From now on we will denote by $\epsilon>0$ a number which is small enough so that the assumptions in \ref{D:tube-nbd} hold 
for both $S=\partial M$ and $S=\tSi$ (recall \ref{R:extension} for the second case), 
and moreover $\forall q\in V_\epsilon(\tSi) \cap V_\epsilon(\partial M) $ we have 
$\nu_{\tSi}(q) \ne \pm \nu_{\partial M}(q)$.
\end{assumption}

\begin{definition}
\label{D:twisted-normal}
Using the notations as above and assuming \ref{ass:e} holds, 
define the \emph{twisted normal vector field} (with parameter $\epsilon$) as the map 
$\tnu: V_\epsilon(\tSi) \to T \tM$ given by
\begin{equation}
\label{E:tN}
\tnu:= \frac{\nu^\ep-g(\nu^\ep,\nu^\ep_{\partial M}) \nu^\ep_{\partial M} }{1- g(\nu^\ep,\nu^\ep_{\partial M})^2},
\end{equation}
where 
$\nu^\ep= (\psi_{cut}^\epsilon \circ \rho_{\tSi}) \nu$, $\quad\nu:=\nu_{\tSi}$,    
and $\nu^\ep_{\partial M}= (\psi^\epsilon_{cut} \circ \rho_{\partial M}) \nu_{\partial M}$  
(recall \ref{D:cutoff-epsilon}).  
\end{definition}

Note that $g(\nu^\ep,\nu^\ep) \leq 1$ and $g(\nu^\ep_{\partial M},\nu^\ep_{\partial M}) \leq 1$ everywhere. 
The denominator in \ref{E:tN} does not vanish by \ref{ass:e}. 

\begin{lemma}
\label{L:tN}
Using the notations in \ref{D:twisted-normal}, 
$\tnu$ satisfies the following properties:\\
(i). $\tnu$ is $C^{k,\beta}$ and supported inside $V_{2\epsilon/3}(\tSi)$.\\
(ii). $\tnu(p) \in T_p \partial M$ for all $p \in \partial M \cap V_\epsilon(\tSi)$.\\
(iii). $\tnu=\nu$ on $V_{\epsilon/3} (\tSi) \setminus V_{\epsilon}(\partial M)$.\\
(iv). $g(\tnu,\nu) \equiv 1$ in $V_{\epsilon/3}(\tSi)$.
\end{lemma}

\begin{proof}
Property (i) is clear from the definition and that $\nu^\ep =0$ outside $V_{2\epsilon/3}(\tSi)$ (recall \ref{D:cutoff-epsilon}). For (ii), note that at any $p \in \partial M$, $\tnu$ is parallel to $\nu^\ep-g(\nu^\ep,\nu_{\partial M}) \nu_{\partial M}$, which is the tangential component of $\nu^\ep$ along $T_p \partial M$. Property (iii) is clear since $\nu^\ep_{\partial M}=0$ outside $V_\epsilon(\partial M)$ and $\nu^\ep=\nu$ in $V_{\epsilon/3}(\tSi)$. Finally, (iv) follows from the fact that $\nu^\ep=\nu$ and $g(\nu,\nu) \equiv 1$ in $V_{\epsilon/3}(\tSi)$.
\end{proof}

\begin{definition}
\label{D:twisted-exp}
Let $X : \Sigma \to M$ be a 2-sided proper $C^{k+1,\beta}$-immersion with a choice of the unit normal $\nu:\Sigma \to TM$ and an extension $X:\tSi \to \tM$. 
Suppose we have fixed an $\epsilon>0$ small such that \ref{ass:e} holds and $\tnu$ is well-defined as in \ref{D:twisted-normal}.  
The \emph{twisted normal exponential map along $\Sigma$} is the flow $\{\tF_s\}_{s \in \R}$ 
generated by the twisted normal vector field $\tnu$ in \ref{E:tN}, i.e. for each $p \in \Sigma$, $s \mapsto \tF_s(p)$ is the unique solution to the ODE:
\[ \frac{\partial}{\partial s} \tF_s(p) = \tnu(\tF_s(p)), \qquad \text{for all } s \in (-\epsilon,\epsilon) \]
with initial value $\tF_0(p)=p$.
\end{definition}

Note that for each fixed $s \in (-\epsilon,\epsilon)$, 
the map $p \mapsto \tF_s(p)$ is a $C^{k,\beta}$-map from $\Sigma$ to $V_\epsilon(\tSi)$ by standard results from ODE theory. 
The following definition is the ``twisted'' version of \ref{d:imm} for the case with boundary.

\begin{definition}[Twisted graph]
\label{D:twisted-graph}
Under the same hypothesis of \ref{D:twisted-exp}, 
for any $\varphi \in C^{k,\beta}(\Omega)$ defined on some domain $\Omega \subset \Sigma$ with $|\varphi(p)| < \epsilon$ for all $p \in \Omega$, 
we define the \emph{twisted perturbation of $X:\Sigma \to M$ by $\varphi$ over $\Omega$ (or of $\Omega$ by $\varphi$ when $X$ is an inclusion map)} 
to be the map $\widetilde{\Immer}[X,\epsilon;\varphi,\Omega]: \Omega \to M$ given by
\[ \widetilde{\Immer}[X,\epsilon;\varphi,\Omega](p):=\tF_{\varphi(p)}(p) \qquad \forall p \in \Omega.\]
We also call the image of $\widetilde{\Immer}[X,\epsilon;\varphi,\Omega]$ the \emph{twisted graph of $\varphi$ over $X$ (or $\Omega$)} 
and we will denote it by $\widetilde{\Graph}[X,\epsilon;\varphi,\Omega] \subset M$. 
\end{definition}

\begin{definition}
\label{D:admissible}
Under the same hypothesis of \ref{D:twisted-exp}, a function $\varphi \in C^{k,\beta}(\Sigma)$ is said to be \emph{admissible} if 
(recall \ref{D:twisted-exp} and \ref{D:twisted-graph}) the map $\widetilde{\Immer}[X,\epsilon; \varphi,\Sigma]:\Sigma \to M$ 
is a $C^{k,\beta}$ proper immersion.
\end{definition}

\begin{remark}
Unlike the case of hypersurfaces \emph{without boundary}, our definitions above depends not only on the hypersurface $\Sigma$ but also on the parameter $\epsilon$. 
This creates additional difficulties as we need to give uniform estimates in terms of the parameter $\epsilon$. 
In addition, another subtle issue is that the constructions above in general also depend 
on the extension $X:\tSi \to \tM$ of the proper immersion $X:\Sigma \to M$. 
An important observation however is that the constructions above are independent of the extension 
in case $\partial M$ is \emph{convex} and that $X:\Sigma \to M$ is a \emph{free boundary} properly immersed hypersurface. 
We will return to these issues after we have given a more precise quantitative description in the next subsection.
\end{remark}

\begin{definition}
\label{D:perturbed-Theta}
Given an admissible function $\varphi \in C^{k,\beta}(\Sigma)$ as in \ref{D:admissible}, 
if we let 
\[ \tX_{\varphi}:=\widetilde{\Immer}[X,\epsilon;\varphi,\Sigma]: \Sigma \to M \] 
be the proper $C^{k,\beta}$-immersion obtained by the twisted perturbation of $X:\Sigma \to M$ by $\varphi$ (recall \ref{D:twisted-graph}), 
then we define the \emph{perturbed boundary angle} and \emph{perturbed mean curvature} respectively
\[ \Theta_\varphi: \partial \Sigma \to \R \quad \text{ and } \quad H_\varphi:\Sigma \to \R\]
to be the boundary angle and mean curvature respectively of the proper $C^{k,\beta}$ immersion $\tX_\varphi:\Sigma \to M$, 
which is oriented by the unit normal $\nu_\varphi:\Sigma \to TM$ depending continuously on $\varphi$.
\end{definition}

\begin{remark}
The continuity of the normal means that $\nu_{t \varphi} \to \nu$ as $t \to 0$. 
When $\partial \Sigma = \emptyset$ and that $X(\Sigma) \cap V_\epsilon(\partial M)=\emptyset$, 
the map $\tX_\varphi$ agrees with the usual normal graph $\Immer[X,\varphi;\Sigma]$ as defined in \ref{d:imm}. 
In this case, the boundary angle $\Theta_\varphi$ is not defined and the mean curvature 
$H_\varphi$ agrees with the standard notion as in \cite[Appendix A]{Kapouleas-Yang10}.
\end{remark}

\begin{figure}[h]
\includegraphics[height=6cm]{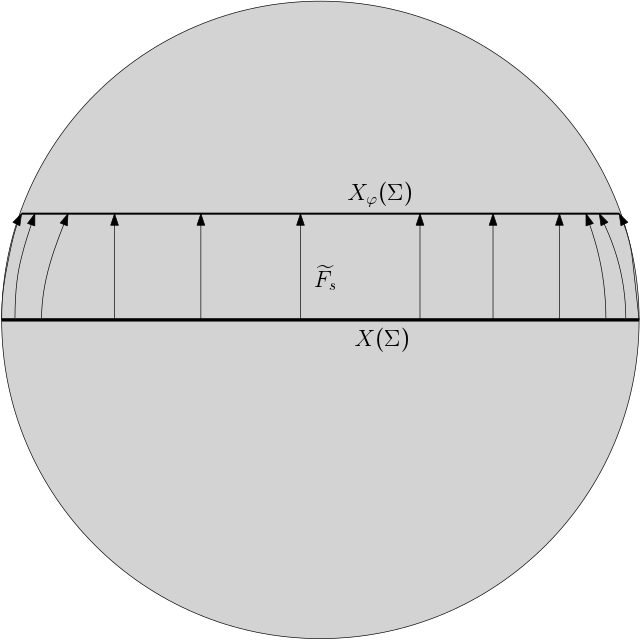}
\caption{A twisted perturbation $X_\varphi$ of a proper immersion $X:\Sigma \to M$}
\label{fig:twisted-graph}
\end{figure}

\subsection*{Local estimates for the boundary angle}

In \ref{P:local-estimate-Theta} we provide a first order expansion of the perturbed boundary angle $\Theta_\varphi$ (recall \ref{D:perturbed-Theta}) in terms of $\varphi$ 
(which we assume sufficiently small in terms of the geometry of $X$, $M$ and $\tM$) 
and prove uniform estimates for the nonlinear terms. 
As in other gluing constructions, these estimates are crucial for the fixed point theorem argument 
used to produce exact solutions to the nonlinear PDEs (see \cite{Kapouleas05} and \cite{Kapouleas11} for a general discussion). 
Note also that instead of estimating the nonlinear terms in terms of invariant geometric quantities 
as for example in \cite{Kapouleas90a}, 
it is easier to use local coordinates as for example in \cite[Appendix A]{Kapouleas-Yang10}: 
Given a proper $C^{k,\beta}$ immersion $X:\Sigma \to M$ as in \ref{D:proper-immersion} we express locally the immersion 
$X:\Sigma \to M$ as $X = (X^1,\ldots,X^{n+1})$ in local coordinate charts of $\Sigma$ and $M$ 
(see Figure \ref{fig:C-bounded}).
To make a quantitative statement we need bounds on the geometry as follows.

\begin{definition}[$c_1$-bounded geometry]
\label{D:C-bdd-geom}
Let $\tM=(B^{n+1},g)$ where $g$ is a smooth Riemannian metric with components $g_{KL}$ in standard coordinates of $B^{n+1} \subset \R^{n+1}$. 
Suppose $M \subset \tM$ is a smooth domain with a smooth ``boundary defining function'' $\rho:B^{n+1} \to \R$ such that $M=\rho^{-1}(-\infty,0]$. 
We say that the pair $(M,\tM)$ has \emph{$c_1$-bounded geometry} if
\begin{equation}
\label{E:M-C-bounded}
 \|g_{KL}, g^{KL}:C^{4,\beta}(B^{n+1},g_0) \| \leq c_1 \quad \text{and} \quad \|\rho:C^{5,\beta}(B^{n+1},g_0)\| \leq c_1, 
\end{equation}
where $g^{KL}$ is the inverse of the matrix $g_{KL}$ and $g_0$ is the Euclidean metric on $B^{n+1}$.
 
Suppose $X:\Sigma \to M$ is a proper $C^{5,\beta}$ immersion as in \ref{D:proper-immersion}, where $\Sigma=B^n$ or $B^n_+$ (recall \ref{E:BS+}), 
with an extension $X: B^n(2) \to \tM$ with $\tSi=B^n(3/2)$ (recall \ref{R:extension}). 
We say that $X:\tSi \to \tM$ has \emph{$c_1$-bounded geometry} if $(M,\tM)$ has $c_1$-bounded geometry and the following holds:
\begin{equation}
\label{E:C-bounded-a}
\|\partial X : C^{4,\beta}(B^n(2),g_0)\| \leq c_1, \quad g_0 \leq c_1 \, X^*g \quad \text{and} \quad X(B^n(2)) \subset B^{n+1}(3/4),
\end{equation}
where $\partial X$ are the partial derivatives of the coordinate functions of $X:B^n(2) \to \R^{n+1}$, here $g_0$ is the Euclidean metric; and
\begin{equation}
\label{E:C-bounded-b}
\inf \left\{ |\nabla (\rho \circ X)|(x): x \in B^n(2), |\rho(X(x))| < c_1^{-1} \right\} \geq c_1^{-1},
\end{equation}
where $\nabla (\rho \circ X)$ is the Euclidean gradient of the function $\rho \circ X:\tSi \to \R$.
\end{definition}

\begin{figure}
\includegraphics[height=8cm]{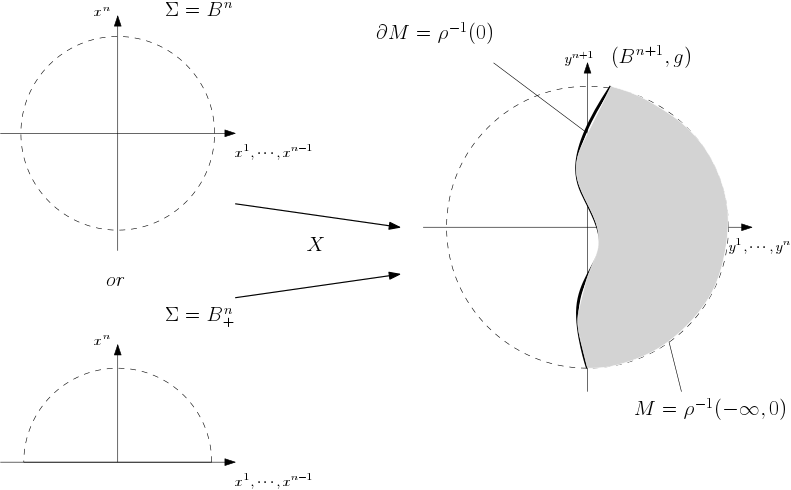}
\caption{A local description of a proper $C^{k,\beta}$ immersion $X:\Sigma \to M$}
\label{fig:C-bounded}
\end{figure}

Note that \ref{E:C-bounded-b} gives a quantitative measure of transversality 
by ensuring that the part of $X(\tSi)$ which is close to $\partial M$ cannot be approximately parallel to $\partial M$. 
Note that \ref{E:M-C-bounded} and \ref{E:C-bounded-a} (but not \ref{E:C-bounded-b})
can be arranged by appropriately magnifying the target (see \ref{P:nonlinear-estimate} for example). Definition \ref{D:C-bdd-geom} also covers the case where the coordinate neighborhood of $M$ under consideration lies completely in the interior of $M$ (in this case we simply take $\rho \equiv -1$ and \ref{E:C-bounded-b} would be trivially satisfied for any $c_1>1$).

In order to define our linear operators in 
\ref{D:linear-operators} we first define appropriately the second fundamental forms 
of an immersed hypersurface $S$ (with or without boundary) in $\tM$. 
For simplicity, we state the definition for an \emph{embedded} hypersurface $S \subset \tM$ but the case of immersion can be defined similarly since the definition is local.

\begin{definition}
\label{D:second-fundamental}
Let $S$ be a 2-sided embedded $C^{k,\beta}$ hypersurface in $(\tM,g)$ with a choice of the unit normal $\nu_S:S \to TM$.  
We define the \emph{second fundamental form of $S$ in $\tM$ at $p \in S$} as the symmetric bilinear form $A_S:T_p \tM \times T_p \tM \to \R$ defined by
\beq
\label{E:2ff}
A_S(u,v) := g(\nabla^{\tM}_{\widetilde{u}} \nu_S, \widetilde{v}  )(p)
\eeq
where $\widetilde{u}$,$\widetilde{v}$ are local vector fields 
in the vicinity of $p$ which are tangential to $S$ 
and agree at $p$ with the orthogonal projection (with respect to $g$)
of $u,v$ respectively to $T_pS$.
\end{definition}

We define now some linear operators which as we will see later 
in our main propositions \ref{P:local-estimate-Theta} and \ref{P:local-estimate-H} 
are the linearizations of the boundary angle 
and mean curvature operators (which are both nonlinear). Note that both operators are independent of the choice of an extension $X:\tSi \to \tM$.

\begin{definition}
\label{D:linear-operators}
Let $X:\Sigma \to M$ be a 2-sided proper $C^{k+1,\beta}$ immersion as in \ref{D:proper-immersion} with a choice of the unit normal $\nu:\Sigma \to TM$. 
We define the linear operators $\Lcal:C^{2,\beta}(\Sigma) \to C^{0,\beta}(\Sigma)$ 
and $\Bcal:C^{2,\beta}(\Sigma) \to C^{1,\beta}(\partial \Sigma)$ by 
(recall \ref{D:boundary-angle} and \ref{D:second-fundamental}) 
\[ \Lcal \varphi:=  \Delta_g \varphi +\Big(|A_\Sigma |_g^2  +  \Ric (\nu,\nu)\Big) \varphi, \] 
\[ \Bcal \varphi := - \sqrt{1-\Theta^2}  \frac{\partial \varphi}{\partial \eta} + \frac{1}{1-\Theta^2} 
\Big(A_{\partial M}(\nu,\nu)-\Theta A_\Sigma(\nu_{\partial M}, \nu_{\partial M}) \Big)\varphi ,\]
where $\Delta_g$ is the intrinsic Laplacian on $\Sigma$; $\eta: \partial \Sigma \to TM$ is the outward unit conormal of $\partial \Sigma$ relative to $\Sigma$, $|A_\Sigma|^2_g$ is its length squared of the second fundamental form of $\Sigma$,  
and $\Ric$ is the Ricci curvature of $(M,g)$. Notice that $\Theta^2 \neq 1$ everywhere on $\partial \Sigma$ by \ref{D:proper-immersion}.ii.c.
\end{definition}

Before we state the main proposition in this section 
we observe that the H\"{o}lder norms are uniformly equivalent with respect to different metrics. 

\begin{lemma}[Equivalence of norms on {$\tSi$}] 
\label{L:C-bounded-equiv}
There exists a constant $C=C(c_1)>0$ such that if $X:\tSi \to \tM$ has $c_1$-bounded geometry as in \ref{D:C-bdd-geom}, 
then we have $\|g_{ij}, g^{ij} :C^{4,\beta}(\tSi,g_0) \| \leq C$, 
where $g_{ij}$ are the components of the induced metric $X^*g$ in the local coordinates $x^1,\cdots,x^n$. 
Moreover, we have a uniform equivalence on the H\"older norms with respect to the induced metric $g$ and the Euclidean metric $g_0$ (recall \ref{D:sim}): 
\beq
\label{E:Holder-equiv}
\| f :C^{k,\beta}(\Omega',g)\| \sim_C \|f:C^{k,\beta}(\Omega',g_0)\| \qquad \text{ for } k=0,1,2, 3,
\eeq
where $(\Omega',g)$ stands either for $(B^{n+1},g)$ or $(B^n,X^*g)$. 
\end{lemma}

\begin{proof} 
This is an easy consequence of \ref{E:M-C-bounded} and \ref{E:C-bounded-a}. 
\end{proof}

Because of \ref{L:C-bounded-equiv}, from now on we will often omit the dependence of our H\"{o}lder norms on the metric 
as they are equivalent up to a uniform constant depending only on $c_1$. 
We collect some bounds on the geometry implied by \ref{E:M-C-bounded}, \ref{E:C-bounded-a} and \ref{E:C-bounded-b}. 
In the lemma below we will use $| \cdot |$ to denote the norm of a vector with respect to the Euclidean metric $g_0$.

\begin{lemma}
\label{L:C-bounded}
There exists a constant $C=C(c_1)>0$ such that if $X:\tSi \to \tM$ has $c_1$-bounded geometry as in \ref{D:C-bdd-geom}, then we have the following:\\
(i). $\inj_y(B^{n+1},g) \geq C^{-1}$ at each $y \in B^{n+1}(3/4)$. The exponential map $\exp (y,v)=\exp_y(v)$ for $(B^{n+1},g)$ is a $C^3$-map on $(y,v) \in B^{n+1}(3/4) \times B^{n+1}(C^{-1})$ and satisfies \ref{E:A3}. \\
(ii). $\|\nu: C^{4,\beta}(\tSi)\| \leq C$ and $\| A_{\tSi}: C^{3,\beta}(\tSi)\| \leq C$ (recall \ref{E:2ff}).\\
(iii). $\|\nu_{\partial M}: C^{4,\beta}(\partial M)\| \leq C$ and $\| A_{\partial M}: C^{3,\beta}(\partial M)\| \leq C$ (recall \ref{E:2ff}).\\
(iv). $\| \Ric_{KL} : C^{2,\beta}(B^{n+1})\| \leq C$, where $\Ric_{KL}$ are the components of the Ricci curvature of $(B^{n+1},g)$ in the local coordinates $y^1,\cdots,y^{n+1}$.\\
(v). $\| \Theta : C^{4,\beta}(\partial \Sigma)\| \leq C$, where $\Theta$ is the boundary angle of the proper immersion $X:\Sigma \to M$ when $\Sigma=B^n_+$. 
\end{lemma}

\begin{proof}
All these statements follow easily from \ref{E:M-C-bounded}, \ref{E:C-bounded-a}, 
\ref{D:boundary-angle}, \ref{L:C-bounded-equiv} and Appendix \ref{A:exp}. 
\end{proof}

The lemma below says that the parameter $\epsilon>0$ used to construct deformations of properly immersed hypersurfaces can be uniformly controlled (depending only on $c_1$) for $X:\tSi \to \tM$ with $c_1$-bounded geometry.

\begin{lemma}
\label{L:C-bounded-eps}
There exists $\epsilon=\epsilon(c_1)>0$ such that if $X:\tSi \to \tM$ has $c_1$-bounded geometry as in \ref{D:C-bdd-geom}, then \ref{ass:e} holds.
\end{lemma}

\begin{proof}
By \ref{L:C-bounded}.i, ii and iii (recall \ref{R:extension}), such an $\epsilon>0$ (depending only on $c_1$) exists for $S=\tSi$ and $\partial M \cap B^{n+1}(3/4)$ satisfying all the hypotheses in \ref{D:tube-nbd}. Furthermore, by \ref{E:C-bounded-b}, there exists an $\epsilon>0$ (depending only on $c_1$) such that $\nu_{\tSi}(q) \neq \pm \nu_{\partial M}(q)$ at every $q \in V_\epsilon(\tSi) \cap V_{\epsilon}(\partial M \cap B^{n+1}(3/4))$.
\end{proof}

We can now state the main proposition in this section which gives a local uniform estimate on the nonlinear terms 
in the expansion of the perturbed boundary angle $\Theta_\varphi$ (recall \ref{D:perturbed-Theta}) in terms of $\varphi$.  
Note that in the proof we use the coordinates of $\R^{n+1}\supset \tM$ to interpret $T\tM$-valued vector fields or maps to $M$ 
as $\R^{n+1}$-valued maps, as for example in \ref{E:nuV} and \ref{E:tX-bound}:

\begin{prop}[Linear and nonlinear parts of the boundary angle]
\label{P:local-estimate-Theta}
There exists a small constant $\epsilon_\Theta(c_1)>0$ such that if $X:\tSi \to \tM$ has $c_1$-bounded geometry as in \ref{D:C-bdd-geom}  with $\Sigma=B^n_+$, and $\varphi:\Sigma \to \R$ is a $C^{2,\beta}$ function satisfying
\[ \|\varphi:C^{2,\beta}(\Sigma,g)\| < \epsilon_\Theta(c_1),\]
then $\varphi$ is admissible (recall \ref{D:admissible}) 
and we have the uniform estimate (recall \ref{D:perturbed-Theta} and \ref{D:linear-operators}) for some constant $C=C(c_1)>0$,
\beq
\label{E:Theta-estimate}
 \|\Theta_\varphi - \Theta - \Bcal \varphi : C^{1,\beta}(\partial \Sigma,g)\| \leq C \|\varphi:C^{2,\beta}(\Sigma,g)\|^2,
\eeq
where $g$ can be either the induced metric $X^*g$ or $g_0$ in accordance with 
\ref{L:C-bounded-equiv}. 
\end{prop}

\begin{proof}
Recall $\Sigma=B^n_+$. 
Let $X:\Sigma \to M$ be a proper $C^{5,\beta}$-immersion with an extension $X:\tSi \to \B^{n+1}$ 
and a choice of the unit normal $\nu:\tSi \to \R^{n+1}$ such that $X:\tSi \to \tM$ has $c_1$-bounded geometry as in \ref{D:C-bdd-geom}. 
We fix once and for all an $\epsilon>0$ given by \ref{L:C-bounded-eps}, which depends only on $c_1$. 
From now on we will use $C>0$ to denote any constant depending only on $c_1$.

Suppose $\varphi:B^n_+ \to \R$ is a $C^{2,\beta}$-function satisfying $\|\varphi:C^{2,\beta}(B^n_+,g)\| < \epsilon_\Theta$. 
First, we show that $\varphi$ is admissible in the sense of \ref{D:admissible} when $\epsilon_\Theta$ is sufficiently small (depending only on $c_1$). 
In other words, we have to prove that $\widetilde{\Immer}[X,\epsilon; \varphi,\Sigma]:\Sigma \to M$ is a proper $C^{2,\beta}$ immersion (recall \ref{D:twisted-graph}).
First of all, we prove the following uniform bound on the twisted normal vector $\tnu$ defined in \ref{D:twisted-normal} 
\begin{equation}
\label{E:tN-bound}
\| \tnu:C^{4,\beta}(V_{\epsilon/3}(\tSi),g)\| \leq C.
\end{equation}
To establish \ref{E:tN-bound}, first observe that $\nu^\ep=\nu$ inside $V_{\epsilon/3}(\tSi)$ and thus
\[ \tnu= \frac{\nu-g(\nu,\nu^\ep_{\partial M}) \nu^\ep_{\partial M}}{1- g(\nu,\nu^\ep_{\partial M})^2} \qquad \text{in $V_{\epsilon/3}(\tSi)$}. \]
By \ref{E:C-bounded-b}, we have the uniform estimate $1-g(\nu,\nu^\ep_{\partial M})^2 \geq C^{-1}$ on $V_{\epsilon/3}(\tSi)$ for $\epsilon$ sufficiently small depending only on $c_1$. Moreover, by \ref{L:C-bounded}.ii and iii and that $\|\psi_{cut}^{\epsilon}\|_{C^k} \leq C(k)/\epsilon^k$ for all $k \in \N$, we have 
\begin{equation}
\label{E:nuV} 
\|\nu: C^{4,\beta}(V_{\epsilon/3}(\tSi),g)\| \leq C \quad \text{and} \quad \|\nu^\ep_{\partial M}: C^{4,\beta}(V_{\epsilon/3}(\tSi),g)\| \leq C.  
\end{equation} 
All of these estimates together yield \ref{E:tN-bound}.

Now, denote $\tX_\varphi=\widetilde{\Immer}[X,\epsilon; \varphi,\Sigma]$ as in \ref{D:perturbed-Theta}. 
By \ref{L:tN}.iv, the twisted normal exponential map $\{\tF_s\}$ generated by $\tnu$ (recall \ref{D:twisted-exp}) satisfies $\tF_t(p) \in V_{\epsilon/3}(\tSi)$ for all $t \in (-\epsilon/3,\epsilon/3)$, $p \in \Sigma$. Using this and \ref{E:tN-bound}, we have for all $t \in (-\epsilon/3,\epsilon/3)$,
\begin{equation*}
\| \tF_t - X - t \tnu : C^{2,\beta}(\Sigma,g)\| \leq C t^2.
\end{equation*}
Using \ref{E:norm-multi}, it is easy to see that the estimate above implies
\begin{equation}
\label{E:tX-bound}
\| \tX_\varphi - X - \varphi \tnu: C^{2,\beta}(\Sigma,g)\| \leq C \|\varphi:C^{2,\beta}(\Sigma,g)\|^2.
\end{equation}
By \ref{E:C-bounded-a} and \ref{E:tN-bound}, we can conclude from \ref{E:tX-bound} that when $\epsilon_\Theta$ is sufficiently small (but depending only on $c_1$), $\tX_\varphi:\Sigma \to M$ is a $C^{2,\beta}$-immersion (recall \ref{D:twisted-graph}).
Hence, we have proved that $\varphi$ is admissible when $\epsilon_\Theta$ is sufficiently small (depending only on $c_1$).

It remains to prove the uniform estimate \ref{E:Theta-estimate}. First of all, from definitions \ref{D:perturbed-Theta} and \ref{D:linear-operators}, the function $\Theta_\varphi - \Theta - \Bcal \varphi: \partial \Sigma \to \R$ depends only on the values of $\varphi$ in an arbitrarily small tubular neighborhood of $\partial \Sigma$ in $\Sigma$. Note that we have 
\[ \tnu= \frac{\nu-g(\nu,\nu_{\partial M}) \nu_{\partial M}}{1-g(\nu,\nu_{\partial M})^2} \qquad \text{ in $V_{\epsilon/3}(\tSi) \cap V_{\epsilon/3}(\partial M)$}.\]
In particular, we have
\begin{equation}
\label{E:tN-along-boundary}
\tnu=\frac{\nu-\Theta \nu_{\partial M}}{1-\Theta^2} \qquad \text{along $\partial \Sigma$},
\end{equation}
where $\Theta:\partial \Sigma \to \R$ is the boundary angle (recall \ref{D:boundary-angle}) for the proper immersion $X:\Sigma \to M$. 
Let $\nu_\varphi:\Sigma \to TM$ be the unit normal of the proper immersion $\tX_\varphi: \Sigma \to M$ (recall \ref{D:perturbed-Theta}). Then, the perturbed boundary angle $\Theta_\varphi: \partial \Sigma \to \R$ is given by
\begin{equation}
\label{E:perturbed-Theta}
\Theta_\varphi(p)=\left. g\Big(\nu_\varphi(p)\, ,\, \nu_{\partial M}\right|_{\tX_\varphi(p)} \Big).
\end{equation}

From \ref{E:tX-bound} and \ref{E:tN-bound},  we have the following estimate on the unit normals:
\begin{equation}
\label{E:N-bound}
\| \nu_\varphi - \nu - \varphi \nabla_{\tnu_{||}} \nu + X_*(\nabla \varphi): C^{1,\beta}(\Sigma,g)\| \leq C \|\varphi:C^{2,\beta}(\Sigma,g)\|^2,
\end{equation}
where $\tnu_{||}=X^*(\tnu- \nu)$ is the tangential (to $\Sigma$) component of $\tnu$, $\nabla$ is the pullback connection by $X:\Sigma \to \tM$.
On the other hand, by \ref{E:tX-bound}, \ref{L:C-bounded}.iii and \ref{L:tN}.ii, we have
\begin{equation}
\label{E:nu-bound}
\| \nu_{\partial M} \circ \tX_\varphi - \nu_{\partial M} \circ X - \varphi (\nabla_{\tnu} \nu_{\partial M}) \circ X:C^{2,\beta}(\partial \Sigma,g)\| \leq C \|\varphi:C^{2,\beta}(\Sigma,g)\|^2.
\end{equation}
Finally, the estimate \ref{E:Theta-estimate} follows directly from \ref{E:perturbed-Theta}, \ref{E:N-bound}, \ref{E:nu-bound}, \ref{E:tN-along-boundary}, \ref{D:boundary-angle}, \ref{D:second-fundamental} and \ref{D:linear-operators}.
\end{proof}

\subsection*{Local estimates for mean curvature}

At the end of this section, we prove a uniform local estimate on the perturbed mean curvature $H_\varphi$ defined in \ref{D:perturbed-Theta}. 
Recall that the mean curvature is defined to be the sum of the principal curvatures.
Note that all the norms on $\Sigma$ in the next proposition can be taken with respect to either $g$ or $g_0$, according to \ref{L:C-bounded-equiv}.

\begin{prop}[Linear and nonlinear parts of the mean curvature]
\label{P:local-estimate-H}
There exists a small constant $\epsilon_H=\epsilon_H(c_1)>0$ such that 
if $X:\tSi \to \tM$ has $c_1$-bounded geometry as in \ref{D:C-bdd-geom}  with $\Sigma=B^n_+$ or $B^n$, and $\varphi:\Sigma \to \R$ is a $C^{2,\beta}$ function satisfying
\begin{equation}
\label{E:eH} 
\|\varphi:C^{2,\beta}(\Sigma,g)\| < \epsilon_H(c_1),
\end{equation}
then $\varphi$ is admissible (recall \ref{D:admissible}) and we have the uniform estimate for some constant $C=C(c_1)>0$ (recall \ref{D:linear-operators} and \ref{D:perturbed-Theta}),
\begin{equation}
\label{E:H-varphi-bound}
\| H_\varphi -H -\Lcal \varphi - \tnu_{||}(H) \varphi :C^{0,\beta}(\Sigma,g)\| \leq C \| \varphi:C^{2,\beta}(\Sigma,g)\|^2,
\end{equation}
where $\tnu_{||}(H)$ is the directional derivative of $H$ along the tangent vector $\tnu_{||}:=X^*(\tnu-\nu)$.

Moreover, 
if $X':\tSi \to \tM$ is the extension of another proper immersion $X':\Sigma \to M$ such that $X':\tSi \to \tM$ has $c_1$-bounded geometry 
and that $X$ and $X'$ agree up to first order at $0$ 
(that is $X(0)=X'(0)$ and $\partial X(0)=\partial X'(0)$), and the same parameter $\epsilon$ as in \ref{L:C-bounded-eps} is chosen for both pairs,  
then we have the following estimates 
\begin{alignat}{1} 
\label{E:metric-est}
\| \partial \tX'_\varphi - \partial \tX_\varphi :C^{1,\beta}(\Sigma)\| \, \leq 
&\, C \, \| \partial^2 X' - \partial^2 X:C^{1,\beta}(\Sigma)\|, 
\\
\label{E:normal-est}
\| \nu'_\varphi - \nu_\varphi :C^{1,\beta}(\Sigma)\| \, \leq &\, C \, \| \partial^2 X' - \partial^2 X:C^{1,\beta}(\Sigma)\|, 
\\
\label{E:local-estimate-H}
\| Q_{X',\varphi} - Q_{X,\varphi}:C^{0,\beta}(\Sigma) \| \, \leq &\, C \, \|\partial^2 X'-\partial^2 X :C^{1,\beta} (\Sigma)\| \, \|\varphi:C^{2,\beta}(\Sigma)\|^2, 
\end{alignat} 
where $\tX_\varphi=\widetilde{\Immer}[X,\epsilon;\varphi,\Sigma]:\Sigma \to M$ is the immersion defined in \ref{D:twisted-graph} with unit normal $\nu_\varphi:\Sigma \to TM$ and $Q_{X,\varphi}:= H_\varphi -H -\Lcal \varphi- \tnu_{||}(H) \varphi$ is the error term in \ref{E:H-varphi-bound}, and similarly for $X'$.
\end{prop}

\begin{proof}
For simplicity we will just present the proof for $\tM=\R^{n+1}$ with the Euclidean metric $g_0$. 
Note that \ref{E:tX-bound} and \ref{E:N-bound} holds for both cases $\Sigma=B^n_+$ and $\Sigma=B^n$, from which the admissibility of $\varphi$ follows. 
The perturbed mean curvature (recall \ref{D:perturbed-Theta}) can be expressed (with $g=g_0$ denoted by $\langle \cdot, \cdot \rangle$) by the formula
\begin{equation}
\label{E:H-local-coord}
H_\varphi=\sum_{i,j=1}^n g^{ij}_\varphi \langle \nu_\varphi, \partial_i \partial_j \tX_\varphi \rangle,
\end{equation}
where $g^{ij}_\varphi$ is the inverse of the induced metric from $\tX_\varphi:\Sigma \to B^{n+1}$ which satisfies the estimate:
\begin{equation}
\label{E:g-bound}
\| g^{ij}_\varphi - g^{ij} +2 g^{ik} g^{j\ell} \langle \partial_\ell X,\tnu \rangle \partial_k \varphi +2 g^{ik} g^{j\ell} \langle \partial_k X, \partial_\ell \tnu \rangle \varphi: C^{2,\beta}(\Sigma)\| \leq C \|\varphi:C^{2,\beta}(\Sigma)\|^2.
\end{equation}
Using the estimates \ref{E:tX-bound}, \ref{E:N-bound} and \ref{E:g-bound} in \ref{E:H-local-coord}, we obtain the uniform estimate \ref{E:H-varphi-bound}. Note that we have the extra zeroth order term $\tnu_{||}(H) \varphi$ in the linearization because $\tnu$ is not normal to $\Sigma$ (see for example \cite[Lemma B.2]{Kapouleas97}). 
For the second part, under the assumption that $X$ and $X'$ agree up to first order at $0$, we have the following simple estimates:
\beq
\label{E:local-estimate-g}
\begin{aligned}
\|\partial X'-\partial X : C^{2,\beta}(\Sigma)\| \, \leq &\, C \, \| \partial^2 X' - \partial^2 X : C^{1,\beta}(\Sigma)\|, 
\\ 
\| g' - g :C^{2,\beta}(\Sigma)\| \, \leq &\, C \, \|\partial^2 X'-\partial^2 X :C^{1,\beta} (\Sigma)\|, 
\\ 
\|\nu'-\nu:C^{2,\beta}(\Sigma)\| \leq &\, C \, \| \partial^2 X' - \partial^2 X : C^{1,\beta}(\Sigma)\|,
\end{aligned} 
\eeq
where $g$ and $g'$ are the induced metric on $\Sigma$ by the immersions $X:\Sigma \to M$ and $X':\Sigma \to M$ respectively, 
whose unit normals are given by $\nu$ and $\nu'$. 
The estimates \ref{E:local-estimate-g} and \ref{E:tN-bound} then imply \ref{E:metric-est}, from which \ref{E:normal-est} follows. 
Finally, \ref{E:local-estimate-H} follows from the expression of mean curvature \ref{E:H-local-coord} 
together with the estimates \ref{E:metric-est} and \ref{E:normal-est}, 
and \ref{E:norm-multi}. 
\end{proof} 

\begin{remark}
\label{R:H-linearized}
Note that there are two special cases of \ref{P:local-estimate-H} that are particularly interesting. 
If $X:\Sigma \to M$ is a minimal immersion (i.e. $H \equiv 0$), then the linearized operator is the same as the standard Jacobi operator. 
The same happens if $\tnu=\nu$ everywhere on $\Sigma$ (for example, if $X(\Sigma) \cap V_\epsilon(\partial M) = \emptyset$). 
In this article, we will have either one of the scenarios so the linearized problem reduced to the standard one. 
Note that in case $\tnu=\nu$ everywhere on $\Sigma$, we have $\widetilde{\Immer}[X,\epsilon;\varphi,\Sigma]=\Immer[X,\varphi;\Sigma]$ (recall \ref{d:imm}).
\end{remark}


\section{Rotationally symmetric free boundary minimal surfaces}
\label{S:rotation-free}

In this section we study free boundary minimal surfaces in $\B^3$ with rotational symmetry. 
For convenience and without loss of generality (see \ref{C:unique-catenoid}) 
we will assume that the axis of symmetry is the $z$-axis: 

\begin{definition}
\label{D:G'}
We define $\group_\infty$ to be the subgroup of isometries of $\B^3$ generated by 
$O(2)$ acting as usual on the $xy$-plane and trivially on the $z$-axis 
and by the reflection about the $xy$-plane defined by
$ (x,y,z) \mapsto (x,y,-z)$. 
\end{definition}

\begin{figure}
\includegraphics[height=6cm]{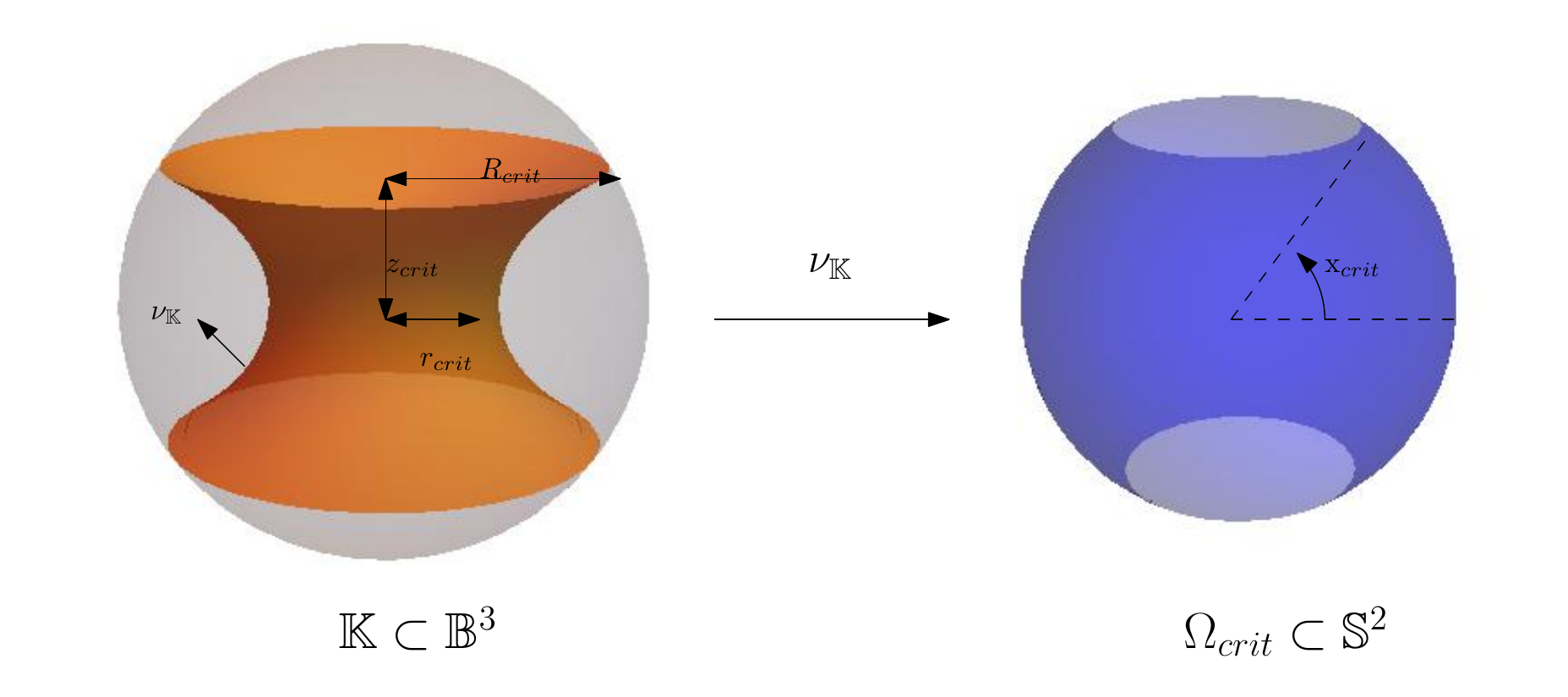}
\caption{The critical catenoid $\K$ and its Gauss map $\nu_\K$}
\label{fig:K}
\end{figure}

In \cite{Fraser-Schoen11}, Fraser-Schoen discovered a rotationally invariant example 
of a free boundary minimal surface in $\B^3$ other than the equatorial disk, which they called the critical catenoid: 

\begin{lemma}[\cite{Fraser-Schoen11}]
\label{L:K-property}
There is a compact embedded free boundary minimal surface in $\B^3$ called 
the critical catenoid, denoted by $\K$ 
(see Figure \ref{fig:K}) 
which satisfies the following: \\ 
(i). $\K$ meets $\Sph^2$ orthogonally along two circles of radius $R_{crit}$ 
lying on the planes $\{z=\pm z_{crit}\}$ where 
$R_{crit} \approx 0.834$ is the unique positive solution to the equation $R^{-1}_{crit}= \coth R_{crit}^{-1}$ 
and $z_{crit}:=\sqrt{1-R_{crit}^2} \approx 0.552$.\\
(ii) $\K$ meets the unit disc $\D:=\{(x,y,z)\in \B^3:z=0\}$ orthogonally 
along a circle of radius $r_{crit}=z_{crit} R_{crit} \approx 0.460$.\\
(iii). $\K$ is invariant under $\group_\infty$ and is the portion inside $\B^3$ of the catenoid 
obtained by rotating the graph of $x=r_{crit}\cosh(z/r_{crit})$. 
\end{lemma}

\begin{proof}
By defining $\K$ as in (iii) for any $r_{crit}\in(0,1)$ we obtain clearly 
a surface with the desired properties except that we have to arrange 
that (i) and (ii) are satisfied. 
This amounts to satisfying the equations 
$$
\frac{R_{crit}}{r_{crit}}=\cosh\frac{z_{crit}}{r_{crit}},
\qquad
\frac{R_{crit}}{z_{crit}}=\sinh\frac{z_{crit}}{r_{crit}},
\qquad
R^2_{crit}+z^2_{crit}=1, 
$$
where the first equation follows from the equation in (iii), 
the second equation amounts to the orthogonality to $\Sph^2$ at the boundary and is obtained 
by differentiating the equation in (iii), 
and the third equation ensures that the circles are the intersections of the catenoid with $\Sph^2$. 
By dividing the third equation by $z^2_{crit}$ 
and using the second we obtain $z^{-2}_{crit}=\cosh^2(z_{crit}/r_{crit})$. 
By using the first equation we conclude then that 
$r_{crit}=z_{crit} R_{crit}$. 
We complete then the proof by dividing the first equation by the second. 
\end{proof}

We adopt now some notation from \cite{Kapouleas14} which we will find useful: 
We will use the spherical coordinates $(\xx,\yy)$ on $\Sph^2 \setminus \{(0,0,\pm 1)\}$ defined by
\beq
\label{E:spherical-coord}
\ThetaSph(\xx,\yy) :=(\cos \xx \cos \yy, \cos \xx \sin \yy, \sin \xx), \qquad \xx \in \left(-\frac{\pi}{2},\frac{\pi}{2} \right), \, \yy \in \R. 
\eeq
Note that $\xx$ and $\yy$ are the geometric latitude and longitude on $\Sph^2$. 
The equator $\Sph^1_{eq}\subset\Sph^2$ is thus identified with $\{\xx =0\}$. 
Note that we orient $\Sph^2$ by the outward unit normal so the map $\ThetaSph$ defined above is orientation-reversing.
We also have 

\begin{definition}[{\cite[Definition 2.18 and lemma 2.19]{Kapouleas14}}] 
\label{Dphieo}
We define smooth rotationally invariant functions
$\phio$ on $\Sph^2$ and  
$\phie$ on $\Sph^2$ punctured at the poles 
by
$$
\phio=\sin\xx,
\qquad
\phie=
1- \sin\xx \, \log\frac{1+\sin\xx}{\cos\xx} 
=
1+ \sin\xx \, \log\frac{1-\sin\xx}{\cos\xx} .
$$
Moreover $\phie$ as a function of $\xx$ has a unique root on $(0,\pi/2)$ 
which in this article we will denote by $\xx_{crit}$ 
(in \cite{Kapouleas14} it was denoted by $\xx_{root}$). 
\end{definition}

\begin{lemma}
\label{L:K-Gauss}
The Gauss map $\nu_\K:\K \to \Sph^2$ 
chosen to point away from the $z$-axis 
is an anti-conformal diffeomorphism onto the spherical domain 
$\Omega_{crit}:= \{ \ThetaSph(\xx,\yy): \xx \in [-\xx_{crit},\xx_{crit}], \yy \in \R\} \subset \Sph^2$. 
Moreover we have \\ 
(i). $\nu_\K\cdot e_z=\phio\circ\nu_\K$
and
$\nu_\K\cdot X_\K=\phie\circ\nu_\K$ 
where $X_\K:\K\to\B^3\subset\R^3$ is the inclusion map. 
Therefore,  
$\phio\circ\nu_\K$
and 
$\phie\circ\nu_\K$ 
are Jacobi fields induced by the translation along the $z$-axis 
and by scaling respectively. \\ 
(ii). 
$R_{crit}=\sin \xx_{crit}$,  
$z_{crit}=\cos \xx_{crit}$, 
$r_{crit}=\frac12 \sin (2\xx_{crit})$,  
and 
$\xx_{crit} \approx 0.986 > \pi/4 $.  
\end{lemma}

\begin{proof}
(i) follows by straightforward calculation as in \cite{Kapouleas14}. 
By \ref{L:boundary-angle} we have $\Theta\equiv0$ on $\partial\K$ and then by \ref{D:boundary-angle} we conclude 
$X_\K\cdot\nu_\K=0$ which by (i) implies 
$\phie\circ\nu_\K=0$. 
By \ref{L:K-property}.i and the definition of $\xx_{crit}$ in \ref{Dphieo}
we have then $(z_{crit},-R_{crit})=(\cos\xx_{crit},-\sin\xx_{crit})$. 
This and \ref{L:K-property}.ii imply the result. 
\end{proof}

\subsection*{Catenoidal annuli orthogonal to $\Sph^2_+$}

We proceed now to classify the $O(2)$-invariant,  
immersed in the upper half ball 
$\B_+^3$ (recall \ref{E:BS+})  
minimal surfaces,  
which meet the upper hemisphere 
$\Sph^2_+$ orthogonally. 
Any such minimal surface is contained in a complete catenoid (or plane) $\Ktilde$ 
whose axis is the $z$-axis, 
so we lose no generality if we classify the catenoids (and planes) with these properties. 
If $\Ktilde$ is a plane it has to be the $xy$-plane, 
so we concentrate on the case where $\Ktilde$ is a catenoid. 
Each such catenoid $\Ktilde$  
is a translation along the $z$-axis of a scaling of the standard complete catenoid. 
Such a catenoid $\Ktilde$ can at most intersect the upper hemisphere $\Sph^2_+$ 
orthogonally once, 
as the intersection must happen above the waist of the catenoid, 
where $\Ktilde$ can be written as the graph of a monotonically increasing radial function 
over the exterior of some disk (with center at the origin) in the $xy$-plane. 
$\Ktilde$ clearly has to intersect the $xy$-plane exactly once. 
Therefore $\Ktilde$ either does not intersect the interior of the upper half ball 
$\B_+^3$ at all, 
or its intersection with $\B^3_+$ is an annulus 
with one boundary circle on $\Sph^2_+$ and the other on $\D$.  
The intersection along the first circle $\Ktilde\cap\Sph^2$ is orthogonal by assumption. 
We define $\theta$ by requiring that 
the angle between the outward normal of $\Ktilde$ 
and $e_z$ along the latter circle $\Ktilde\cap\D$ 
is $\theta+\pi/2 \in (0,\pi)$. 
As we will see there is at most one $\Ktilde$ for each $\theta$, 
so there is no ambiguity if we denote the radius of $\Ktilde\cap\D$ by $r_\theta$.  
Note that for the critical catenoid we have $\theta=0$ and $r_0 = r_{crit}$. 
A positive $\theta$ implies that $\Ktilde\cap\D$ 
lies above the waist of $\Ktilde$ and a negative $\theta$ 
that it lies below the waist, 
$\theta=\pi/2$ corresponds to the $xy$-plane.

\begin{lemma}
\label{L:cat-ODE}
There exists some $\theta_{min} \in(-\pi/2,0)$ such that the following hold: \\ 
(i). There is no $\Ktilde$ as above with $\theta<\theta_{min}$. \\ 
(ii). 
For each $\theta\in  [\theta_{min},\frac{\pi}{2})$ 
there is exactly one $\Ktilde$ as above, 
which we will denote by $\Ktilde_\theta$. 
$\Ktilde_\theta$ can be obtained by rotating the graph of $x=f_\theta(z)$ 
around the $z$-axis where 
$f_\theta:\R \to \R_+$ is given by
\[ f_\theta(z):=r_\theta \cosh \frac{z}{r_\theta \cos \theta} + r_\theta \sin \theta \sinh \frac{z}{r_\theta \cos \theta},\]
where $r_\theta>0$ is a constant depending smoothly on $\theta$. 
Moreover, $r_\theta$ is a decreasing function of $\theta$ with $r_{\theta_{min}}=1$ and 
$r_{\theta}\to e^{-1}$ as $\theta\to\frac{\pi}2 -$. See Figure \ref{fig:catenoid-annulus}.
\end{lemma}

\begin{figure}
\includegraphics[height=6cm]{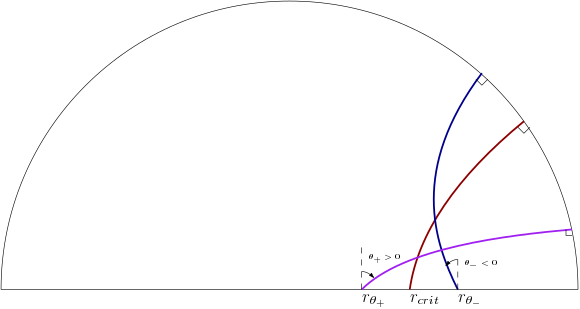}
\caption{The generating curves of the catenoids $\Ktilde_\theta$ inside $\B^3_+$ with $\theta_- <0<\theta_+$}
\label{fig:catenoid-annulus}
\end{figure}

\begin{proof}
Given $h\in(0,1)$ 
there is clearly a unique catenoid $\Khat_h$ whose axis is the $z$-axis 
and which intersects $\Sph^2_+$ orthogonally along a circle 
which contains the points $(\pm\sqrt{1-h^2},0,h)$.
It is well known that $\Khat_h$ can be obtained by rotating the graph of 
$x=f(z)$ around the $z$-axis where $f$ is given by 
\beq
\label{E:cat-ODE-sol} 
f(z)=\tilde{r} \cosh \frac{z-\tilde{z}}{\tilde{r}},
\eeq
for some $\tilde{r}>0$, $\tilde{z} \in \mathbb{R}$ which depend on $h$. 
$\Khat_h$ interects then the $xy$-plane along a circle of some radius $\widehat{r}>0$ 
with angle $\theta$ in the sense that 
the angle between the outward normal of $\Khat_h$ 
and $e_z$ 
is $\theta+\pi/2$. 
We have then 
\beq
\label{E:cat-ODE-sol-a}
-\sinh \frac{\tilde{z}}{\tilde{r}}=\tan \theta, 
\qquad
\sinh \frac{h-\tilde{z}}{\tilde{r}} =\frac{\sqrt{1-h^2}}{h}, 
\qquad
\tilde{r} \cosh \frac{h-\tilde{z}}{\tilde{r}} = \sqrt{1-h^2}, 
\eeq 
where the first equation amounts to $f'(0)=\tan\theta$ (by the definition of $\theta$),  
the third equation amounts to that $(\pm\sqrt{1-h^2},0,h)$ is contained in $\Khat_h$, 
and the second equation amounts to $f'(h)=\frac{\sqrt{1-h^2}}{h}$ 
(the orthogonality of $\Khat_h$ to $\Sph^2_+$ along the circle containing $(\pm\sqrt{1-h^2},0,h)$.  

To complete the proof it is enough to check that $d\theta/dh<0$ and $d\widehat{r}/dh>0$, 
so that $d\widehat{r}/d\theta<0$ would follow. 
We solve 
the last two equations in 
\ref{E:cat-ODE-sol-a} 
for $\tilde{r}$ and $\tilde{z}$ 
to get
\[ \tilde{r}=h \sqrt{1-h^2}, \qquad  \tilde{z}=h - \tilde{r} \arcsinh \frac{\sqrt{1-h^2}}{h}.\]
From these we see that 
\[ \frac{d}{dh} \left( \frac{\tilde{z}}{\tilde{r}} \right) = \frac{1}{h (1-h^2)^{3/2}} >0.\]
By the first equation in \ref{E:cat-ODE-sol-a} this implies $d\theta/dh <0$.

Clearly now $\widehat{r}=f(0)=\tilde{r} \cosh \frac{\tilde{z}}{\tilde{r}}$. 
Hence by differentiating with respect to $h$ we obtain 
\[ \frac{d \widehat{r}}{dh}= \frac{h}{1-h^2} \left( (3-2h^2) \sinh \frac{1}{\sqrt{1-h^2}} -2 \sqrt{1-h^2} \cosh \frac{1}{\sqrt{1-h^2}} \right).\]
Using the elementary inequality 
\[ \frac{3-2h^2}{2 \sqrt{1-h^2}} \geq \sqrt{2} > \coth (1) > \coth \frac{1}{\sqrt{1-h^2}} \qquad \text{ for } h \in (0,1), \]
we conclude that $d\widehat{r}/dh>0$ 
and the proof is complete. 
\end{proof}

From the proposition and our discussion above, we have the following uniqueness theorem 
(see also \cite{Fraser-Schoen11}, \cite{Fraser-Schoen15a} and \cite{Maximo-Nunes-Smith13} for similar uniqueness theorems). 
Note that in our uniqueness theorem 
there is no apriori assumption on the topology of the surface 
or rotational invariance requirement in the interior. 

\begin{corollary}
\label{C:unique-catenoid}
The only embedded free boundary minimal surfaces in $\B^3$ with at least one rotationally invariant 
(about the $z$-axis) boundary component on $\Sph^2$ are the equatorial disk $\D$ and the critical catenoid $\K$.
\end{corollary}

\begin{proof}
By Bj\"{o}rling's uniqueness theorem \cite{Bjorling}, 
the minimal surface is rotationally invariant in a neighborhood of the rotationally invariant boundary component. 
By unique continuation of minimal surfaces, the entire surface is a piece of either a complete catenoid or the equatorial plane. 
In the first case, 
applying \ref{L:cat-ODE} to the upper and lower half of this complete catenoid we get 
$r_\theta=r_{-\theta}$, 
where $\theta+\frac\pi2$ is the angle at which the catenoid intersects the equatorial plane. 
By monotonicity of $r_\theta$ with respect to $\theta$, we have $\theta=0$, 
which implies that the free boundary minimal surface is the critical catenoid $\K$.
\end{proof}

\begin{definition}
\label{E:Kthetaplus}
For $\theta \in (\theta_{min},\pi/2)$ 
we define 
$
\Ktilde^+_\theta := \Ktilde_\theta\cap\R^3_+   
$
and 
$
\K^+_\theta
:= \Ktilde_\theta\cap\B^3_+   
$.
We also define $\Ktilde^-_\theta\subset\R^3_-$ and $\K^-_\theta\subset\B^3_-$ to be the mirror images under reflection with respect to the $xy$-plane of 
$\Ktilde^+_\theta$ and $\K^+_\theta$ respectively.  
For future reference we define 
\beq
\label{D:initial-config}
\Wcal_\theta:= \K^+_\theta \cup \K^-_\theta \cup \D,  
\qquad
\Wcaltilde_\theta:= \Ktilde^+_\theta \cup \Ktilde^-_\theta \cup \{z=0\} \supset \Wcal_\theta , 
\qquad
\Ctheta := \K^\pm_\theta \cap \D.  
\eeq
\end{definition}

By 
\ref{L:cat-ODE} 
each  $\K^+_\theta$ meets $\Sph^2_+$ orthogonally and $\D$ at an angle $\theta+\frac{\pi}2$ along 
$\Ctheta$ which is the circle of radius $r_\theta$ on the plane $\{z=0\}$.  
Note that 
$\K^\pm_0=\K \cap \B^3_\pm$ and 
$\Wcal_\theta$ is a perturbartion of our initial configuration $\Wcal_0=\K\cup\D$.  
Clearly $\Wcal_\theta$ and its 
complete extension $\Wcaltilde_\theta$ contain the circle $\Ctheta$ 
and $\Wcaltilde_\theta\setminus \Ctheta$ is smooth and embedded. 
$\Wcal_\theta$ and $\Wcaltilde_\theta$ are symmetric under reflections 
with respect to the $xy$-plane and lines on the $xy$-plane through the origin. 
We parametrize now $\K^\pm_\theta$ by a fixed cylinder independent of $\theta$: 

\begin{definition} 
\label{D:half-catenoids}
For each $\theta \in (\theta_{min},\frac{\pi}{2})$ 
we define a diffeomorphism $X_{\K^+_\theta}$ (or $X_{\K^-_\theta}$)  
from the cylinder $[0,z_{crit}] \times \Sph^1$ 
(recall that $\Sph^1 = \R/2\pi\Z$) 
onto  
$\K^+_\theta$ (or $\K^-_\theta$) by 
\[ X_{\K^\pm_\theta}(s,y):=(f_\theta( z) \cos y, f_\theta( z) \sin y, \pm z), \qquad \text{with } z=\frac{z_\theta}{z_{crit}} s \]
where $z_{crit}$ is as in \ref{L:K-property}.i, $f_\theta$ as in \ref{L:cat-ODE}, 
and $\{z=z_\theta\}$ is the plane containing the circle $\partial \K^+_\theta \cap \Sph^2_+$. 
We also define 
two families of diffeomorphisms 
$\Fcal_{\K^\pm_\theta}:=X_{\K^{\pm}_\theta} \circ X_{\K^{\pm}_0}^{-1} :\K^{\pm}_0 \to \K^{\pm}_\theta$. 
\end{definition}

\begin{definition}
\label{D:half-disks}
For each $\theta \in (\theta_{min},\frac{\pi}{2})$, we define the annulus $\A_\theta$ and the disk $\D_\theta$ contained in the equatorial disk $\D$ by
\[ \A_\theta:=\{(x,y,z) \in \D : x^2+y^2 \geq r_\theta^2\}, \qquad 
\D_\theta:=\{(x,y,z) \in \D : x^2+y^2 \leq r_\theta^2\},\]
where $r_\theta$ is defined as in \ref{L:cat-ODE}, and they are oriented by the unit normals $\nu_{\A_\theta}=-e_z$ and 
$\nu_{\D_\theta}=e_z$ respectively. 
Moreover, we define the family of diffeomorphisms $\Fcal_{\A_\theta}:\A_0 \to \A_\theta$ and $\Fcal_{\D_\theta}: \D_0 \to \D_\theta$ by
\[ \Fcal_{\A_\theta}(x,y,z):=\left( r_\theta+ \frac{1-r_\theta}{1-r_{crit}} (\sqrt{x^2+y^2}-r_{crit})\right)(x,y,z), \]
\[ \Fcal_{\D_\theta}(x,y,z):=\left( \frac{r_\theta}{r_{crit}} \sqrt{x^2+y^2}\right)(x,y,z). \]
\end{definition}

\begin{lemma}[Norm comparison]
\label{L:K-norms}
For $\epsilon>0$ and $|\theta|$ sufficiently small in terms of $\epsilon$, 
and for any function $u:\Omega \to \R$ defined on a domain $\Omega \subset S$  
where $S$ is any of $\K^+_0$, $\K^-_0$, $\A_0$, or $\D_0$,  
we have (recall \ref{D:sim})
\[ \| u \circ \Fcal_{S_\theta}^{-1} : C^{k,\beta}(\,\Fcal_{S_\theta}(\Omega)\, ) \| \sim_{1+\epsilon} \| u : C^{k,\beta}(\Omega)\|,\]
where $S_\theta$ is the corresponding 
$\K^+_\theta$, $\K^-_\theta$, $\A_\theta$, or $\D_\theta$,  
and the norms are taken with respect to the induced metric on $S_\theta \subset \B^3$. 
The same estimate holds if $\Omega$ is assumed to be a (one-dimensional) domain in the circle $S\cap\Sph^2$. 
\end{lemma}

\begin{proof}
It follows directly from the smooth dependence on $\theta$ 
of $\Fcal_{S_\theta}^* g_0$ on $S$. 
\end{proof}

We have now the following so that we can identify different $\Wcal_\theta$'s: 

\begin{definition}
\label{D:FW}
For each $\theta \in (\theta_{min},\frac{\pi}{2})$ 
we define a bijection $\Fcal_{\Wcal_\theta}:\Wcal_0 \to \Wcal_\theta$ 
by requiring 
$\left.  \Fcal_{\Wcal_\theta} \right|_{S_0} =\Fcal_{S_\theta}$, 
where $S_\theta$ is any of 
$\K^+_\theta$, $\K^-_\theta$, $\A_\theta$, or $\D_\theta$   
(recall \ref{D:half-catenoids} and \ref{D:half-disks}). 
Note that the restriction of $\Fcal_{\Wcal_\theta}$ to 
$\Wcal_0\setminus\Circle_0$ is a smooth diffeomorphism onto 
$\Wcal_\theta\setminus\Ctheta$.  
\end{definition}

\subsection*{Kernels of the Standard Pieces}
In this subsection, 
we study the kernels of the linearized equations on the four standard pieces 
$\K^+_0$, $\K^-_0$, $\A_0$, and $\D_0$.  
Note that these standard pieces are subsets of the equatorial disk $\D$ or the critical catenoid $\K$, which are minimal (recall \ref{R:H-linearized}), 
with $\D=\A_0 \cup \D_0$ and $\K=\K^+_0 \cup \K^-_0$. 
We will show that there is no rotationally invariant solutions to the linearized equations on each of these standard pieces.

\begin{lemma}
\label{L:linear-disk}
There is no non-trivial harmonic function on $\D_0$ with homogeneous Dirichlet boundary data.
\end{lemma}

\begin{proof}
This follows directly from the maximum principle for harmonic functions.
\end{proof}

\begin{lemma}
\label{L:linear-annulus}
There is no non-trivial solution to the following boundary value problem on $\A_0$:
\[ \left\{ \begin{array}{cl}
\Delta f=0 & \text{ on }\A_0,\\
f=0 & \text{ along } \partial \A_0 \setminus \Sph^2, \\
-\frac{\partial f}{\partial \eta} +f=0 & \text{ along } \partial \A_0 \cap \Sph^2. \end{array} \right.\]
\end{lemma}

\begin{proof}
Let $(r,\theta)$ be the polar coordinate system on the $xy$-plane. If $f=f(r,\theta)$ is a solution to the boundary value problem, then 
\begin{equation*}
\left\{ \begin{array}{c}
\frac{\partial^2 f}{\partial r^2} + \frac{1}{r} \frac{\partial f}{\partial r} +\frac{1}{r^2} \frac{\partial^2 f}{\partial \theta^2} =0 \\ 
f(r_{crit},\theta)=0\\
-\frac{\partial f}{\partial r}(1,\theta)+f(1,\theta) =0
\end{array} \right. .
\end{equation*}
By separation of variables, write $f(r,\theta)=\sum_{m=0}^\infty R_m(r) \Theta_m(\theta)$, 
the angular component $\Theta_m(\theta)$ is a linear combinations of $\sin (m\theta)$ and $\cos (m \theta)$, 
$m=0,1,2,3,\ldots$, and the radial component $R_m(r)$ satisfies the following ODE:
\begin{equation*}
\left\{ \begin{array}{c}
r^2 R_m''(r)+r R'_m(r) -m^2 R_m(r) =0 \\
R_m(r_{crit})=0 \\
R'_m(1)=R_m(1) 
\end{array} \right. .
\end{equation*}
The general solutions to the ODE is $R_0(r)=A+B \log r$ when $m=0$, and $R_m(r)=A r^m + B r^{-m}$ when $m=1,2,3,\ldots$. For $m>0$, the boundary conditions imply
\begin{equation*}
\left\{ \begin{array}{rl}
A r_{crit}^m  + B r_{crit}^{-m} &=0\\
(m-1)A-(m+1)B&=0 
\end{array} \right. ,
\end{equation*}
which has no nontrivial solutions. When $m=0$, the boundary conditions imply 
\begin{equation*}
\left\{ \begin{array}{rl}
A+B \log r_{crit} & =0\\
A-B&=0 
\end{array} \right. ,
\end{equation*}
which has no nontrivial solution since $r_{crit} > e^{-1}$. This proves the lemma.
\end{proof}

\begin{lemma}
\label{L:linear-cat}
Let $S=\K^+_0$ or $\K^-_0$. Then there is no non-trivial rotationally symmetric solution to the following boundary value problem on $S$:
\[ \left\{ \begin{array}{cl}
\Delta_S f +|A_S|^2 f=0 & \text{ on }S,\\
f=0 & \text{ along } \partial S \setminus \Sph^2 = \K\cap\D , \\
-\frac{\partial f}{\partial \eta} +f=0 & \text{ along } \partial S \cap \Sph^2. \end{array} \right.\]
\end{lemma}

\begin{proof}
Using \ref{L:K-property}, \ref{E:spherical-coord}, and \ref{L:K-Gauss}, 
we can write the equations in spherical coordinates $(\xx,\yy)$ as (assuming the solution is independent of $\yy$) :
\[
\left\{ \begin{array}{cl}
f''(\xx)- (\tan \xx) f'(\xx)+2 f(\xx) =0  & \text{ for } \xx \in (0,\xx_{crit})\\
f(0)=0 & \\
- (\cot \xx_{crit}) \,  f'(\xx_{crit})+f(\xx_{crit})=0. &
\end{array} \right. 
\]
By \ref{L:K-Gauss} a solution of the ODE is a linear combination of 
$\phie$ and $\phio$. 
Since $\phie(0)=1$ and $\phio(0)=0$ the space of the ODE solutions which 
satisfy the Dirichlet boundary condition is spanned by 
$\phio$. 
The Robin condition for $f_{odd}$ amounts to 
$\cot \xx_{crit} \cos \xx_{crit} =\sin \xx_{crit} $ which is equivalent to 
$\xx_{crit}=\pi/4$. 
This is not true by \ref{L:K-Gauss}.ii and the proof is complete.

Alternatively we can consider the space of ODE solutions which satisfy the Robin condition. 
$\phie$ does not satisfy the Robin condition because it vanishes at $\xx_{crit}$ and if its derivative 
also vanished it would vanish identically. 
An easy calculation shows that 
\beq
\label{E:robin}
\phi_{Robin}:=(\cos2\xx_{crit}) \, \phie - (\cos\xx_{crit})\,\phieder(\xx_{crit}) \, \phio  
\eeq
spans this space. 
Clearly $\phi_{Robin}(0)\ne0$ and the lemma follows. 
\end{proof}

\begin{remark}
Note that clearly $\phi'_{Robin}(0)\ne0$ 
also, 
and therefore \ref{L:linear-cat} holds as well with a Neumann condition along $\K\cap\D$ instead of Dirichlet. 
This implies 
the existence of $\K^\pm_\theta$ for small $\theta$ independently of \ref{L:cat-ODE}. 
\end{remark}


\section{The initial surfaces}
\label{S:initial}
\label{S:desing}

In this section we first construct and study the desingularizing surfaces,  
and then we use them to replace a neighborhood of $\Ctheta$ in $\Wcal_\theta$ (recall \ref{D:initial-config}) 
so that we obtain the initial surfaces which are smooth and embedded. 
As in \cite{Kapouleas97} for example, 
the desingularizing surfaces are modeled in general on the classical singly periodic Scherk surfaces \cite{Scherk35} which form a one-parameter family 
of embedded minimal surfaces parametrized by $\alpha \in (0,\pi/2)$ where $2 \alpha$ is the angle between two of their four asymptotic half-planes 
(see \cite{Kapouleas11} for example). 
Because of the extra symmetries in our construction, 
we only need to use the Scherk surface (unique up to rigid motions and scaling) with $\alpha=\pi/4$ whose asymptotic planes are perpendicular. 

\subsection*{The singly periodic Scherk surface}

\begin{definition}
\label{D:scherk}
We denote $\Sch$ to be the Scherk surface defined by 
\begin{equation*}
\Sch  :=  \{  (x,y,z)\in\R^3 : \sinh x \sinh z = \cos y  \},
\end{equation*}
oriented by the unit normal $\nu_\Sch$ such that $\nu_\Sch \cdot e_z < 0$ on $\Sch \cap \{x>0\}$. 
\end{definition}

Note that $\Sch$ is also the most symmetric surface in the one-parameter family of Scherk surfaces. 
Some of the extra symmetries it possesses can be imposed in our constructions (see \ref{r:seven}). 
The Scherk surface $\Sch$ is singly periodic along the $y$-axis with period $2\pi$. 
Moreover, away from the $y$-axis, $\Sch$ is asymptotic to the planes $\{z=0\}$ and $\{x=0\}$ near infinity. 
The symmetries of $\Sch$ are summarized in the lemma below 
(see also Figures \ref{fig:planes} and \ref{fig:lines}).

\begin{figure}
\centering
\begin{minipage}{0.45\linewidth}
\includegraphics[width=.8\linewidth]{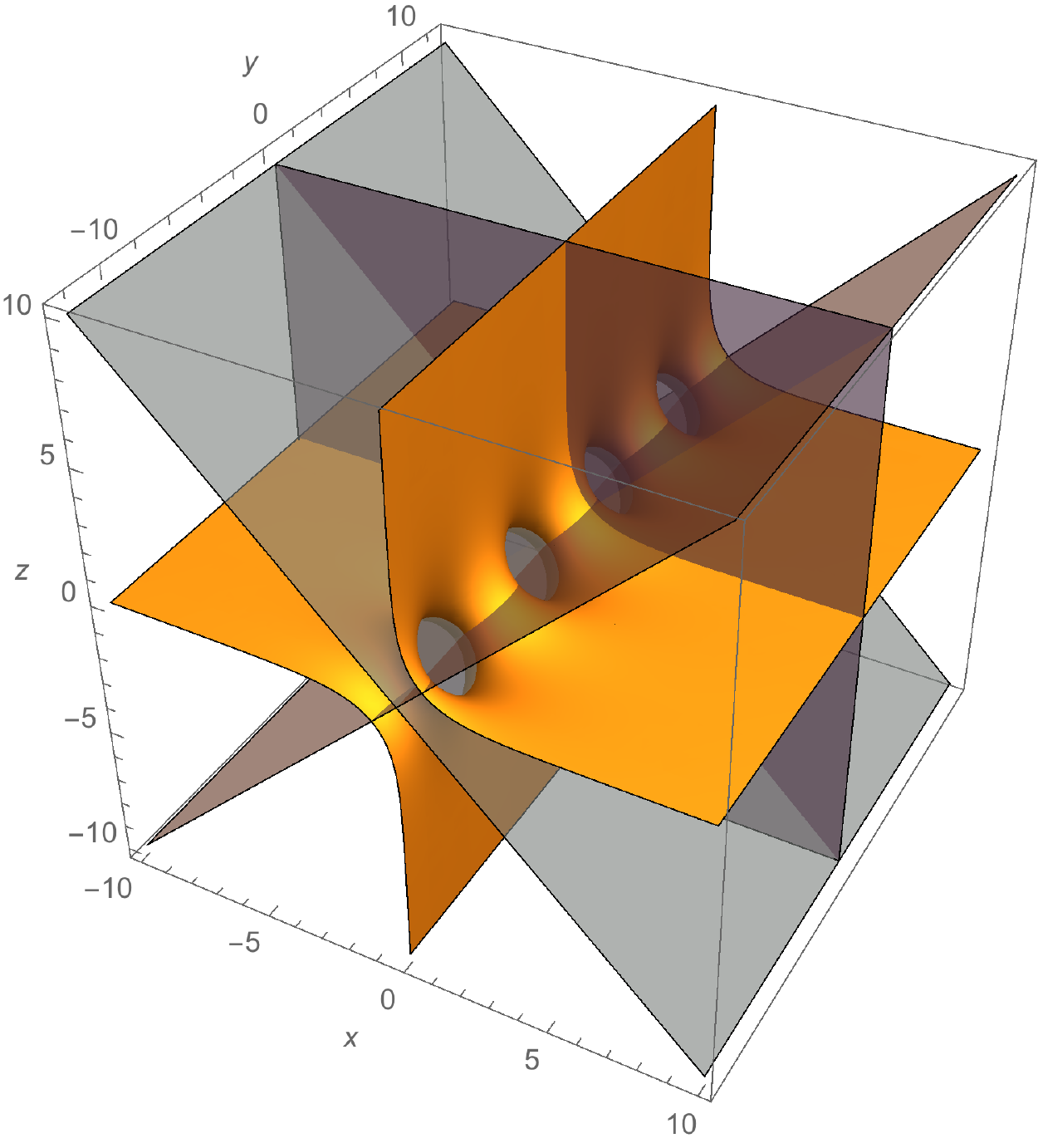}
\caption{$\mathcal{S}$ with its planes of symmetry}
\label{fig:planes}
\end{minipage}
\quad 
\begin{minipage}{0.45\linewidth}
\includegraphics[width=.8\linewidth]{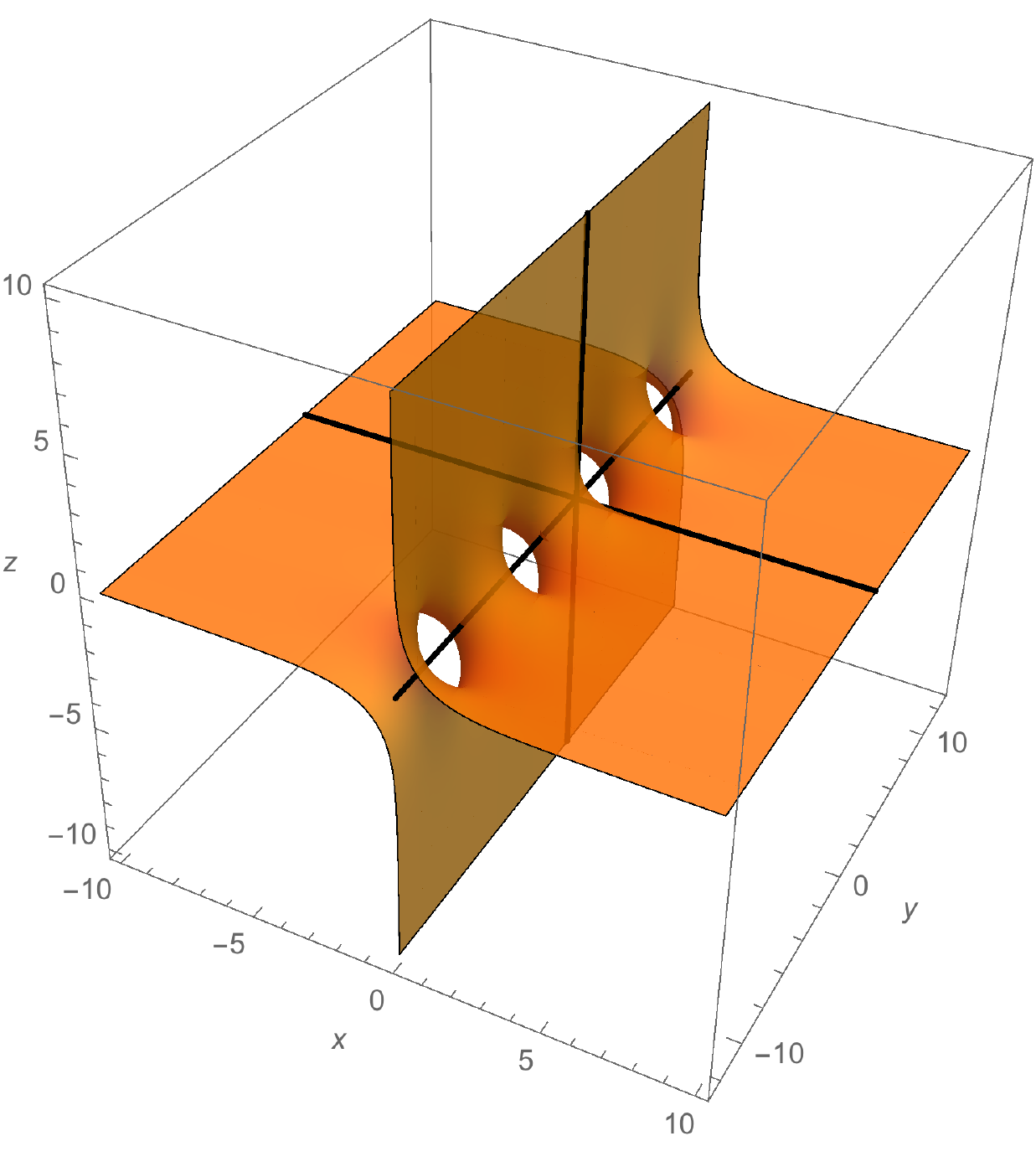}
\caption{$\mathcal{S}$ with its lines of symmetry}
\label{fig:lines}
\end{minipage}
\end{figure}

\begin{lemma}
\label{L:ScherkSym}
The Scherk surface $\mathcal{S}$ is a singly periodic complete embedded minimal surface 
with period $2\pi$ along the $y$-axis and it is invariant under reflections about \\
(i). the planes $\{y=n\pi\}$ ($n\in\Z$), $\{x=z\}$ and $\{x=-z\}$;\\
(ii). the lines $\{x=z=0\}$, $\{y=(n+\frac12)\pi,z=0\}$ and $\{x=0,y=(n+\frac12)\pi\}$ ($n\in\Z$).\\ 
The group $\group'_\Sch$ generated by these reflections is the group of symmetries of $\mathcal{S}$. 
\end{lemma}

\begin{proof}
This can be checked directly using the defining equation for $\Sch$ in \ref{D:scherk}. 
\end{proof}

Note that the lines of symmetry $\{y=(n+\frac12)\pi,z=0\}$ and $\{x=0,y=(n+\frac12)\pi\}$ lie on the surface $\Sch$ 
(see Figure \ref{fig:lines}).
We now pick some of the symmetries which we would like to preserve in our constructions: 

\begin{definition}
\label{D:GSch}
Let 
$\YYbar$, $\YYbar_{\pi}$, and $\YYhbar_{\frac{\pi}{2}}$, 
be the reflections about the planes $\{y=0\}$, $\{y=\pi\}$ and the line $\{y=\pi/2,z=0\}$, respectively,  
or equivalently given by 
$$
\YYbar (x,y,z):=(x,-y,z), \quad \YYbar_{\pi}(x,y,z):=(x,2\pi-y,z), \quad \YYhbar_{\frac{\pi}{2}}(x,y,z)=(x,\pi-y,-z).  
$$
Note that all the reflections defined above are orientation-reversing diffeomorphisms on $\Sch$. 
We define subgroups of symmetries 
$\group^0_\Sch \subset \group_\Sch \subset \group'_\Sch$ 
as the subgroups generated by $\YYbar$ and $\YYbar_{\pi}$, 
and by $\YYbar$, $\YYbar_{\pi}$, and  $\YYhbar_{\frac{\pi}{2}}$ respectively. 
\end{definition}

\begin{remark}
\label{r:seven}
The symmetries in $\group_\Sch$ will be imposed on our constructions.
Because of the extra symmetries corresponding to reflections with respect to lines 
contained in the surfaces one can reduce the dimension of the extended substitute kernel 
(and therefore the number of parameters for the family of initial surfaces also) 
from six (per circle of intersection) as in \cite{Kapouleas97} to one, 
thus greatly simplifying the construction. 
In particular there is no need to dislocate the wings relative to the core, 
and so the $\underline{\varphi}$ parameters of \cite{Kapouleas97} are not needed. 
\end{remark}

\begin{figure}
\includegraphics[height=6cm]{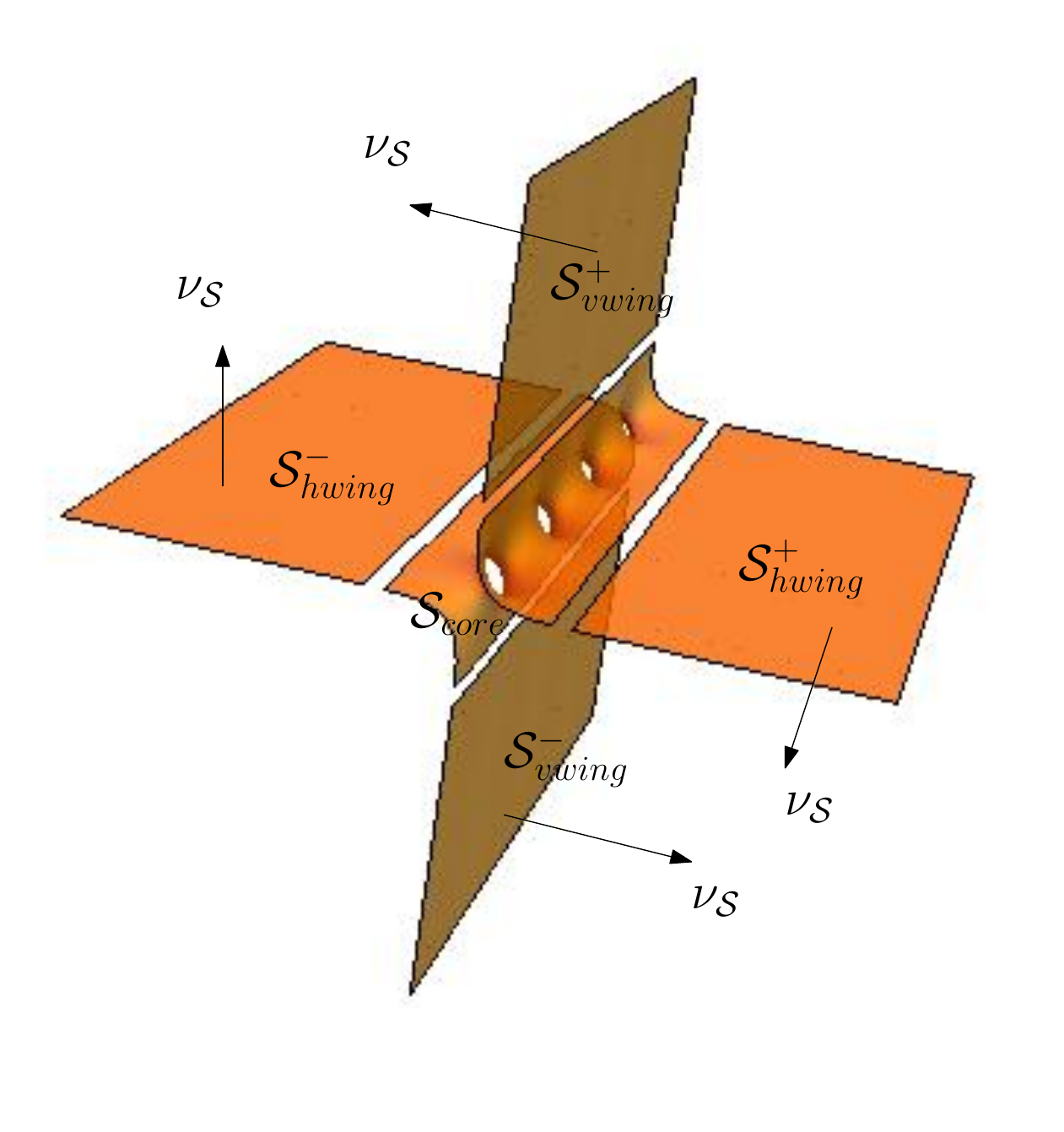}
\caption{A core-wings decomposition of $\Sch$ with its unit normal $\nu_\Sch$}
\label{fig:core-wing}
\end{figure}

The model Scherk surface $\Sch$ can be divided roughly into five regions: 
a central core, two wings asymptotic to the horizontal plane $\{z=0\}$ and two wings asymptotic to the vertical plane $\{x=0\}$. 
(see Figure \ref{fig:core-wing}).
The core is within a finite distance from the $y$-axis and contains all the topology of the surface. 
Each wing is simply connected and can be expressed as the graph of a small function over its asymptotic plane near infinity. 
The location of the transition from the core to the wings is not important 
as long as it is far enough from the axis 
to ensure that the wings are sufficiently close to the asymptotic planes in order to get good uniform estimates. 
Lemma \ref{L:phi-wing} below tells us that the wings decay exponentially fast to their asymptotic planes near infinity. 
Notice that it suffices to give the description of one wing since the others can be similarly described by reflecting across the planes $\{x=z\}$ and $\{x=-z\}$. 
Recall that $\R^2_+:=\{(y,s): s \ge 0\}$ is the half-space equipped with the standard orientation and flat metric $g_0$.

\begin{definition}[Core-wings decomposition]
\label{D:core-wings}
We assume given $a>2$. 
We define the immersions $X^\pm_{hor}, X^\pm_{ver}:\R^2_+ \to \R^3$ by
\[ X^\pm_{hor}(y,s):=(\pm(a+s),y,0), \quad X^\pm_{ver}(y,s):=(0,-y,\pm(a+s)), \]
of the horizontal and vertical asymptotic half-planes of $\mathcal{S}$. Moreover, we define a function $\varphi_{wing}:\R^2_+ \to \R$ by
\[ \varphi_{wing}(y,s):=-\log \left( \frac{\cos y}{\sinh (a+s)}+\sqrt{1+\frac{\cos^2 y}{\sinh^2 (a+s)}} \right), \]
and the immersions $X^\pm_{hwing},X^\pm_{vwing}:\R^2_+ \to \R^3$ of the horizontal and vertical wings (recall \ref{d:imm}) and their images by
\[ X^\pm_{hwing}:=\Immer[X^\pm_{hor},\varphi_{wing};\R^2_+], \quad X^\pm_{vwing}:=\Immer[X^\pm_{ver},\varphi_{wing};\R^2_+],\]
\[ \Sch^\pm_{hwing}:=\Graph[X^\pm_{hor},\varphi_{wing};\R^2_+], \quad \Sch^\pm_{vwing}:=\Graph[X^\pm_{ver},\varphi_{wing};\R^2_+].\]
Note that $\Sch^\pm_{hwing}$ and $\Sch^\pm_{vwing}$ are disjoint subsets of $\Sch$ and we define the core of the Scherk surface as
\[ \Sch_{core}:=\Sch \setminus ( \Sch^\pm_{hwing} \cup \Sch^\pm_{vwing}).\]
We also define a smooth function $s:\Sch \to \R$ to be equal to the coordinate $s$ induced on $\Sch^\pm_{hwing} \cup \Sch^\pm_{vwing}$ by the immersions $X^\pm_{hwing},X^\pm_{vwing}:\R^2_+ \to \R^3$ and equal to any smooth negative function on $\Sch_{core}$ which is symmetric with respect to $\group_\Sch$.
\end{definition}

Therefore (depending on $a>2$) we have the following core-wings decomposition of $\Sch$:
\[ \Sch = \Sch_{core} \cup \Sch^\pm_{hwing} \cup \Sch^\pm_{vwing},\]
with $\Sch_{s < 0}=\Sch_{core}$ and $\Sch_{s \ge 0}= \Sch^\pm_{hwing} \cup \Sch^\pm_{vwing}$ (recall \ref{E:sub-super}). 
The fact that the wings approach their asymptotic half planes exponentially fast is given by the lemma below.

\begin{lemma}
\label{L:phi-wing}
$\|\varphi_{wing}:C^5(\R^2_+,g_0,e^{-s})\| < C\, e^{-a}$ for some absolute constant $C$.
\end{lemma}

\begin{proof}
This follows from the exact expression for $\varphi_{wing}$ in \ref{D:core-wings}. 
\end{proof}

\begin{lemma}
\label{L:Scherk-equivariant}
(i) $\varphi_{wing}$ satisfies the following symmetries:
\[ \varphi_{wing}(-y,s)=\varphi_{wing}(y,s), \quad \varphi_{wing}(2\pi-y,s)=\varphi_{wing}(y,s),\]
\[ \varphi_{wing}(\pi-y,s)=-\varphi_{wing}(y,s).\]
(ii) $X^\pm_{hor}$ satisfies the following symmetries (recall \ref{D:GSch}):
\[ X^\pm_{hor}(-y,s)= \YYbar \circ X^\pm_{hor}(y,s), \quad X^\pm_{hor}( 2\pi -y,s) =\YYbar_{\pi} \circ X^\pm_{hor}(y,s),\]
\[ X^\pm_{hor}(\pi-y,s)=\YYhbar_{\frac{\pi}{2}} \circ X^\pm_{hor}(y,s).\]
(iii) $X^\pm_{ver}$ satisfies the following symmetries (recall \ref{D:GSch}):
\[ X^\pm_{ver}(-y,s)= \YYbar \circ X^\pm_{ver}(y,s), \quad X^\pm_{ver}( 2\pi -y,s) =\YYbar_{\pi} \circ X^\pm_{ver}(y,s),\]
\[ X^\pm_{ver}(\pi-y,s)=\YYhbar_{\frac{\pi}{2}} \circ X^\mp_{ver}(y,s).\]
\end{lemma}

\subsection*{The desingularizing surfaces}

In this subsection we perturb the model Scherk surface $\Sch$ to a family of surfaces $\Sch_{\theta,\tau}$ 
depending smoothly on two small continuous parameters $\tau$ and $\theta$. 
These surfaces will be constructed as the images of a smooth family of immersions $\Zscr_{\theta,\tau}:\Sch \to \R^3$ 
such that $\Zscr_{0,0}$ is the identity map on $\Sch$ and $\Zscr_{\theta,\tau}$ converges locally uniformly to $\Zscr_{0,0}$ as $\tau,\theta \to 0$. 
This allows us to study the geometric and analytic properties of $\Sch_{\theta,\tau}$ 
from the corresponding properties of $\Sch$ by a perturbation argument 
(additional care needs to be taken as these surfaces are non-compact). 
In the next subsection, we will use these surfaces, after suitably translated and rescaled, to desingularize the singularity circle $\Ctheta$ 
in the initial configuration $\Wcal_\theta$ (recall \ref{D:initial-config}) to obtain a family of smooth initial surfaces.
We first have the following.

\begin{convention}
\label{Con:parameters}
\label{Con:gamma}
\label{Con:delta}
\label{Con:initial}
We will assume that the parameters $\tau$ and $\theta$ satisfy
\[ |\tau| < \dtau \qquad \text{ and } \qquad |\theta|< \dtheta,\]
for some small constants $\dtau,\dtheta>0$.
For future use we fix constants $\beta,\gamma \in (0,1)$, for example $\beta=\gamma=3/4$, and also 
a small constant $\ds>0$.
We will always assume that $a$ is as large as needed in absolute terms, 
$\dtheta$ and $\ds$ are as small as needed in terms of $a$, 
and $\dtau$ is as small as needed in terms of $a, \dtheta, \ds, \beta, \gamma$. 
Finally the initial surfaces $M_{\theta,m}$ we will construct in \ref{D:initial}  
will depend on parameters $\theta$ and $m$,  
where $\theta\in [-\dtheta,\dtheta]$ as above,  
$m\in\N\cap(1/\dtau,\infty)$, 
and when $m$ is chosen we have $\tau = m^{-1}$. 
\qed
\end{convention}

We discuss now the geometric meaning of the parameters: 
The parameter $\theta$ measures the amount of unbalancing which must be introduced due to the existence of a one-dimensional kernel 
(modulo the symmetry group $\group_\Sch$) to the linearized equation on $\Sch$ (see \ref{L:Scherk-kernel}). 
The parameter $\tau$ will describe the bending needed to wrap the axis ($y$-axis) of the Scherk surface into a circle of radius $\tau^{-1}$, 
which will later be rescaled and translated to fit the circle of singularity $\Ctheta$ in the initial configuration $\Wcal_\theta$. 
We start by defining the family of maps $\Xi_\theta$ which create unbalancing.

\begin{definition}
\label{D:unbalance}
We define a family of smooth maps $\Xi_\theta:\R^3 \to \R^3$ by 
\[ \Xi_\theta:=\psi \id_{\R^3}+(1-\psi) R_\theta,\]
where $\psi:=\psi_{cut}[2,1](z):\R^3 \to \R$ (recall \ref{E:cutab}), 
$\id_{\R^3}$ is the identity map on $\R^3$,  
and $R_\theta:\R^3 \to \R^3$ is a map that acts on $\{ \pm z > 1\}$ by rotation around the $y$-axis with angle $\mp \theta$.  
\end{definition}

The following properties of $\Xi_\theta$ are easy to verify from the definitions.

\begin{lemma}
\label{L:unbalance}
(i). $\Xi_\theta$ depends smoothly on the parameter $\theta$ for $|\theta| < \dtheta$ and $\Xi_0=\id_{\R^3}$.\\
(ii). For $|\theta|$ sufficiently small, 
$\Xi_\theta(\Sch)$ is an embedded surface and 
$\Xi_\theta$ restricts to a diffeomorphism from $\Sch$ to $\Xi_\theta(\Sch)$. 
Moreover $\Xi_\theta$ rotates the vertical wings $\Sch^\pm_{vwing}$ about the $y$-axis by an angle $\mp \theta$, 
and keeps the horizontal wings $\Sch^\pm_{hwing}$ pointwise fixed (recall \ref{D:core-wings}).\\
(iii). $\Xi_\theta$ is $\group_\Sch$-equivariant, 
that is it commutes with all the symmetries in $\group_\Sch$ (recall \ref{D:GSch}).
\end{lemma}

Next, we define the family of maps $\Bscr_\tau$ which introduce the bending wrapping the $y$-axis 
around a circle of radius $\tau^{-1}$. 
In order to get an embedded surface, we are mainly interested in the values of $\tau$ such that $\tau^{-1}=m \in \N$, where $m$ is large. 
To facilitate the presentation we first define a discrete subgroup 
$\group_m$ 
of the continuous group $\group_\infty$ of symmetries defined in \ref{D:G'}: 

\begin{definition}
\label{D:Gm}
For any $m \in \N$, $m \geq 3$, we define $\group_m$ to be the group of isometries of $\R^3$ generated by
\[ \YYbar'(x,y,z):=(x,-y,z), \qquad \YYbar'_{\frac{\pi}{m}}(x,y,z):=(x \cos \frac{2\pi}{m} + y \sin \frac{2\pi}{m},x \sin \frac{2\pi}{m} - y \cos \frac{2\pi}{m},z), \]
\[ \YYhbar'_{\frac{\pi}{2m}}(x,y,z):=(x \cos \frac{\pi}{m} + y \sin \frac{\pi}{m},x \sin \frac{\pi}{m} - y \cos \frac{\pi}{m}, -z), \]
which are respectively the reflections about the planes $\{y=0\}$, $\{y= x \tan \frac{\pi}{m}\}$ and the line $\{ y = x \tan \frac{\pi}{2m}, z=0\}$.
\end{definition}

\begin{definition}
\label{D:bending}
We define the family of smooth maps $\Bscr_\tau:\R^3 \to \R^3$ by taking $\Bscr_0=\id_{\R^3}$ and for $\tau \neq 0$,
\[ \Bscr_\tau(x,y,z):=(-\tau^{-1},0,0)+ ((\tau^{-1}+x) \cos \tau y, (\tau^{-1}+x) \sin \tau y ,z). \]
\end{definition}

\begin{lemma}
\label{L:bending}
(i). $\Bscr_\tau$ depends smoothly on the parameter $\tau \in \R$ and $\Bscr_\tau(0)=0$.\\
(ii). When $\tau \neq 0$, $\Bscr_\tau$ wraps the $y$-axis isometrically onto the circle in the plane $\{z=0\}$ centered at $(-\tau^{-1},0,0)$ through the origin. Moreover, it restricts to an isometry on each vertical plane $\{y=c\}$, $c\in \R$, onto its image.\\
(iii). When $\tau^{-1} =m \in \Z$, the maps $\Bscr'_\tau(x,y,z):=\Bscr_\tau(x,y,z) + (\tau^{-1},0,0)$ 
are equivariant with respect to $\group_\Sch$ and $\group_m$:
\[  \Bscr'_\tau \circ \YYbar= \YYbar' \circ \Bscr'_\tau, \quad \Bscr'_\tau \circ \YYbar_{\pi}= \YYbar'_{\frac{\pi}{m}} \circ \Bscr'_\tau, \quad \Bscr'_\tau \circ \YYhbar_{\frac{\pi}{2}} = \YYhbar'_{\frac{\pi}{2m}} \circ \Bscr'_\tau. \]
\end{lemma}

Roughly speaking, we will first apply the unbalancing map $\Xi_\theta$ to the Scherk surface $\Sch$ and then apply the bending map $\Bscr_\tau$ to wrap the axis around a circle. However, the resulting surface would not be approximately minimal since the vertical asymptotic half planes of $\Sch$ would be bent into cones, which are not minimal surfaces. Therefore, the resulting surface will be asymptotic to these non-minimal cones near the vertical ends, whose error in the mean curvature would be too large to be corrected by a fixed point argument. To remedy this situation, we need to introduce a further bending so that the vertical asymptotic half planes become half catenoids, which are minimal. We will then build the wings of our desingularizing surface as the graph of $\varphi_{wing}$ (recall \ref{D:core-wings}) over such bent catenoids. Since $\varphi_{wing}$ is a function defined on $\R^2_+$, we need to give a parametrization $(X^\pm_{ver})_{\theta,\tau}:\R^2_+ \to \R^3$ of the bent catenoids. The formula in our definition is motivated by \ref{L:vertical-asymp}.ii. Note that the definitions \ref{D:vertical-asymp} and \ref{D:horizontal-asymp} are consistent with Definition 3.7 and 3.8 in \cite{Kapouleas97}.

\begin{definition}
\label{D:vertical-asymp}
We define the smooth maps 
$(X^\pm_{ver})_{\theta,\tau}:\R^2_+ \to \R^3$ by taking $(X^\pm_{ver})_{\theta,0}=\Xi_\theta \circ X^\pm_{ver}$ 
(recall \ref{D:core-wings} and \ref{D:unbalance}) and for $\tau \neq 0$,
\[ (X^\pm_{ver})_{\theta,\tau}(y,s):= (-\tau^{-1},0,0) + 
\tau^{-1} (\,\rhotilde_{\theta,\tau}(s)\, \cos \tau y, \rhotilde_{\theta,\tau}(s) \sin \tau y,\pm \tau (a+s)\cos \theta),\]
where 
$\rhotilde_{\theta,\tau}(s):= \rho_\theta(\tau (a+s))$ and  
$\rho_\theta(t):=\cosh t + \sin \theta \sinh t$. 
\end{definition}

Using \ref{L:cat-ODE} it is clear that the function $\rho_\theta$ actually generates a catenoid which meets the plane $\{z=0\}$ at an angle $\pi/2+\theta$.  

\begin{lemma}
\label{L:vertical-asymp}
(i). $(X^\pm_{ver})_{\theta,\tau}$ depends smoothly on the parameters $\tau \in \R$ and $\theta \in (-\pi/2,\pi/2)$.\\
(ii). For $\tau \neq 0$, 
$(X^\pm_{ver})_{\theta,\tau}$ is a conformal minimal immersion into a subset of the complete catenoid 
which meets the plane $\{z=0\}$ with angle $\pi/2+\theta$ along the circle $\Bscr_\tau(\{x=z=0\})$ 
through the origin centered at $(-\tau^{-1},0,0)$ and the conformal factor is $\rhotilde_{\theta,\tau}^2$.  
\\
(iii). 
Assuming \ref{Con:parameters} 
we have the following uniform estimates 
where $A$ and $g$ are the second fundamental form (recall \ref{D:second-fundamental}) 
and the induced metric of the immersion $(X^\pm_{ver})_{\theta,\tau}$ respectively 
and $\partial$ denotes the partial derivatives with respect to the standard coordinates of $\R^2_+$ 
(recall \ref{E:sub-super} and \ref{D:weighted-Holder} and note that $g_0$ can be replaced by $g$):\\
$\phantom{kk}$ 
(a) $\| \rhotilde^2_{\theta,\tau}-1 : C^3((\R^2_+)_{s \le 5 \delta_s |\tau|^{-1}},g_0,a+s)\| \leq C |\tau|$,\\
$\phantom{kk}$ 
(b) 
$\| \partial^2 \, (X^\pm_{ver})_{\theta,\tau} :C^3(\, (\R^2_+)_{s \leq 5 \delta_s |\tau|^{-1}} , g_0)\| \leq C |\tau| ,$\\ 
$\phantom{kk}$ 
(c) 
$\| |A|_g^2:C^3(\, (\R^2_+)_{s \leq 5 \delta_s |\tau|^{-1}} , g_0)\| \leq C \tau^2 ,$\\
$\phantom{kk}$ 
(d) $\| g-g_0:C^{3}(\, (\R^2_+)_{s \leq 5 \delta_s |\tau|^{-1}} , g_0)\| \leq C \delta_s.$  
\\
(iv). When $\tau^{-1}=m \in \N$, the maps $(X^\pm_{ver})'_{\theta,\tau}(y,s):=(X^\pm_{ver})_{\theta,\tau}(y,s)+(\tau^{-1},0,0)$ satisfy the symmetries (recall \ref{D:Gm}):
\[ (X^\pm_{ver})'_{\theta,\tau}(s,-y)= \YYbar' \circ (X^\pm_{ver})'_{\theta,\tau}(y,s), \quad (X^\pm_{ver})'_{\theta,\tau}(s, 2\pi -y) =\YYbar'_{\frac{\pi}{m}} \circ (X^\pm_{ver})'_{\theta,\tau}(y,s),\]
\[ (X^\pm_{ver})'_{\theta,\tau}(s,\pi-y)=\YYhbar'_{\frac{\pi}{2m}} \circ (X^\mp_{ver})'_{\theta,\tau}(y,s).\]
\end{lemma}
\begin{proof}
(i) and (ii) follow easily by calculations. 
To prove the estimates in (iii), 
first of all we note that if $\ds>0$ is sufficiently small in absolute terms, 
then we have $|\cosh t| < C$ and $|\sinh t| < C |t|$ for all $t \in (-10\ds,10\ds)$, which imply that 
\[ \| \rho_\theta^2 -1 : C^3((-10\ds,10\ds),dt^2,|t|)\| \leq C.\]
If $|\tau|$ is small enough in terms of $a$ and $\ds$, then we have $|\tau a| < 5 \ds$. 
Therefore, $|\tau (a+s)| < 10 \ds$ when $s \leq 5 \ds |\tau^{-1}|$ and (a) follows by the definition of $\rhotilde_{\theta,\tau}$ and scaling. 
(d) follows from (a) and (ii). 
For (b-c), again we observe that if $|\theta|$ is small in absolute terms, 
then these are valid if we substitute $\tau=1$. 
By scaling we conclude their proof. 
(iv) follows from the definitions.
\end{proof}

The situation for the horizontal wings is simpler since the horizontal asymptotic half planes are fixed pointwise by the map $\Xi_\theta$ 
and remain planar after the action of $\Bscr_\tau$. 
However the parametrization does get distorted during the process. 
Therefore, the graph of $\varphi_{wing}$ over the perturbed immersion still loses minimality. 
For this reason, we have to understand the perturbation on the immersions of the horizontal asymptotic half planes as well.

\begin{definition}
\label{D:horizontal-asymp}
We define the smooth maps $(X^\pm_{hor})_{\tau}:\R^2_+ \to \R^3$ by taking $(X^\pm_{hor})_0=X^\pm_{hor}$ (recall \ref{D:core-wings}) and for $\tau \neq 0$,
\[ (X^\pm_{hor})_\tau (y,s):= (-\tau^{-1},0,0)+ \tau^{-1} \rhotilde^\pm(s)\, (\cos \tau y,\sin \tau y,0), \]  
where 
$\rhotilde^\pm_\tau(s):= \rho^\pm(\tau (a+s)) = e^{\pm \tau (a+s)} .$ 
and 
$\rho^\pm(t):=e^{\pm t}$. 
\end{definition}

\begin{lemma}
\label{L:horizontal-asymp}
(i). $(X^\pm_{hor})_\tau$ depends smoothly on the parameter $\tau \in \R$.\\
(ii). For $\tau \neq 0$, $(X^+_{hor})_\tau$ ($(X^-_{hor})_\tau$) is a conformal minimal immersion onto 
the exterior (punctured at the origin interior) 
of the circle of radius $\tau^{-1}$ centered at $(-\tau^{-1},0,0)$  
in the plane $\{z=0\}$.  
The conformal factor is $(\rhotilde^\pm_\tau)^2$.  
\\
(iii). 
Assuming \ref{Con:parameters} 
we have the following uniform estimates 
where $g$ is the induced metric of the immersion $(X^\pm_{hor})_\tau$ 
and $\partial$ denotes the partial derivatives with respect to the standard coordinates of $\R^2_+$ 
(recall \ref{E:sub-super} and \ref{D:weighted-Holder} and note that $g_0$ can be replaced by $g$):\\
$\phantom{kk}$ 
(a). $\| (\rhotilde^\pm_\tau)^2-1 : C^3((\R^2_+)_{s \leq 5 \delta_s |\tau|^{-1}},g_0,a+s)\| \leq C |\tau|$,\\
$\phantom{kk}$ 
(b). $\| g-g_0:C^{3}((\R^2_+)_{s \leq 5 \delta_s |\tau|^{-1}},g_0)\| \leq C \delta_s,$  
\\ 
$\phantom{kk}$ 
(c). 
$\| \partial^2 \, (X^\pm_{hor})_{\tau} :C^3(\, (\R^2_+)_{s \leq 5 \delta_s |\tau|^{-1}} , g_0)\| \leq C |\tau| .$
\\
(iv). When $\tau^{-1}=m \in \N$, 
the maps $(X^\pm_{hor})'_{\tau}(y,s):=(X^\pm_{hor})_{\tau}(y,s)+(\tau^{-1},0,0)$ 
satisfy the symmetries (recall \ref{D:Gm}):
\[ (X^\pm_{hor})'_{\tau}(s,-y)= \YYbar' \circ (X^\pm_{hor})'_{\tau}(y,s), \quad (X^\pm_{hor})'_{\tau}(s, 2\pi -y) =\YYbar'_{\frac{\pi}{m}} \circ (X^\pm_{hor})'_{\tau}(y,s),\]
\[ (X^\pm_{hor})'_{\tau}(s,\pi-y)=\YYhbar'_{\frac{\pi}{2m}} \circ (X^\pm_{hor})'_{\tau}(y,s).\]
\end{lemma}

\begin{proof}
The proofs are similar to the ones in \ref{L:vertical-asymp}.
\end{proof}

We are ready to define now the family of desingularizing surfaces (recall \ref{E:glue}, \ref{d:imm} and \ref{D:core-wings}).
Note that we truncate the function $\varphi_{wing}$ 
before we use it to build the graphs over the perturbed immersions defined in \ref{D:vertical-asymp} and \ref{D:horizontal-asymp}. 
This is necessary so that the desingularizing surface 
(after translating and rescaling to fit the singularity circle $\Ctheta$ in $\Wcal_\theta$) 
glues back to the rest of $\Wcal_\theta$ to form a  smooth surface.

\begin{definition}[Desingularizing surfaces]
\label{D:desing}
\label{N:truncation}
We define 
$\Sch_{\theta,\tau} :=   \Zscr_{\theta,\tau} \,(\, \Sch_{s \leq 5 \delta_s |\tau|^{-1}} \,)\,$, 
where 
the map $\Zscr_{\theta,\tau}:\Sch \to \R^3$ is defined by 
\[ \Zscr_{\theta,\tau}:=\left\{ \begin{array}{cl}
\Bscr_\tau \circ \Xi_\theta & \text{ on } \Sch_{core},\\
\Psi[0,1;s](\Bscr_\tau \circ \Xi_\theta,\Immer[(X^\pm_{hor})_{\tau},\psi_{trun} \varphi_{wing};\R^2_+] \circ (X^\pm_{hor})^{-1}) & \text{ on } \Sch^\pm_{hwing},\\
\Psi[0,1;s](\Bscr_\tau \circ \Xi_\theta,\Immer[(X^\pm_{ver})_{\theta,\tau},\psi_{trun} \varphi_{wing};\R^2_+] \circ (X^\pm_{ver})^{-1})  & \text{ on } \Sch^\pm_{vwing}, 
\end{array} \right. \]
where $\psi_{trun} :\R^2_+\to[0,1]$ is defined by 
$\psi_{trun} :=\psi_{cut}[4 \ds |\tau|^{-1},3\ds |\tau|^{-1}](s)$ when $\tau\ne0$ and simply by 
$\psi_{trun} \equiv 1$ when $\tau=0$. 
\end{definition}

\begin{lemma}
\label{L:desing}
$\Zscr_{\theta,\tau}$ is a family of smooth immersions depending smoothly on the parameters 
$\tau \in \R$ and $\theta \in (-\pi/2,\pi/2)$ with $\Zscr_{\theta,0}=\Xi_\theta|_{\Sch}$. 
Moreover (assuming \ref{Con:parameters}) we have the uniform estimates 
\begin{equation} 
\label{E:gS-est}
\begin{alignedat}{2} 
\|\, \Zscr^*_{\theta,\tau}  g- g_\Sch \, :\,  C^4(\,  \Sch_{s \leq 5 \delta_s |\tau|^{-1}} \, , g_\Sch \, ) \, \| 
\,& \le &\, 
C \, (\,\dtheta+\ds), 
\\ 
\|\, |A|^2\circ \Zscr_{\theta,\tau} - |A|^2_\Sch \, : \,  C^3(\,  \Sch_{s \leq 5 \delta_s |\tau|^{-1}} \, , g_\Sch \, ) \, \| 
\,& \le &\, 
C \, (\,\dtheta+\ds), 
\end{alignedat} 
\end{equation} 
where $\Zscr^*_{\theta,\tau}  g$ and  $|A|^2\circ \Zscr_{\theta,\tau}$ are the pullbacks by $\Zscr_{\theta,\tau}$  
of the induced metric and the squared length of the second fundamental form of $\Sch_{\theta,\tau}$, 
and $g_\Sch$ and $ |A|^2_\Sch $ are the induced metric and the squared length of the second fundamental form of $\Sch$. 
\end{lemma}

\begin{proof}
The first part of the lemma follows easily from \ref{L:unbalance}.i, \ref{L:bending}.i, 
\ref{L:vertical-asymp}.i and \ref{L:horizontal-asymp}.i. 
The estimates on $\Sch_{s\le5}$ follow by smooth dependence on a fixed compact set with 
$C\,(|\tau|+|\theta|)\le C\,\dtheta$ in the right hand side 
(recall \ref{Con:parameters}). 
The estimates on the remaining region follow from 
\ref{L:vertical-asymp}.iii and \ref{L:horizontal-asymp}.iii using 
\ref{L:phi-wing}, \ref{E:metric-est} and \ref{E:normal-est}. 
\end{proof}

\subsection*{Mean curvature of desingularizing surfaces}

In this subsection we estimate the mean curvature of the (immersed) desingularizing surface $\Sch_{\theta,\tau}$. 
Each of the maps $\Xi_\theta$ and $\Bscr_\tau$ introduces some mean curvature 
and there are also some non-zero mean curvature in the transition regions connecting different regions.
We first consider the mean curvature caused by the unbalancing map $\Xi_\theta$, 
which will also serve the purpose of our (extended) substitute kernel later: 

\begin{definition}[Substitute kernel]
\label{D:sub-kernel}
Let $H_\theta$ be the mean curvature of the immersed surface $\Xi_\theta(\Sch)$ pulled back to a function on $\Sch$. We define $w:\Sch \to \R$ to be
\[ w:= \left. \frac{d}{d\theta} \right|_{\theta =0} H_\theta. \]
\end{definition}

The function $w$ above gives the linearization of the mean curvature of $\Sch_{\theta,\tau}$ in the $\theta$-direction at $\theta=\tau=0$. 
By smooth dependence of parameters, it is easy to get uniform estimates in a fixed compact subset (modulo symmetries). 
To get uniform estimates on the wings, which converge to an unbounded set, 
we use the exponential decay provided by \ref{L:phi-wing}. 
To accommodate the truncation error 
created by the cutoff function $\psi_{trun}$ (recall \ref{N:truncation}) 
we only establish slower decay like $e^{-\gamma s}$.

\begin{prop}
\label{L:desing-H-estimate}
Let $H_{\theta,\tau}$ be the mean curvature of the (immersed) desingularizing surface $\Sch_{\theta,\tau}$ defined in \ref{D:desing} pulled back as a function on $\Sch$. 
Assuming \ref{Con:gamma} we have the following uniform estimates:
\[ \| H_{\theta,\tau} - \theta w : C^{0,\beta}(\Sch_{s \leq 5 \ds |\tau|^{-1}},g_\Sch, e^{-\gamma s})\| \leq C (|\tau|+|\theta|^2),\]
where $g_\Sch$ is the induced metric on $\Sch$ as a surface in $\R^3$.
\end{prop}

\begin{proof}
Since the quotient $\Sch_{s \leq 5}/\group_\Sch^0$ is a fixed compact subset (recall \ref{D:GSch}), by \ref{L:desing} we have the required estimate on $\Sch_{s \leq 5}$ by Taylor expansion near $\tau=\theta=0$. It remains to prove the estimate on the wings of $\Sch$. By \ref{L:unbalance}.ii and \ref{D:sub-kernel}, $w$ is supported inside $\Sch_{core}$, so it suffices to prove the estimate without the $w$-term on the wings $\Sch^\pm_{hwing}$ and $\Sch^\pm_{vwing}$ (recall \ref{D:core-wings}). 

First we note that the complete Scherk surface $\Sch$ has injectivity radius greater than $1/10$. Let $B_p$ be the geodesic ball with radius $1/100$ in $(\Sch,g_\Sch)$ centered at some $p \in \Sch_{1 \leq s \leq 5 \ds |\tau|^{-1}}$. By \ref{D:weighted-Holder} it suffices to prove the following estimate
\[  e^{\gamma s} \|H_{\theta,\tau} :C^{0,\beta}(B_p, g_\Sch)\| \leq C(|\tau| + |\theta|^2).\]
Recall that by \ref{D:desing} the wings of $\Sch_{\theta,\tau}$ can be expressed as the graph of the function $\psi_{trun} \varphi_{wing}$ over its asymptotic catenoids or planes given by the minimal immersions $(X^\pm_{ver})_{\theta,\tau}$ and $(X^\pm_{hor})_\tau$ (recall \ref{L:vertical-asymp} and \ref{L:horizontal-asymp}). We will divide the proof into three cases: $s \in [0, 3 \ds |\tau|^{-1}]$, $s\in [3  \ds |\tau|^{-1},4 \ds |\tau|^{-1}]$ and $s \in [4  \ds |\tau|^{-1},5\ds |\tau|^{-1}]$. Note that if $\tau=0$, then $H_{\theta,0}$ is supported in $\Sch_{core}$ by \ref{L:unbalance}.ii so the estimate holds trivially in this case. We will assume from now on that $\tau \neq 0$.

By \ref{N:truncation}, $\Sch_{\theta,\tau}$ is the graph of $\psi_{trun} \varphi_{wing}$ over the minimal immersions $(X^\pm_{ver})_{\theta,\tau}$ and $(X^\pm_{hor})_\tau$. Hence it is minimal on $s \in [4  \ds |\tau|^{-1},5 \ds |\tau|^{-1}]$ (recall \ref{L:vertical-asymp}.ii and \ref{L:horizontal-asymp}.ii). Let $X:B^2(2) \to \R^3$ be the restriction of $(X^\pm_{ver})_{\theta,\tau}$ or $(X^\pm_{hor})_\tau$ to any disk $B^2(2)$ of radius $2$ contained in $((\R^2_+)_{s \leq 6 \ds |\tau|^{-1}},g_0)$. By \ref{L:vertical-asymp}.iii.c and \ref{L:horizontal-asymp}.iii.b, $X$ satisfies \ref{E:C-bounded-a} for some universal constant $c_1>0$. 
On the other hand, by \ref{E:norm-multi} and \ref{L:phi-wing}
\[ \| \psi_{trun} \varphi_{wing}:C^{2,\beta}(\R^2_+,g_0,e^{-s})\| \leq C e^{-a} .  \]
Thus the function $\psi_{trun} \varphi_{wing}$ would have $C^{2,\beta}$-norm less than $\epsilon_H(c_1)$ in \ref{P:local-estimate-H} 
if $a$ is chosen sufficiently large in absolute terms. 
Then \ref{P:local-estimate-H} gives (also using that the various metrics are uniformly equivalent by \ref{L:vertical-asymp}.iii.d and \ref{L:horizontal-asymp}.iii.b)
\[ e^{\gamma s} \|H_{\theta,\tau}  :C^{0,\beta}(B_p, g_\Sch)\| \leq C e^{-(1-\gamma)s} \leq C e^{-3(1-\gamma)\ds |\tau|^{-1} }\leq C |\tau|,\]
as long as $\tau$ is sufficiently small in terms of $\ds$ and $\gamma$ (recall \ref{Con:gamma}). 

It remains the case where $s \in [0, 3 \ds |\tau|^{-1}]$. We will need to use the strengthened estimate in \ref{P:local-estimate-H}. 
Let $X':B^2(2) \to \R^3$ be the affine linear map which is the linearization of $(X^\pm_{ver})_{\theta,0}$ or $X^\pm_{hor}$ at the center of $B^2(2)$. Obviously $X'$ also satisfies \ref{E:C-bounded-a} for the same $c_1>0$ and that $X'$ agrees with $X$ up to first order at the center of $D$. 
By taking $a$ in \ref{L:phi-wing} sufficiently large in terms of $\epsilon_H(c_1)$, 
\ref{P:local-estimate-H} can be applied to the graphs of $\varphi_{wing}$ over $X$ and $X'$. 
By \ref{D:desing} the graph of $\varphi_{wing}$ over $X$ lies inside $\Sch_{\theta,m}$ whose mean curvature can be given as in \ref{P:local-estimate-H} by \ref{D:linear-operators}, \ref{D:second-fundamental}, \ref{L:vertical-asymp}.ii, \ref{L:horizontal-asymp}.ii
\[ H_{\theta,\tau} = \rho^{-2} \Delta_{g_0} \varphi_{wing} + |A|^2 \varphi_{wing} + Q_{X,\varphi_{wing}}, \]
where $H_{\theta,\tau}$ is the mean curvature of $\Sch_{\theta,\tau}$ pulled back to a function on $B^2(2)$, $\rho=\tilde{\rho}_{\theta,\tau}$ or $\tilde{\rho}_\tau^\pm$ as in \ref{L:vertical-asymp}.ii and \ref{L:horizontal-asymp}.ii, and $|A|^2$ is the norm-squared second fundamental form of $X$ as defined in \ref{D:second-fundamental} and \ref{D:linear-operators}. On the other hand, the graph of $\varphi_{wing}$ over $X'$ lies inside $\Sch$ (up to a translation and rescaling in $\R^3$) and hence is minimal. Therefore, by \ref{P:local-estimate-H} we have similarly
\[ 0=\rho^{-2}(p_0) \Delta_{g_0} \varphi_{wing} +Q_{X',\varphi_{wing}}, \]
where $p_0$ is the center of the disk $B^2(2)$. Combining these two expressions and using \ref{L:vertical-asymp}.iii, \ref{L:horizontal-asymp}.iii, \ref{E:local-estimate-H} and \ref{L:phi-wing}
\[  e^s \|H_{\theta,\tau}  :C^{0,\beta}(B_p, g_\Sch )\| \leq C |\tau|,\]
where we have used that $\|\partial^2 X -\partial^2 X':C^{1,\beta}(B^2(2),g_0)\| \leq C |\tau|$. This proves the proposition.
\end{proof}


\subsection*{The initial surfaces}  
In this subsection, we construct for each large $m\in\N$ and small $\theta$ as in 
\ref{Con:delta},  
an initial surface $M_{\theta,m}$ which depends smoothly on $\theta$. 
In the proof of the main theorem \ref{T:finalthm} we will show that for each $m$ sufficiently large, 
we can use a fixed point argument to find $\theta^*$ (depending on $m$) 
such that there exists a function $\varphi^*$ whose twisted graph over the initial surface 
$M_{\theta^*,m}$ is a minimal surface which intersects $\Sph^2$ orthogonally. 

$M_{\theta,m}$ is constructed by desingularizing $\Wcal_\theta$ 
using surfaces $\Sigma_{\theta,m}$ obtained by shrinking and translating 
the desingularizing surfaces $\Sch_{\theta,1/m}$ defined in \ref{D:desing}. 
The scaling and translation have to be chosen carefully so that the ``axis'' of $\Sigma_{\theta,m}$ 
matches the circle of singularity $\Ctheta$ in the initial configuration $\Wcal_\theta$ (recall \ref{D:initial-config}): 

\begin{definition}[Scaled desingularizing surfaces] 
\label{D:homothety} 
We define (recall \ref{D:desing}) 
$$
\Sigma_{\theta,m} := \Hcal_{\theta,m} (\Sch_{\theta,1/m}) = 
\Zcal_{\theta,m} \,(\, \Sch_{s \leq 5 \delta_s m} \,)\,, 
$$  
where 
$\theta,m$ are as in \ref{Con:initial}, 
$\Zcal_{\theta,m} := \Hcal_{\theta,m} \circ \Zscr_{\theta,1/m} : \Sch \to \R^3$,  
and 
$\Hcal_{\theta,m} : \R^3 \to \R^3$ 
is the affine homothety defined by
\begin{equation}
\label{N:scale} 
\Hcal_{\theta,m}(x,y,z):= \, \lambda\,(x+m,y,z),  
\qquad \text{ where } 
\lambda:=m^{-1} r_\theta. 
\end{equation} 
\end{definition}

Note that by \ref{L:cat-ODE} $m \lambda=r_\theta<1$ is uniformly bounded away from $0$ 
and by \ref{D:GSch}, \ref{D:Gm}, \ref{L:unbalance}.iii, \ref{L:bending}.iii, \ref{L:Scherk-equivariant}, 
\ref{L:vertical-asymp}.iv and \ref{L:horizontal-asymp}.iv,   
\begin{equation} 
\label{E:equivariant} 
\Zcal_{\theta,m} \circ \YYbar= \YYbar' \circ \Zcal_{\theta,m}, \quad 
\Zcal_{\theta,m} \circ \YYbar_{\pi}= \YYbar'_{\frac{\pi}{m}} \circ \Zcal_{\theta,m}, \quad 
\Zcal_{\theta,m} \circ \YYhbar_{\frac{\pi}{2}} = \YYhbar'_{\frac{\pi}{2m}} \circ \Zcal_{\theta,m}, 
\end{equation} 
and therefore $\Zcal_{\theta,m}$ is equivariant with respect to $\group_\Sch$ and $\group_m$.  
Moreover 
$\Zcal_{\theta,m}$ 
maps the axis $\{x=z=0\}$ of the Scherk to $\Ctheta$ (recall \ref{D:initial-config}).  
This then implies that the four connected components of $\partial \Sigma_{\theta,m} $ 
have neighborhoods in $\Sigma_{\theta,m}$ 
which are actually contained in $\Wcal_\theta$  
(recall \ref{L:vertical-asymp}.ii, \ref{L:horizontal-asymp}.ii, \ref{D:desing}, \ref{D:initial-config}, and \ref{Con:delta}). 
We conclude that 
$\Wcaltilde_\theta \setminus \partial \Sigma_{\theta,m} $ consists of five connected 
components four of which are disjoint from the interior of $\Sigma_{\theta,m}$ and 
can be used to smoothly extend $\Sigma_{\theta,m} $: 

\begin{definition}[Initial surfaces]
\label{D:initial}
We  define $\Sigmatilde_{\theta,m}$ to be the union of 
$\Sigma_{\theta,m}$ and the four connected components 
of $\Wcaltilde_\theta\setminus \partial \Sigma_{\theta,m} $ 
which are disjoint from the interior of $\Sigma_{\theta,m}$. 
We define then the initial surfaces $M_{\theta,m}$ as
\[ M_{\theta,m}:= \Sigmatilde_{\theta,m} \cap \B^3 \subset \Sigmatilde_{\theta,m} . \] 
Note that for simplicity in subscripts we may write $M$ instead of $M_{\theta,m} $.  
For future reference 
we fix a continuous function $s$ on $\Sigmatilde_{\theta,m}$ which is $\group_m$-invariant (in the sense of \ref{D:sym-function})
and agrees with the pushforward by $\Hcal_{\theta,m}$ of $s$ on 
$ \Sch_{s \leq 5 \delta_s m} $  
and takes values in $(5 \delta_s m,6 \delta_s m]$ on $\Sigmatilde_{\theta,m} \setminus \Sigma_{\theta,m}$.  
\end{definition}

\begin{lemma}
\label{P:initial}
For $m$ large enough 
the initial surfaces $M_{\theta,m}$ are smooth, embedded, $\group_m$-invariant, 
compact oriented surfaces in $\B^3$, with genus $m-1$. 
They meet $\Sph^2$ orthogonally along their boundary which 
consists of three connected components and satisfies 
$\partial M_{\theta,m}= \Sph^2\cap M_{\theta,m} $.   
Moreover, 
as $m \to \infty$, 
the surfaces $M_{\theta,m}$ converge in the Hausdorff sense to $\Wcal_\theta$ 
and the convergence is smooth away from $\Ctheta$.
\end{lemma}

\begin{proof}
This follows from \ref{E:equivariant},  
that the function $\psi_{trun}$ is independent of $y$,  
and the preceding discussion. 
\end{proof}


\section{Solving the linearized equations}
\label{S:linear}

In Section \ref{S:initial}, we have constructed our initial surfaces $M_{\theta,m}$ which are free boundary surfaces but not minimal.  
In the next section we will estimate the nonlinear terms and then prove the main theorem. 

\subsection*{The linearized free boundary minimal surface equation}
We first discuss how the symmetries imposed apply to the functions we use to appropriately correct the initial surfaces.

\begin{definition}
\label{D:sym-function}
Let $f:\Omega \to \R$ be a function defined on a $\group_\Sch$-invariant subset $\Omega \subset \Sch$. We say that $f$ is $\group_\Sch$-invariant if it satisfies (recall \ref{D:GSch})
\[ f \circ \YYbar = f, \quad f \circ \YYbar_{\pi} = f, \quad f \circ \YYhbar_{\frac{\pi}{2}} = f.\]
We say that $f$ is $\group_\Sch$-symmetric if it satisfies
\[ f \circ \YYbar = f, \quad f \circ \YYbar_{\pi} = f, \quad f \circ \YYhbar_{\frac{\pi}{2}} = -f.\]
We define similarly for the group $\group_m$ (recall \ref{D:Gm}) if $f:\Omega' \to \R$ is a function defined on a $\group_m$-invariant subset $\Omega' \subset \R^3$ with $\YYbar$, $\YYbar_{\pi}$ and $\YYhbar_{\frac{\pi}{2}}$ replaced by $\YYbar'$, $\YYbar'_{\frac{\pi}{m}}$ and $\YYhbar'_{\frac{\pi}{2m}}$ respectively.
\end{definition}

\begin{notation}
\label{N:sym}
We use the subscript ``sym'' for subspaces of function spaces which are $\group_\Sch$-symmetric or $\group_m$-symmetric.
\qed 
\end{notation}

\begin{lemma} 
\label{L:sym} 
(i). 
The graph of a $\group_\Sch$-symmetric (or $\group_m$-symmetric) function over $\Sch$ 
(or $M_{\theta,m}$) is $\group_\Sch$-invariant (or $\group_m$-invariant). 
\\
(ii). 
The mean curvature of a graph as in (i) is 
$\group_\Sch$-symmetric (or $\group_m$-symmetric). 
\\
(iii). 
The product of a symmetric function with an invariant function is symmetric. 
\\
(iv).
The function $w$ defined in \ref{D:sub-kernel} is $\group_\Sch$-symmetric and supported inside $\Sch_{core}$.
\end{lemma} 

\begin{proof}
(i) and (ii) follow from the observation that the Gauss map satisfies 
\[ \nu \circ \YYbar = \YYbar \circ \nu,  \qquad \nu \circ \YYbar_{\pi} = \YYbar_{\pi} \circ \nu, \qquad \nu \circ \YYhbar_{\frac{\pi}{2}} = - \YYhbar_{\frac{\pi}{2}} \circ \nu\]
and similarly for the other three isometries. 
(Equivalently all isometries in consideration reserve the orientation of the surface involved but only the first two 
of each group reverse the orientation of the ambient $\R^3$.) 
(iii) follows from the definitions and (iv) follows from \ref{L:unbalance}.
\end{proof} 

The linearized operators to the free boundary minimal surface equation $H=0$ and $\Theta=0$ 
is given below (recall \ref{P:local-estimate-Theta}, \ref{P:local-estimate-H} and \ref{R:H-linearized}).

\begin{definition}[Jacobi operators]
\label{D:linear-operator}
Let $S \subset \mathbb{B}^3$ be a smooth surface
with each of its boundary components either contained in $\Sph^2$ or completely disjoint from $\Sph^2$,  
and let $\Delta$ denote the intrinsic Laplace operator, 
$|A|^2$ the norm-squared of the second fundamental form with respect to the induced metric on $S \subset \R^3$, 
and $\eta$ the outward unit conormal of $\partial S$ with respect to $S$. 
We define the Jacobi operator $\Lcal=\Lcal_S: C^{2,\beta}(S) \to C^{0,\beta}(S)$ and the boundary Jacobi operator 
$\Bcal: C^{2,\beta}(S) \to C^{1,\beta}(\partial S\cap\Sph^2)$ by
\[ \Lcal v : = \Delta v + |A|^2 v \qquad \text{ and } \qquad  \Bcal v := -\frac{\partial v}{\partial \eta} +v. \]
\end{definition}

Given inhomogeneous data $(E,E^\partial) \in C^{2,\beta}(S) \times C^{1,\beta}(\partial S\cap\Sph^2)$ 
(with $S$ as in \ref{D:linear-operator}),   
we need to solve on $S$ the linearized free boundary minimal surface equation 
\begin{equation} 
\label{E:LE}
\Lcal  v = E  \text{ on } S, \qquad 
\Bcal  v = E^\partial  \text{ along } \partial S \cap \Sph^2, \qquad
v=0  \text{ along } \partial S\setminus \Sph^2. 
\end{equation} 

\subsection*{Solving the linearized equation on $\Wcal_\theta$} 

The main proposition \ref{P:linear-global} of this section shows that, 
modulo a one-dimensional cokernel (and suitable choice of parameters), 
the linearized equation \ref{E:LE} is solvable with estimates 
when $S=M_{\theta,m}$ defined as in \ref{D:initial}. 
(Note that $\partial S \subset   \Sph^2$ in this case.) 
This is achieved by combining approximate semi-local solutions 
on regions $\Mtilde_{\theta,m}[0]$ and $\Mtilde_{\theta,m}[1]$ inside $M$ 
and then iterating. 
The various regions will be defined in \ref{D:M-regions} and the semi-local solutions 
are obtained by solving on the model surfaces $\Sch$ and $\Wcal_\theta$ and then transferring 
to the corresponding regions of $M$ by using $\PiSch$ and $\PiWtheta$ which will be defined in \ref{D:Pi}.  
Recall now from \ref{D:initial-config} that the initial configuration $\Wcal_\theta$ is the union of $\K^+_\theta$, $\K^-_\theta$, $\A_\theta$ and $\D_\theta$. 
In the following lemma \ref{L:linear-ends}
we show that we can always solve the linearized equation \ref{E:LE} on 
$\Wcal_\theta$ due to the non-existence of kernels on each of the four pieces. 

\begin{definition}[H\"{o}lder norms on $\Wcal_\theta$]  
\label{D:normWcal}
Let $C^{k,\beta}(\Wcal_\theta)$ be the space of functions $u'$ on $\Wcal_\theta$ which have restrictions 
$\left. u'\right|_S\in C^{k,\beta}(S)$ for each 
$S=\K^+_\theta,\K^-_\theta,\A_\theta,\D_\theta$.  
(Note that specifying such a function is equivalent to specifying functions on 
$\K^+_\theta$, $\K^-_\theta$, $\A_\theta$ and $\D_\theta$ which agree on $\Circle_\theta$.)   
For $u'\in C^{k,\beta}(\Wcal_\theta)$ we define its norm 
$ \| u':C^{k,\beta}(\Wcal_\theta)\| :=\max_{S=\K^+_\theta,\K^-_\theta,\A_\theta,\D_\theta} \left\| \left. u'\right|_S:C^{k,\beta}(S)\right\|.$ 
For $u'\in C^{2,\beta}(\Wcal_\theta)$   
such that the restriction of $\Lcal_S\left(\left.u'\right|_S\right)$ to $\Ctheta$ is independent of 
$S=\K^+_\theta,\K^-_\theta,\A_\theta,\D_\theta$, 
we define  
$\Lcal_{\Wcal_\theta}u' \in C^{0,\beta}(\Wcal_\theta)$ 
by 
$\left.\left(\Lcal_{\Wcal_\theta}u'\right)\right|_S:= \Lcal_S\left(\left.u'\right|_S\right)$ 
for each 
$S=\K^+_\theta,\K^-_\theta,\A_\theta,\D_\theta$. 
\end{definition}

\begin{lemma}[Linear estimates on $\Wcal_\theta$]
\label{L:linear-ends}
If $m$ is sufficiently large in absolute terms, 
then there is a bounded linear map 
$\Rcal_{\Wcal_\theta} : C^{0,\beta}_{sym}(\Wcal_\theta) \times C^{1,\beta}_{sym}(\partial \Wcal_\theta) \to C^{2,\beta}_{sym}(\Wcal_\theta)$, 
such that given inhomogeneous data $E' \in C^{0,\beta}_{sym}(\Wcal_\theta)$ and $E'^\partial \in C^{1,\beta}_{sym}(\partial \Wcal_\theta)$, 
$u' =\Rcal_{\Wcal_\theta}(E',E'^\partial)$ restricts to the unique solution of the linearized equation 
\ref{E:LE} on each $S=\K^+_\theta, \K^-_\theta, \A_\theta, \D_\theta$, 
so that in the sense of \ref{D:normWcal} we have 
\[ \Lcal_{\Wcal_\theta} u'  = E'  \text{ on } \Wcal_\theta, \qquad 
\Bcal  u' = {E'}^\partial  \text{ on } \Wcal_\theta \cap \Sph^2, \qquad
u'=0  \text{ on } \Ctheta . \] 
Moreover, there is a universal constant $C>0$ such that 
\[ \|u':C^{2,\beta}(\Wcal_\theta)\| \leq \, C(\, \|E' : C^{0,\beta}(\Wcal_\theta)\| + \|E'^\partial:C^{1,\beta}(\Wcal_\theta \cap \Sph^2)\| \, ). \]
\end{lemma}

\begin{proof}
Because of the smooth dependence on $\theta$ and the smallness of $\theta$ 
by \ref{Con:gamma} we can consider the Jacobi operators on $\Wcal_\theta$ 
as small perturbations of the ones on $\Wcal_0$ 
(use \ref{L:K-norms} for the equivalence of the norms), 
and therefore it is enough to prove the lemma in the case $\theta=0$. 
By assuming $m$ sufficiently large and separating variables we can easily ensure that 
any kernel for $\Lcal$ with Robin and Dirichlet boundary conditions as usual  
will have to be rotationally invariant.
Using \ref{L:linear-disk}, \ref{L:linear-annulus} and \ref{L:linear-cat}, we conclude that such kernel is trivial. 
Standard elliptic estimates from \cite{Agmon-Douglis-Nirenberg59} (see Theorem 7.3, and Remark 2 on p.669) imply then the lemma. 
\end{proof}

\subsection*{Solving the linearized equation on $\mathcal{S}$}

We now solve the linearized equation \ref{E:LE} on the model Scherk surface $\Sch$ 
which is complete with no boundary, hence instead of boundary conditions we impose 
exponential decay along the wings. 

It is a standard fact that the Gauss map $\nu_\Sch$ of the Scherk surface restricts to an anti-conformal diffeomorphism 
from a fundamental region $\Sch \cap \{0 \leq y \leq \pi\}$ (with respect to the group $\group^0_\Sch$ - recall \ref{D:GSch}) 
onto the hemisphere minus four points $\mathbb{S}^2 \cap \{y \leq 0\} \setminus \{(\pm 1,0,0) \cup (0,0,\pm 1)\}$. 
Hence, the Gauss map pulls back the metric $g_{\Sph^2}$ to a conformally equivalent metric $h$ with its associated linear operator defined as:
\[ h:=\frac{1}{2}|A|^2  g_\Sch \qquad  \text{ and} \qquad \mathcal{L}_h:= \Delta_h +2,  \]
where $A$ is the second fundamental form of $\Sch \subset \R^3$ and $\Delta_h$ is the intrinsic Laplacian with respect to the conformal metric $h$. By conformal invariance of the Laplacian in dimension two, $\Lcal_h=2 |A|^{-2} \Lcal_\Sch$ and hence the operators $\Lcal_h$ and $\Lcal_\Sch$ have the same kernel. 

It is well known that any ambient Killing vector field restricts to a Jacobi field on the minimal surface. Using the translations in $\R^3$, for any unit vector $e$, the function $e \cdot \nu_\Sch$ lies in the kernel of $\Lcal_\Sch$, and thus in the kernel of $\Lcal_h$. (Note that there are other Jacobi fields arising from rotations and scalings in $\R^3$. However, they are either unbounded or not $\group_\Sch$-symmetric in the sense of \ref{D:sym-function}). 
The following lemma below says that modulo the symmetries $\group_\Sch$, there is only a one-dimensional kernel. 

\begin{lemma}
\label{L:Scherk-kernel}
The kernel of $\Lcal_h$ on $\Sch$ which is $\group_\Sch$-symmetric and bounded is spanned by the function $e_x \cdot \nu_\Sch$.
\end{lemma}

\begin{proof}
Since $\Sch$ has asymptotically planar ends, its Gauss map can be extended to a non-constant holomorphic map from a compact Riemann surface $\Sch^*$ with all the branching values, $(\pm 1,0,0)$ and $(0,0,\pm 1)$, lying on an equator of $\Sph^2$. Therefore, we can apply Theorem 20 in \cite{Montiel-Ros91} to conclude that the multiplicity of the eigenvalue $2$ for the operator $\Lcal_h$ on $\Sch^*$ is exactly equal to $3$, which are generated by $e_x \cdot \nu_\Sch$, $e_y \cdot \nu_\Sch$ and $e_z \cdot \nu_\Sch$, where $e_x,e_y,e_z$ are the standard coordinate basis in $\R^3$. Among these, only $e_x \cdot \nu_\Sch$ is $\group_\Sch$-symmetric.
\end{proof}

Because of the existence of the one-dimensional kernel in \ref{L:Scherk-kernel}, we can at best solve the linearized equation \ref{E:LE} modulo a one-dimensional co-kernel. The next lemma shows that the function $w$ we defined in \ref{D:sub-kernel} serves the purpose of a (extended) substitute kernel.

\begin{lemma}
\label{L:sub-kernel}
Let $L^2(\Sch/\group^0_\Sch,h)$ be the Hilbert space of $L^2$-integrable functions with respect to the metric $h$. Then, the functions $2 |A|^{-2} w$ and $e_x \cdot \nu_\Sch$ belong to $L^2(\Sch/\group^0_\Sch,h)$ and are not orthogonal to each other. Therefore, for any $\Ehat \in C^{0,\beta}_{sym}(\Sch,g_\Sch,e^{-\gamma s})$, there exists $\mu \in \R$ such that $|\mu| \leq C \|2 |A|^{-2} \Ehat: L^2(\Sch/\group^0_\Sch,h)\|$ and $2 |A|^{-2}(\Ehat-\mu w)$ is orthogonal to $e_x \cdot \nu_\Sch$ in $L^2(\Sch/\group^0_\Sch,h)$.
\end{lemma}

\begin{proof}
Since $|A|^2$ is $\group_\Sch$-invariant and by \ref{L:sym}.iv $w$ is $\group_\Sch$-symmetric 
with compact support (modulo $\group_\Sch^0$), $2 |A|^{-2} w \in L^2(\Sch/\group^0_\Sch,h)$. 
From \ref{L:Scherk-kernel} and that $\Sch/\group^0_\Sch$ has finite $h$-area, we have $e_x \cdot \nu_\Sch \in L^2(\Sch/\group^0_\Sch,h)$ as well. Recall that $w$ is supported on the core $\Sch_{core}$. Using the balancing formula for the Killing field $e_x$, the $L^2(\Sch/\group^0_\Sch,h)$ product of the two functions is 
\[ \int_{\Sch_{core}/ \group^0_\Sch} e_x \cdot w \nu_\Sch \, dA 
= \left. \frac{d}{d\theta}\right|_{\theta=0} \int_{(\Sch_{\theta,0}/ \group^0_\Sch)_{s \leq 0}} H_\theta \cdot e_x \, dA 
= \left. \frac{d}{d\theta}\right|_{\theta=0}\int_{\partial (\Sch_{\theta,0}/ \group^0_\Sch)_{s \leq 0}} \eta_\theta  \cdot e_x \, ds, \]
where $dA$ and $ds$ are with respect to the original Scherk metric $g_\Sch$. 
This is non-zero as $ \left. \frac{d}{d\theta}\right|_{\theta=0} \eta_\theta \approx e_x$ 
on the vertical wings of $\Sch$ and vanishes on the horizontal wings. 
Here $H_\theta$ is the mean curvature vector of $\Sch_{\theta,0}$ and $\eta_\theta$ is the outward unit conormal of 
$\partial (\Sch_{\theta,0})_{s \leq 0}$ relative to $(\Sch_{\theta,0})_{s \leq 0}$.
\end{proof}

We now solve the linearized equation on $\Sch$ with appropriate decay. Note that we only solve modulo a one-dimensional space which corresponds to the 
(approximate) kernel of the operator: 

\begin{lemma}[Linear estimates on $\Sch$]
\label{L:linear-Scherk}
For any $\Ehat \in C^{0,\beta}_{sym}(\Sch,g_\Sch,e^{-\gamma s})$, there exists unique $\uhat \in C^{2,\beta}_{sym}(\Sch,g_\Sch,e^{-\gamma s})$ and $\mu \in \R$ such that 
\[ \Lcal_\Sch \uhat = \Ehat + \mu w \qquad \text{ on } \Sch,\]
where $w$ is the function defined in \ref{D:sub-kernel} and $g_\Sch$ is the induced metric of $\Sch \subset \R^3$ with Jacobi operator $\Lcal_\Sch$ (recall \ref{D:linear-operator}). Moreover, we have 
\[ \|\uhat:C^{2,\beta}(\Sch,g_\Sch,e^{-\gamma s})\| + |\mu| \leq C \|\Ehat:C^{0,\beta}(\Sch,g_\Sch,e^{-\gamma s})\| \]
for some universal constant $C>0$.
\end{lemma}

\begin{proof}
The uniqueness part follows \ref{L:sub-kernel} that if $\Ehat=0$ then we must have $\mu=0$ and hence $\uhat=0$ if it is decaying at the rate $e^{-\gamma s}$ 
since the kernel $e_x \cdot \nu_\Sch$ does not decay along the vertical wings. 
By an argument of \cite[Lemma 7.2]{Kapouleas97}, 
we can assume without loss of generality that the inhomogeneous term $E$ is supported in $\Sch_{s \leq 2}$. 
Let $\Ehat \in C^{0,\beta}_{sym}(\Sch,g_\Sch,e^{-\gamma s})$ be given and supported in $\Sch_{s \leq 2}$. 
By \ref{L:sub-kernel}, there exists $\mu \in \R$ such that $2 |A|^{-2} (\Ehat-\mu w)$ is orthogonal to $e_x \cdot \nu_\Sch$ in $L^2(\Sch/\group^0_\Sch,h)$ and 
\beq
\label{E:mu-bound}
 |\mu| \leq C \|2 |A|^{-2} \Ehat: L^2(\Sch/\group^0_\Sch,h) \| \leq C \|\Ehat :C^{0,\beta}(\Sch,g_\Sch,e^{-\gamma s})\|,
\eeq
since $|A|^2$ is uniformly bounded away from zero on $\Sch_{s \leq 2}$ and the area grows linearly and hence dominated by the exponential decay. 
Using \ref{L:Scherk-kernel}, there exists a unique $\uhat \in L^2(\Sch/\group^0_\Sch,h)$ which is orthogonal to $e_x \cdot \nu_\Sch$ and
\[ \Lcal_h \uhat=2 |A|^{-2} (\Ehat-\mu w), \]
with $\|\uhat: L^2(\Sch/\group^0_\Sch,h)\| \leq C \|\Ehat :C^{0,\beta}(\Sch,g_\Sch,e^{-\gamma s})\|$. Therefore $\uhat$ solves the desired linearized equation
\[ \Lcal_\Sch  \uhat = \Ehat-\mu w.\]
Note that for any $c \in \R$, $\uhat+c (e_x \cdot \nu_\Sch)$ is also a solution to the same equation. Therefore, to get the required estimate, it suffices to prove that there exists some $c \in \R$ such that 
\[ \|\uhat +c (e_x \cdot \nu_\Sch): C^{2,\beta}(\Sch,g_\Sch,e^{-\gamma s})\| \leq C \|\Ehat: C^{0,\beta}(\Sch,g_\Sch,e^{-\gamma s})\|.\]

We now prove the existence of such a constant $c$. Since $\Sch_{s \leq 2}$ has bounded geometry, by de-Giorgi-Nash-Moser theory and Schauder estimates in standard linear PDE theory and \ref{E:mu-bound}, we have
\[ \|\uhat:C^{2,\beta}(\Sch_{s \leq 3},g_\Sch)\| \leq C  \|\Ehat:C^{0,\beta}(\Sch,g_\Sch,e^{-\gamma s})\|.\]
In particular, $\|\uhat:C^0(\partial \Sch_{s \leq 2}) \|\leq C  \|\Ehat: C^{0,\beta}(\Sch,g_\Sch,e^{-\gamma s})\|$. Since both $\uhat$ and $e_x \cdot \nu_\Sch$ are $\group_\Sch$-symmetric functions, there exists a unique $c \in \mathbb{R}$ such that $-c(e_x \cdot \nu_\Sch)$ matches the first harmonics of $\uhat$ on $\Sch_{s \geq 2}$ and hence $\uhat+c(e_x \cdot \nu_\Sch)$ would have the required decay. The required estimate then follows from $\|\uhat:C^0(\partial \Sch_{s \leq 2}) \|\leq C  \|\Ehat: C^{0,\beta}(\Sch,g_\Sch,e^{-\gamma s})\|$.
\end{proof}

\subsection*{Solving the linearized equation on $M_{\theta,m}$}

In this subsection we state and prove Proposition 
\ref{P:linear-global} where 
we solve with estimates the linearized equation \ref{E:LE} 
on an initial surface $M_{\theta,m}$ defined as in \ref{D:initial}. 
We first define various domains of the initial surfaces $M_{\theta,m}$, 
projections to the standard models, 
some cutoff functions, and the global norms we will need: 

\begin{definition}
\label{D:M-regions}
Assuming \ref{Con:delta} 
so that $\abar:=\log \lambda^{-7} < 5\ds m-1$ (recall \ref{N:scale}), 
we define the following regions of $M_{\theta,m}$:
\[ M_{\theta,m}[0] := M_{\theta,m}  \cap \{ s \leq 5 \ds m-1 \}, \qquad \Mtilde_{\theta,m}[0]:= M_{\theta,m}   \cap \{s \leq 5 \ds m\}=\Sigma_{\theta,m}, \]
\[ M_{\theta,m}[1] := M_{\theta,m}   \cap \{ s \geq \abar +1\}, \qquad \Mtilde_{\theta,m}[1] := M_{\theta,m}   \cap \{ s \geq \abar \}.\] 
\end{definition}

By \ref{E:equivariant} $\Zcal_{\theta,m}$ is an infinite covering map onto its image. 
The group of deck transformations is generated by the translation $\Tcal_m:\R^3\to\R^3$ 
defined by 
\begin{equation}
\label{E:translation}
\Tcal_m(x,y,z):=(x, y+2 \pi m , z). 
\end{equation}
Hence 
$\Zcal_{\theta,m}:\Sch\to \R^3$ factors through an embedding (diffeomorphism onto its image) 
\begin{equation} 
\label{E:Zcalhat}
\Zcalhat_{\theta,m}:\Schhat_m \to \R^3,  
\qquad \text{ where } 
\Schhat_m := \Sch\,/\,\Tcal_m 
\end{equation} 
is the quotient surface under the identifications induced 
by the group generated by $\Tcal_m$, 
and the embeddedness follows from \ref{P:initial} and \ref{L:horizontal-asymp}.iv.  
Note that in the following definition $\PiSch$ involves scaling,  
$\Mtilde_{\theta,m}[1]$ is the graph of the function $\lambda \psi_{trun} \varphi_{wing}$ 
transplanted to a subset of $\Wcal_\theta\setminus\Circle_\theta$, 
and $\PiWtheta$ is the identity map on a neighborhood of $\partial M_{\theta,m}  $: 

\begin{definition}
\label{D:Pi}
We define a smooth map (diffeomorphism onto its image) 
$\PiSch: \Mtilde_{\theta,m}[0] \to \Schhat_m$ 
as the restriction to $\Mtilde_{\theta,m}[0]$ 
of the inverse of $\Zcalhat_{\theta,m}$, 
considered as a diffeomorphism from $\Schhat_m$ onto its image and defined as in \ref{E:Zcalhat}. 

We also define a smooth map (diffeomorphism onto its image) 
$\PiWtheta : \Mtilde_{\theta,m}[1] \to \Wcal_\theta\setminus\Circle_\theta$ 
to be the nearest point projection from $\Mtilde_{\theta,m}[1]$ to $\Wcal_\theta \setminus\Circle_\theta$ 
(recall \ref{R:graph}). 
\end{definition}

\begin{definition}
\label{D:psi01}
We define the cutoff functions $\psihat,\psi' \in C^\infty(M_{\theta,m}  )$ by (recall \ref{E:cutab})
\[ \psihat=\psi_{cut}[\ell,\ell-1] \circ s, \qquad  \psi'=\psi_{cut}[\abar,\abar+1] \circ s ,\]
where $s$ is the function on $M_{\theta,m}  $ as defined in \ref{D:initial}.
\end{definition}

\begin{lemma}
\label{L:psi01}
(i). $\psihat$ is supported on $\Mtilde_{\theta,m}[0]$ with $\psihat \equiv 1$ on $M_{\theta,m}[0] \subset \Mtilde_{\theta,m}[0]$.
\\
(ii). $\psi'$ is supported on $\Mtilde_{\theta,m}[1]$ with $\psi' \equiv 1$ on $M_{\theta,m}[1] \subset \Mtilde_{\theta,m}[1]$.
\\
(iii). $\psihat$, $\psi'$ are $\group_m$-invariant functions on $M_{\theta,m}  $ (recall \ref{D:sym-function}). 
\\
(iv). $\|\psihat \circ \Zcal_{\theta,m}:C^3(\Sch,g_\Sch)\| \leq C$ (recall \ref{D:homothety}) and $\|\psi':C^3(M_{\theta,m}  ,\lambda^{-2} g)\| \leq C $. 
\end{lemma}

\begin{proof}
(i) and (ii) follow from \ref{E:cutab} and \ref{D:M-regions}, while (iii) holds since $s$ is a $\group_m$-invariant function on $M_{\theta,m}  $ by \ref{D:initial}. 
(iv) follows from the definitions. 
\end{proof}

\begin{definition}[Global weighted norms]
\label{D:global-norm}
For $\Omega$ a $\group_m$-invariant domain in 
$M_{\theta,m}\subset \Munder_{\theta,m}$ (defined as in \ref{D:initial}),  
$k=0,2$, and 
$u\in C^{k,\beta}_{sym}(\Omega)$,  
we define 
$$
\|u\|_{k,\beta,\gamma;\Omega} :=  
\lambda^{-k+1} \, 
\|u :C^{k,\beta}(\, \Omega \, ,\,  \lambda^{-2} g\, ,\, f_k \,)\, \|, 
$$
where $f_k:=\max(e^{-\gamma s},b_k)$,  
$b_0:=e^{-5 \gamma \ds m}$,  
$b_2:=\lambda^{-6} b_0$, 
$\lambda$ is as in \ref{D:homothety}, 
and the weighted norm is as in 
\ref{D:weighted-Holder}.   
For $(E,E^\partial)\in C^{0,\beta}_{sym}(\Omega) \times C^{1,\beta}_{sym}(\partial M_{\theta,m}  \cap\Omega)$ 
we also define 
$$ 
\|\,(E,E^\partial)\,\|_{0,\beta,\gamma;\Omega}:= 
\max\left(
\|E\|_{0,\beta,\gamma;\Omega}\, , \,  
b_0^{-1} \|E^\partial :C^{1,\beta}(\partial M_{\theta,m}  \cap\Omega \, ,\, \lambda^{-2} g)\|
\right). 
$$
\end{definition}

\begin{lemma}[Norms and operators comparison]
\label{L:norm-comparison}
\label{L:operator-comparison}
With the assumptions and the notation of 
\ref{D:global-norm} and \ref{Con:delta}, 
we have the following: 
\\
(i). 
If $\Omega\subset \Mtilde_{\theta,m}[0]$  
and $\uhat\in C^{k,\beta}_{sym}(\, \PiSch(\Omega) \,)$, 
then we have with $\epsilon=C\,(\ds+\dtheta)$: 
\\
$\phantom{kkk}$ 
(a). $ \|\uhat\circ\PiSch \|_{k,\beta,\gamma;\Omega} 
\, \le \, 
\lambda^{-k+1} \, 
\|\uhat \circ\PiSch :C^{k,\beta}(\, \Omega \, ,\,  \lambda^{-2} g \, ,\, e^{-\gamma s} \,)\, \|\, 
$
\hfill 
\\
$\phantom{kkk}$ 
\hfill  
$\,\, \sim_{(1+\epsilon)} \,\,  
\lambda^{-k+1} \, 
\|\uhat :C^{k,\beta}(\, \PiSch(\Omega) \, ,\,  g_\Sch \, ,\, e^{-\gamma s} \,)\, \|\,. 
$ 
\\
$\phantom{kkk}$ 
(b). 
$\| \psihat \{\Lcal (\uhat \circ \PiSch)  - \lambda^{-2} (\Lcal_\Sch \uhat) \circ \PiSch\}\|_{0,\beta,\gamma;\Omega} 
\leq \, \epsilon \, 
\lambda^{-1} \, \|\uhat :C^{2,\beta}(\, \PiSch(\Omega) \, ,\,  g_\Sch \, ,\, e^{-\gamma s} \,)\, \| $ ($k=2$). 
\\
(ii). 
If $\Omega\subset\Mtilde_{\theta,m}[1]$  
and $u'\in C^{k,\beta}_{sym}(\, \PiWtheta(\Omega) \,)$, 
then we have 
\\
$\phantom{kkk}$ 
(a). 
$
\|u'\circ\PiWtheta :C^{k,\beta}(\, \Omega \, ,\,  \lambda^{-2} g \, )\, \| 
\,\, \sim_{(1+C\lambda^7)}\,\,  
\|u' :C^{k,\beta}(\, \PiWtheta(\Omega) \, ,\,  \lambda^{-2} g \, )\, \|  
$.  
\\
$\phantom{kkk}$ 
(b).  
$ \| \psi' \{\Lcal (u' \circ \PiWtheta) - (\Lcal_{\Wcal_\theta} u') \circ \PiWtheta\}\|_{0,\beta,\gamma;\Omega} \leq 
C \lambda^7 \|u'\circ\PiWtheta \|_{2,\beta,\gamma;\Omega} $ ($k=2$).
\\
$\phantom{kkk}$ 
(c). 
$\Bcal (u' \circ \PiWtheta) = (\Bcal u') \circ \PiWtheta$ on $\Omega\cap\partial M_{\theta,m}  $ ($k=2$). 
\end{lemma}

\begin{proof}
(i) follows from \ref{D:global-norm}, \ref{E:norm-f12}, \ref{L:desing} and \ref{L:psi01}.iv. 
(ii) (a) follows from the fact that 
that $\Mtilde_{\theta,m}[1]$ is the graph of the function $\lambda \psi_{trun} \varphi_{wing}$ 
over a subset of $\Wcal_\theta\setminus\Circle_\theta$, 
and by using the definition of $\abar$ (so that $e^{-\abar}=\lambda^7$) and \ref{L:phi-wing} to estimate. 
(ii) (b) follows from the same observation together with \ref{L:psi01}.iv. Note that there is a scaling factor of $\lambda^2$ which 
is offset by the difference of the powers of $\lambda$ in $\| \cdot \|_{k,\beta,\gamma;\Omega}$ for $k=0,2$ in \ref{D:global-norm}. 
(ii) (c) follows from the observation that $M_{\theta,m}$ agrees with $\Wcal_\theta$ in a neighborhood of $\partial M_{\theta,m}$ 
(so the Jacobi operators $\Bcal$ for $M_{\theta,m}$ and $\Wcal_\theta$ coincide).
\end{proof}

We consider now the linearized equation \ref{E:LE} with $S=M_{\theta,m}$. 
We define the ``extended substitute kernel'' on $M_{\theta,m}  $ to be the span of $w \circ \PiSch$. 
Given $(E,E^\partial) \in C^{0,\beta}_{sym}(M_{\theta,m}  ) \times C^{1,\beta}_{sym}(\partial M_{\theta,m}  )$, 
we construct an approximate solution $u_1$ modulo $w \circ \PiSch$,  
by combining semi-local approximate solutions as follows:  
Note that by \ref{L:sym}.iii and \ref{L:psi01}.i and iii, 
$\psihat E \in C^{0,\beta}_{sym}(M_{\theta,m}  )$ is supported inside $\Mtilde_{\theta,m}[0]$. 
Recall from \ref{D:Pi} that $\PiSch$ is a diffeomorphism from $\Mtilde_{\theta,m}[0]$ onto its image in $\Schhat_m$.  
We define uniquely $\Ehat \in C^{2,\beta}_{sym}(\Schhat_m,g_\Sch,e^{-\gamma s})$ supported in $\Schhat_m \cap \{s \leq 5 \ds m\}$ by 
\beq
\label{E:Ehat}
\Ehat \circ \PiSch = \lambda^2 \psihat E \qquad \text{ on } \Mtilde_{\theta,m}[0].
\eeq 
By \ref{L:linear-Scherk} then, there exist unique $\uhat \in C^{2,\beta}_{sym}(\Schhat_m,g_\Sch,e^{-\gamma s})$ and $\mu_1 \in \R$ such that
\begin{equation}
\label{E:uhat}
\Lcal_\Sch \uhat  = \Ehat + \mu_1 w \qquad \text{ on } \Schhat_m.
\end{equation}
Note that $w$ descends to a function on $\Schhat_m$ by \ref{L:sym}.iv  
and that $\partial M_{\theta,m}  =M_{\theta,m}   \cap \Sph^2=\Wcal_\theta\cap\Sph^2$ by the definitions. 
We define uniquely $E' \in C^{0,\beta}_{sym}(\Wcal_\theta)$ 
supported on $\PiWtheta(\Mtilde_{\theta,m}[1])$, 
and $E'^\partial \in C^{1,\beta}(\partial M_{\theta,m}  )$,  
by requesting 
\begin{equation}
\label{E:E'}
E' \circ \PiWtheta = (1-\psihat^2) E- [\Lcal,\psihat] (\uhat \circ \PiSch) \text{ on } \Mtilde_{\theta,m}[1], \qquad E'^\partial \circ \PiWtheta 
=E' \text{ along } \partial M_{\theta,m} .
\end{equation}
Note that by \ref{L:psi01}.i, 
$(1-\psihat^2)E$ is supported on $M_{\theta,m}   \setminus M_{\theta,m}[0] \subset \Mtilde_{\theta,m}[1]$ 
and 
$[\Lcal,\psihat] (\uhat \circ \PiSch)$ is supported on 
$\Mtilde_{\theta,m}[0] \setminus M_{\theta,m}[0] \subset \Mtilde_{\theta,m}[0] \cap \Mtilde_{\theta,m}[1]$, 
and therefore $E'$ is in fact supported on $\PiW(\Mtilde_{\theta,m}[1] \setminus M_{\theta,m}[0])$. 
Finally by appealing to \ref{L:linear-ends} we define 
\begin{equation}
\label{E:u1}
u' :=\Rcal_{\Wcal_\theta}(E',E'^\partial) \qquad \text{ and } \qquad 
u_1:=\psihat (\uhat \circ \PiSch) + \psi' (u' \circ \PiWtheta).  
\end{equation}
Note that by \ref{L:psi01}.i-ii $\psihat (\uhat \circ \PiSch)$ 
and $\psi' (u' \circ \PiW)$ are supported in 
$\Mtilde_{\theta,m}[0]$ and $\Mtilde_{\theta,m}[1]$ respectively.  

\begin{definition}
\label{D:Rappr}
We define a linear map 
$\Rcal_{M,appr}:C^{0,\beta}_{sym}(M_{\theta,m}  ) \times C^{1,\beta}_{sym}(\partial M_{\theta,m}  ) \to C^{2,\beta}_{sym}(M_{\theta,m}  ) \times \R \times C^{0,\beta}_{sym}(M_{\theta,m}  ) \times C^{1,\beta}_{sym}(\partial M_{\theta,m}  )$   
by taking $\Rcal_{M,appr}(E,E^\partial)=(u_1,\mu_1,E_1,E_1^\partial)$, 
where $\mu_1$ was defined in \ref{E:uhat}, $u_1$ in \ref{E:u1}, and 
\begin{equation}
\label{E:error-appr}
E_1:= \Lcal u_1 - E - \mu_1 \lambda^{-2} (w \circ \PiSch), \qquad E_1^\partial:= \Bcal u_1 - E^\partial ,
\end{equation}
where $\Lcal$ and $\Bcal$ are the Jacobi operators for the initial surface $M_{\theta,m}  $ as in \ref{D:linear-operator}. 
\end{definition}

\begin{prop}[Linear estimates on $M_{\theta,m}$]
\label{P:linear-global}
Assuming \ref{Con:delta} a linear map 
$\Rcal_M:C^{0,\beta}_{sym}(M_{\theta,m}  ) \times C^{1,\beta}_{sym}(\partial M_{\theta,m}  ) \to C^{2,\beta}_{sym}(M_{\theta,m}  ) \times \R$ 
can be defined by
\[ \Rcal_M (E,E^\partial):=(u,\mu):=\sum_{n=1}^\infty (u_n,\lambda^{-1} \mu_n) \in C^{2,\beta}_{sym}(M_{\theta,m}  ) \times \R \]
for $(E,E^\partial) \in C^{0,\beta}_{sym}(M_{\theta,m}  ) \times C^{1,\beta}_{sym}(\partial M_{\theta,m}  )$, 
where the sequence $\{(u_n,\mu_n,E_n,E_n^\partial)\}_{n \in \N}$ is defined inductively for $n \in \N$ by
\[ (u_n,\mu_n,E_n,E^\partial_n):=-\Rcal_{M,appr}(E_{n-1},E_{n-1}^\partial), \qquad (E_0,E_0^\partial)=-(E,E^\partial).\]
Moreover the following hold.
\\
(i). $\Lcal u =E+\mu \lambda^{-1} w \circ \PiSch$ on $M_{\theta,m}  $, $\Bcal u = E^\partial$ along $\partial M_{\theta,m}  $.
\\
(ii). $\|u\|_{2,\beta,\gamma;M} +|\mu| \leq C \|\,(E,E^\partial)\, \|_{0,\beta,\gamma;M}$.
\\
(iii). $\Rcal_M$ depends continuously on the parameter $\theta$.
\end{prop}

\begin{proof}
By \ref{L:psi01}.i and iv, \ref{E:norm-multi}, 
\ref{E:gS-est}, 
and \ref{D:global-norm}, we have $\|\psihat E\|_{0,\beta,\gamma;M} \leq C \|E\|_{0,\beta,\gamma;M}$. 
By \ref{L:linear-Scherk} and \ref{E:gS-est},  
we have the following estimate for $\uhat$ and $\mu_1$:
\begin{equation}
\label{E:uhat-est}
\|\uhat:C^{2,\beta}(\Schhat_m,g_\Sch,e^{-\gamma s})\| + | \mu_1| 
\leq \|\Ehat:C^{0,\beta}(\Schhat_m,g_\Sch,e^{-\gamma s})\| 
\leq C \lambda \|E\|_{0,\beta,\gamma;M}.
\end{equation}
By \ref{E:E'}, \ref{L:psi01}.iv, \ref{L:norm-comparison}.i, and \ref{E:uhat-est}, 
we have that $\|E' \circ \PiWtheta\|_{0,\beta,\gamma;M} \leq C\|E\|_{0,\beta,\gamma;M}$. 
By \ref{L:norm-comparison}.ii, \ref{E:norm-scale}, \ref{E:norm-f12}, 
and since $f_0 \leq Cb_0$ on the support of $E'\circ \PiWtheta$ 
(which is contained in $\Mtilde_{\theta,m}[1] \setminus M_{\theta,m}[0] $), we have
\[ 
\|E':C^{0,\beta}(\Wcal_\theta, g \,)\, \| \leq C\lambda^{-\beta} \|E':C^{0,\beta}(\Wcal_\theta, \lambda^{-2}g \,)\, \| 
\leq C\lambda^{-1-\beta} b_0 \|E\|_{0,\beta,\gamma;M}.   \]
Estimating similarly $E^\partial$ and applying \ref{L:linear-ends}, we obtain the estimate 
\begin{equation}
\label{E:u'-est-} 
\|u':C^{2,\beta}(\Wcal_\theta, \lambda^{-2}g )\| 
\leq \|u':C^{2,\beta}(\Wcal_\theta, g )\| 
\leq C \lambda^{-2} b_0 \|\,(E,E^\partial)\,\|_{0,\beta,\gamma;M}.
\end{equation}
By \ref{L:norm-comparison}.ii.a, \ref{E:norm-f12}, and since $f_2 \geq \lambda^{-6} b_0$ (recall \ref{E:norm-f12}) we conclude that 
\begin{equation}
\label{E:u'-est} 
\|\, u'\circ\PiWtheta\,\|_{2,\beta,\gamma;M}
\leq C \lambda^{4} \|\,(E,E^\partial)\,\|_{0,\beta,\gamma;M}.
\end{equation}
Combining with 
\ref{E:u1}, \ref{L:psi01}.iv, \ref{L:norm-comparison}.i.a, and \ref{E:uhat-est} we have the estimate
\beq 
\label{E:u1-est}
\|u_1\|_{2,\beta,\gamma;M} \leq C \, \|\,(E,E^\partial)\,\|_{0,\beta,\gamma;M}.
\eeq
Using now \ref{E:u1} and \ref{E:error-appr} we obtain
\begin{eqnarray*}
E_1 &=&  \Lcal \{ \psihat (\uhat \circ \PiSch) \} + \Lcal \{ \psi' (u' \circ \PiWtheta) \} -E -\lambda^{-2} \mu_1 (w \circ \PiSch) \\
&=&  [\Lcal,\psihat](\uhat \circ \PiSch) +  \psihat \Lcal(\uhat \circ \PiSch) + [\Lcal,\psi'](u' \circ \PiWtheta) + \psi' \Lcal(u' \circ \PiWtheta) \\
&&-E -\lambda^{-2} \mu_1 (w \circ \PiSch) \\
&=& E_{1,I} +E_{1,II}+ E_{1,III} \\
&& +\lambda^{-2} \psihat (\Lcal_\Sch \uhat) \circ \PiSch +\psi' (\Lcal_{\Wcal_\theta} u') \circ \PiWtheta -E-\lambda^{-2} \mu_1 (w \circ \PiSch) + [\Lcal,\psihat](\uhat \circ \PiSch)
\end{eqnarray*}
where $E_{1,I}, E_{1,II}, E_{1,III} \in C^{0,\beta}_{sym}(M_{\theta,m}  )$ are supported respectively on $\Mtilde_{\theta,m} [1] \setminus M_{\theta,m}[1]$, $\Mtilde_{\theta,m}[0]$ and $\Mtilde_{\theta,m}[1]$ by \ref{L:psi01}.i and ii, and where they are defined by
\begin{equation}
\label{E:errors123}
\begin{array}{rl}
E_{1,I}&:= [\Lcal,\psi'] (u' \circ \PiWtheta), \\
E_{1,II}&:= \psihat \{ \Lcal(\uhat \circ \PiSch) -\lambda^{-2} (\Lcal_\Sch \uhat) \circ \PiSch\},\\
E_{1,III}&:= \psi' \{ \Lcal(u' \circ \PiWtheta) - (\Lcal_{\Wcal_\theta} u') \circ \PiWtheta \}.
\end{array}
\end{equation}
Using \ref{E:uhat}, \ref{E:u1}, \ref{L:sym}.iv, \ref{L:psi01} and that $E' \circ \PiWtheta$ is supported on $\Mtilde_{\theta,m}[1] \setminus M_{\theta,m}[0]$, 
the leftover terms on the right hand side above all cancel and we have the decomposition
\begin{equation}
\label{E:E1-decomp}
E_1=E_{1,I}+E_{1,II}+ E_{1,III}, 
\end{equation}
where the terms on the right hand side are defined in \ref{E:errors123}.

By \ref{E:errors123}, \ref{L:psi01}.iv, \ref{E:u'-est-},  
and $(\min f_0)^{-1}=e^{\gamma(\abar+1)} \leq C\lambda^{-7/\gamma}$ on $\Mtilde_{\theta,m}[1] \setminus M_{\theta,m}[1]$ (where $E_{1,I}$ is supported), 
we have 
\[ \|E_{1,I}\|_{0,\beta,\gamma;M} \leq C\lambda^{1-7/\gamma} \|\, u'\, :\, C^{2,\beta}(\Wcal_0 , \lambda^{-2} g \, )\, \| \leq 
C \lambda^{-1-7/\gamma} b_0 \|\,(E,E^\partial)\,\|_{0,\beta,\gamma;M} \,. \] 
By \ref{E:errors123}, \ref{L:operator-comparison}.i, and \ref{E:uhat-est}, we have
\[  \|E_{1,II}\|_{0,\beta,\gamma;M} \leq C \, (\ds+\dtheta)\, \lambda^{-1} \|\uhat:C^{2,\beta}(\Schhat_m,g_\Sch,e^{-\gamma s})\| \leq 
C \, (\ds+\dtheta)\, \|E\|_{0,\beta,\gamma;M}.\]
By \ref{E:errors123}, \ref{L:operator-comparison}.ii and \ref{E:u'-est},
\[ \|E_{1,III}\|_{0,\beta,\gamma;M} \leq \,C \,\lambda^{11} \, \|\,(E,E^\partial)\,\|_{0,\beta,\gamma;M} \,. \] 
Combining these estimates and by the decomposition \ref{E:E1-decomp}, we conclude  
\[ \|E_{1}\|_{0,\beta,\gamma;M} \, \leq \,C \,( \lambda^{-1-7/\gamma} b_0 + \ds + \dtheta +  \lambda^{11} ) \, 
\|\,(E,E^\partial)\,\|_{0,\beta,\gamma;M} \, 
\leq \frac12 \|\,(E,E^\partial)\,\|_{0,\beta,\gamma;M} \,, \] 
where for the last inequality we assumed that $\ds$ and $\dtheta$ 
are small enough in absolute terms and also that $m$ is large enough in accordance with \ref{Con:delta}. 
By \ref{L:norm-comparison}.ii.c, \ref{E:error-appr}, \ref{E:u1}, \ref{E:E'}, we have $E_1^\partial=0$. 
Arguing inductively we conclude that $\forall n\in\N$ we have $E_n^\partial=0$ and 
\[ \|E_{n}\|_{0,\beta,\gamma;M} \, \leq 
2^{-n} \|\,(E,E^\partial)\,\|_{0,\beta,\gamma;M} \,. \] 
The proof is then completed by using the earlier estimates.
\end{proof}


\section{Nonlinear terms and the fixed point argument}
\label{S:main}

In this section, we will give uniform estimates on the nonlinear terms of the mean curvature and the intersection function for the twisted graph of a function $\varphi$ over an initial surface $M=M_{\theta,m}$ (recall \ref{D:twisted-graph}). Then, we combine the results from all previous sections to prove the main theorem \ref{T:finalthm}, which implies \ref{T:mainthm} in the introduction.

\subsection*{The nonlinear terms}

We now prove a global version of the uniform estimates (\ref{P:local-estimate-Theta} and \ref{P:local-estimate-H}) 
for the mean curvature and the intersection function for twisted graphs of a function over our initial surfaces 
when the function is small with respect to the global weighted norms defined in \ref{D:global-norm}. 
Note that if $\varphi \in C^{2,\beta}_{sym}(M)$ with $\varphi$ sufficiently small so that 
(recall \ref{D:twisted-graph}) the twisted graph 
$\widetilde{\Graph}[\epsilon_0,\varphi;M] \subset \R^3$ is well-defined,  
then $\widetilde{\Graph}[\epsilon_0,\varphi;M]$ is $\group_m$-invariant by \ref{L:sym}.i. 
The $\epsilon_0>0$ can be chosen to be a sufficiently small universal constant since all our initial surfaces $M_{\theta,m}$ are free boundary minimal surfaces \emph{near $\partial \mathbb{B}^3$} with uniformly bounded geometry. 

\begin{prop}
\label{P:nonlinear-estimate}
There exists a universal constant $\epsilon_0>0$ such that if $M=M_{\theta,m}$ 
is as in \ref{D:initial} and $\varphi \in C^{2,\beta}(M)$ satisfies $\|\varphi\|_{2,\beta,\gamma;M} \leq \epsilon_0$, 
then $\varphi$ is admissible on $M$ (recall \ref{D:admissible}), 
$\widetilde{\Graph}[\epsilon_0,\varphi;M]$ is well defined and properly embedded.
Moreover, if $H_\varphi$ is the mean curvature of $\widetilde{\Graph}[\epsilon_0,\varphi;M]$ pulled back to $M$ by $\widetilde{\Immer}[\epsilon_0,\varphi;M]$,    
$H$ is the mean curvature of $M$, 
and $\Theta_\varphi$ is the perturbed intersection function as a function on $\partial M$ as in \ref{D:perturbed-Theta} 
for the proper immersion $\widetilde{\Immer}[\epsilon_0,\varphi;M]$, 
then we have (recall \ref{D:linear-operator})
\[ \| \, ( \, H_\varphi - H -\Lcal \varphi \, , \, \Theta_\varphi - \Bcal \varphi  \, ) \, \|_{0,\beta,\gamma;M}  \leq \,  C \, \|\varphi\|_{2,\beta,\gamma;M}^2. \]
\end{prop}

\begin{proof}
The complete surface $\Sigmatilde_{\theta,m}$ has injectivity radius larger than $1/10$ with respect to the metric $\lambda^{-2} g$. Notice that each initial surface $M_{\theta,m}$ is a free boundary minimal surface in a neighborhood of $\partial \B^3$ (recall \ref{R:H-linearized}). Let $B_p$ be the geodesic ball of radius $1/100$ in $(\Sigmatilde_{\theta,m},\lambda^{-2}g)$ where $p \in M$. It is clear that we can define an immersion $X:B^2(2) \to (B_p,\lambda^{-2} g)$ such that it satisfies \ref{E:C-bounded-a} for some universal constant $c_1>0$. Therefore, \ref{P:local-estimate-H} and scaling implies that if $\| \lambda^{-1} \varphi:C^{2,\beta}(B_p,\lambda^{-2}g)\| < \epsilon_H(c_1)$, then 
\[ \| H_\varphi - H -\Lcal \varphi:C^{0,\beta}(B_p,\lambda^{-2}g) \| \leq C \lambda^{-4} \|  \varphi:C^{2,\beta}(B_p,\lambda^{-2}g)\|^2.\]
The estimate for the mean curvature then follows from \ref{D:global-norm} and that $\lambda^{-1} f_0(p)^{-1} f_2(p)^2 \leq \lambda^{-7} b_2 \leq C$. To estimate the intersection function, take $B_p$ to be centered at $p \in \partial M$ and one can define similarly an immersion $X:B^2_+ \to (B_p \cap M, \lambda^{-2} g)$ which has an extension $X:B^2(2) \to (B_p,\lambda^{-2} g)$ with $c_1$-bounded geometry as in \ref{D:C-bdd-geom} for some universal constant $c_1>0$. Thus,  \ref{P:local-estimate-Theta} and scaling implies that if $\| \lambda^{-1} \varphi:C^{2,\beta}(B_p \cap M,\lambda^{-2}g)\| < \epsilon_{\Theta}(c_1)$, then
\[ \| \Theta_\varphi -\Bcal \varphi: C^{1,\beta}(B_p \cap \partial M,\lambda^{-2} g)\| \leq C \lambda^{-2} \| \varphi:C^{2,\beta}(B_p \cap M,\lambda^{-2}g)\|^2.\]
The estimate for the intersection function then follows from \ref{D:global-norm} and that $b_0^{-1} f_2(p)^2 \leq C$. By our construction it is clear that the twisted graph $\widetilde{\Graph}[\epsilon_0,\varphi;M]$ is globally embedded. 
This finishes the proof of the proposition.
\end{proof}

\subsection*{The main theorem}

\begin{theorem}
\label{T:finalthm}
There is an absolute constants $\epsilon_0>0$ and $\Cbar>0$ such that if $m$ is sufficiently large depending on $\Cbar$, 
then there exists $\theta$ with $|\theta| \leq \Cbar m^{-1}$ and 
$\varphi \in C^{2,\beta}_{sym}(M)$ with $\|\varphi\|_{2,\beta,\gamma;M} \leq 2 \Cbar m^{-1}$, 
where $M=M_{\theta,m}$ is as in \ref{D:initial} and 
$\| \cdot \|_{2,\beta,\gamma;M}$ is as in \ref{D:global-norm}, 
such that $\Sigma_{m-1}:=\widetilde{\Graph}[\epsilon_0,\varphi;M]$ (recall \ref{D:twisted-graph})  
is a properly embedded free boundary minimal surface in $\B^3$
satisfying
\ref{T:mainthm}. 
\end{theorem}

\begin{proof}
As usual \cite{Kapouleas90a,kapouleas:wente,Kapouleas97,Kapouleas-Yang10,Kapouleas14} 
the proof uses Schauder's fixed point theorem \cite[Theorem 11.1]{Gilbarg-Trudinger01}.  
This theorem asserts that any continuous mapping (not necessarily linear) from a compact convex set in a Banach space into itself must have a fixed point. 
The elemensts of the Banach space in our case are functions defined on the initial surfaces together with the unbalancing parameter $\theta$.  

Let $\Cbar>0$ be a constant to be chosen sufficiently large in absolute terms later. 
Let $m$ be a fixed positive interger which is sufficiently large in terms of $\Cbar$ as in \ref{Con:parameters}. 
We assume $|\theta| \leq \Cbar m^{-1}$ and write $M_\theta=M_{\theta,m}$ throughout the proof.

\textit{Step 1: Identifying the function spaces:} In order to define a continuous function on a fixed Banach space 
(independent of $\theta$), we have to first identify functions defined on different initial surfaces. 
Using the diffeomorphisms $\Fcal_{\Wcal_\theta}$ and $\Zcalhat_{\theta,m}$ defined in \ref{D:FW} and \ref{E:Zcalhat} respectively, 
we can construct a family of smooth $\group_m$-equivariant diffeomorphisms $\Fcal_\theta:M_0 \to M_\theta$ such that if $m$ is sufficiently large in terms of $\Cbar$, 
for any $\varphi \in C^{2,\beta}_{sym}(M_\theta)$,
\beq
\label{E:initial-equiv-norm}
\|\varphi\|_{2,\beta,\gamma;M_\theta} \sim_2 \|\varphi \circ \Fcal_\theta\|_{2,\beta,\gamma;M_0}, 
\eeq
where $\|\varphi\|_{2,\beta,\gamma;M_\theta}$ denotes the weighted global norm defined on $M_\theta$ in \ref{D:global-norm}.

\textit{Step 2: The compact convex set $K$:} Consider the set
\beq
\label{E:set-K}
K:=\{ (\varphi,\theta) \in C^{2,\beta}_{sym}(M_0) \times \R : \|\varphi\|_{2,\beta,\gamma;M_0} \leq \Cbar m^{-1}, |\theta| \leq \Cbar m^{-1} \}.
\eeq
We claim that $K$ is a compact convex subset of the Banach space $C^{2,\beta'}_{sym}(M_0) \times \R$ for any fixed $\beta' \in (0,\beta)$. Convexity is obvious since $\| \cdot \|_{2,\beta,\gamma}$ is a norm. Compactness follows from  Arzela-Ascoli's theorem since $M_0$ is compact.

\textit{Step 3: Defining the map $\Jcal$:} We define a nonlinear map 
\[ \Jcal:K \to C^{2,\beta'}_{sym}(M_0) \times \R \] 
as follows: 
Given $(\varphi,\theta) \in K$, let $\varphi_\theta \in C^{2,\beta}_{sym}(M_\theta)$ be defined by 
\beq
\label{E:varphi-theta}
\varphi_\theta:=\varphi \circ \Fcal_\theta^{-1} - \varphi_H,
\eeq
where $(\varphi_H,\theta_H) := \Rcal_M(H-\theta \lambda^{-1} w \circ \PiSch,0)$ with $\Rcal_M$ be the linear map defined in \ref{P:linear-global} and $H$ is the mean curvature of $M_\theta$. Therefore, we have
\beq
\label{E:varphi-H}
\Lcal \varphi_H = H + (\theta_H - \theta) \lambda^{-1} w \circ \PiSch \quad \text{ on } M_\theta, \qquad
\Bcal \varphi_H = 0 \quad \text{ along } \partial M_\theta.
\eeq
Let us assume first that $\| \varphi_\theta \|_{2,\beta,\gamma;M_\theta} < \epsilon_0$ where $\epsilon_0$ is the constant in \ref{P:nonlinear-estimate}, then the graph $\Graph[\varphi_\theta,M_\theta]$ is well-defined. Let $H_{\varphi_\theta}$ and $\Theta_{\varphi_\theta}$ be the mean curvature and intersection function of $\Graph[\varphi_\theta,M_\theta]$ pulled back to functions on $M_\theta$ and $\partial M_\theta$ respectively as in \ref{P:nonlinear-estimate}. Next, we define
\beq
\label{E:phitilde}
 (\phitilde,\thetatilde):=\Rcal_M(H_{\varphi_\theta}-H-\Lcal \varphi_\theta, \Theta_{\varphi_\theta} - \Bcal  \varphi_\theta),
\eeq
which by \ref{P:linear-global} satisfies
\beq
\label{E:phi-tilde}
\Lcal  \phitilde =H_{\varphi_\theta}-H-\Lcal \varphi_\theta +\thetatilde \lambda^{-1} w \circ \PiSch \quad \text{ on } M_\theta, \qquad
\Bcal \phitilde =\Theta_{\varphi_\theta} - \Bcal \varphi_\theta \quad \text{ along } \partial M_\theta. 
\eeq
Finally, we define $\Jcal(\varphi,\theta):=(-\phitilde \circ \Fcal_\theta,\theta_H + \thetatilde)$.

\textit{Step 4: $\Jcal$ is a well-defined contraction map:} 
We show now that if $\Cbar$ is sufficiently large in absolute terms and $m$ is sufficiently large in terms of $\Cbar$, 
then the map $\Jcal$ in step 2 is well-defined and is a contraction map on $K$, 
that is $\Jcal(K) \subset K$.

The first thing to check is that $\|\varphi_\theta\|_{2,\beta,\gamma;M_\theta} < \epsilon_0$ so that \ref{P:nonlinear-estimate} can be applied. 
Recall from the construction that $M \setminus \Sigma_{\theta,m} \subset \Wcal_\theta$ is minimal so $H -\theta \lambda^{-1} w \circ \PiSch$ is supported on $\Sigma_{\theta,m}=\Mtilde_{\theta,m}[0]$. 
Using \ref{D:global-norm} and \ref{L:norm-comparison}.i.a, we can rewrite the estimate in \ref{L:desing-H-estimate} as
\beq
\label{E:HH-estimate}
 \|H -\theta \lambda^{-1} w \circ \PiSch\|_{0,\beta,\gamma;M_\theta} \leq C(m^{-1}+|\theta|^2).
 \eeq
Since $|\theta| \leq \Cbar m^{-1}$, by \ref{P:linear-global} we have $|\theta_H|+\|\varphi_H\|_{0,\beta,\gamma;M_\theta} \leq C(m^{-1} + \Cbar^2 m^{-2})$.
Therefore, using \ref{E:initial-equiv-norm}, \ref{E:varphi-theta} and that $(\varphi,\theta) \in K$ as in \ref{E:set-K}, we get
\beq
\label{E:varphi:theta}
\|\varphi_\theta\|_{2,\beta,\gamma;M_\theta} \leq C \|\varphi\|_{2,\beta,\gamma;M_0} + \|\varphi_H\|_{2,\beta,\gamma;M_\theta} \leq C \Cbar^2 m^{-1}. 
\eeq
From the estimate above, we see that if $m$ is sufficiently large in terms of $\Cbar$, then we would have $\|\varphi_\theta\|_{2,\beta,\gamma;M_\theta} < \epsilon_0$ in \ref{P:nonlinear-estimate} and hence $\Graph[\epsilon_0,\varphi_\theta ; M_\theta]$ is a well-defined embedding and we have the quadratic estimate
\beq
\label{E:quad-estimate}
\| \, ( \, H_{\varphi_\theta}-H -\Lcal \varphi_\theta \, , \, \Theta_{\varphi_\theta} - \Bcal \varphi_\theta \, ) \, \|_{0,\beta;\partial M_\theta} \leq C \|\varphi_\theta\|_{2,\beta,\gamma;M_\theta}^2.
\eeq
Now, combining \ref{E:varphi-H}, \ref{E:phi-tilde}, and \ref{E:varphi-theta}, we have 
\beq
\label{E:H-final}
 H_{\varphi_\theta} = \Lcal (\phitilde+\varphi \circ \Fcal^{-1}_\theta) + (\theta-\theta_H-\thetatilde) \lambda^{-1} w \circ \PiSch  \qquad \text{ on } M_\theta,
\eeq
\beq
\label{E:Theta-final}
\Theta_{\varphi_\theta}= \Bcal (\phitilde+\varphi \circ \Fcal^{-1}_\theta)  \qquad \text{ along } \partial M_\theta.
\eeq
Recall that $\Jcal(\varphi,\theta):=(-\phitilde \circ \Fcal_\theta,\theta_H + \thetatilde)$. 
By \ref{E:phitilde}, \ref{P:linear-global}, \ref{E:initial-equiv-norm}, \ref{E:varphi:theta}, and \ref{E:quad-estimate}, we have the estimate
\[ |\thetatilde|+ \|\phitilde \circ \Fcal_\theta\|_{2,\beta,\gamma;M_0} \leq C \|  \varphi_\theta\|^2_{2,\beta,\gamma;M_\theta} \leq C \Cbar^4 m^{-2}.  \]
Using this, \ref{P:linear-global} and \ref{E:HH-estimate}, we have $|\theta_H+\thetatilde| \leq Cm^{-1} +C \Cbar^4 m^{-2}$.
Therefore, if we first choose $\Cbar$ sufficiently large in absolute terms and then choose $m$ large enough in terms of $\Cbar$, we can arrange that $\Jcal(K) \subset K$.

\textit{Step 5: The fixed point argument:} From step 4 we have a well-defined contraction map $\Jcal:K \to K$ where $K$ is a compact convex subset of a Banach space by step 2. Continuity of $\Jcal$ follows from the definitions and the continuity of the linear maps $\Rcal_M$ in $\theta$ and the diffeomophisms $\Fcal_\theta$. Therefore, we can apply Schauder's fixed point theorem to obtain a fixed point $(\varphi^*,\theta^*) \in K$ of $\Jcal$. The proof is finished once we show that the graph of $\varphi^*$ over the initial surface $M_{\theta^*}$ is a minimal surface intersecting $\Sph^2$ orthogonally. In other words, we have to show that $H_{\varphi^*} \equiv 0$ and $\Theta_{\varphi^*} \equiv 0$.

Since $(\varphi^*,\theta^*)$ is a fixed point of $\Jcal$, 
which means $\varphi^*= -\phitilde \circ \Fcal_\theta$ and $\theta^*=\theta_H + \thetatilde$. 
Hence, we have $\phitilde+\varphi^* \circ \Fcal^{-1}_\theta=0$ and $\theta-\theta_H-\thetatilde=0$. 
By \ref{E:H-final} and \ref{E:Theta-final} respectively, we get $H_{\varphi^*} \equiv 0$ and $\Theta_{\varphi^*} \equiv 0$ 
and the proof is complete. 
\end{proof}


\appendix

\section{Local exponential map estimates}
\label{A:exp}

\begin{prop}
\label{P:exp-estimate}
Let $g$ be a Riemannian metric on $B^n$ with coordinates $x^1,\cdots,x^n$. Let $g_{ij}:=g(\partial_{x^i},\partial_{x^j})$ be the metric components in this coordinate system, $g^{ij}$ be the inverse and $\Gamma_{ij}^k$ be the Christoffel symbols. Suppose that
\beq
\label{E:A1} \| g_{ij} : C^4 (B^n,g_0)\| \leq c_1,  \qquad \text{and} \qquad c_1^{-1} g_0 \leq g 
\eeq
for some constant $c_1>1$, then there exists a constant $C$ depending on $c_1$ (and $n$) such that
\beq
\label{E:A2}
\|g^{ij}:C^4(B^n,g_0)\| \leq C, \qquad \|\Gamma_{ij}^k :C^3(B^n,g_0)\| \leq C,
\eeq
and that the exponential map $\exp: B^n_{1-2C^{-1}} \times B^n_{C^{-1}} \to B^n$ with respect to $g$ is a well defined $C^3$ map such that for any multi-indices $I$, $J$ with $|I|+|J| \leq 3$, we have the pointwise estimates
\beq
\label{E:A3}
|\partial^{|I|}_{x^I} \partial^{|J|}_{v^J} (\exp(x,v) - x -v )| \leq C \, |v|^{\max(2-|J|,0)},
\eeq
where $| \cdot |$ denotes the norm of a vector with respect to the Euclidean metric $g_0$.
\end{prop}

\begin{proof}
The inverse of a matrix $A$ is given by $A^{-1}=(\det A)^{-1} \text{adj}(A)$ where $\text{adj}(A)$ is the adjoint matrix of $A$. From \ref{E:A1} we get the estimate \ref{E:A2} and that the metrics $g \sim_{c_1} g_0$ are uniformly equivalent (recall \ref{D:sim}). From the definition of exponential map:
\[ \exp(x,v):=\gamma_{x,v}(1), \]
where $\gamma_{x,v}(t):[0,1] \to B^n$ is the unique geodesic (relative to $g$) starting at $x$ with initial velocity $v$, 
that is $\gamma_{x,v}(0)=x$ and $\gamma'_{x,v}(0)=v$. 
In other words, $\gamma_{x,v}$ is the unique solution to the geodesic equation with such initial conditions (here $\gamma=(\gamma^1,\cdots,\gamma^n)$ are the coordinate expression of $\gamma$):
\begin{equation}
\label{E:geodesic}
 \left\{ \begin{array}{cr}
(\gamma^k)''(t)=\Gamma^k_{ij}(\gamma(t)) (\gamma^i)'(t) (\gamma^j)'(t), & k=1,2,\cdots,n\\
\gamma(0)=x, \; \gamma'(0)=v. & 
\end{array} \right. 
\end{equation}
By standard ODE theory and \ref{E:A2}, the exponential map $\exp$ is well defined for $(x,v) \in B^n_{1-2C^{-1}} \times B^n_{C^{-1}}$ for some constant $C$ depending on $c_1$.

It remains to prove \ref{E:A3}. The smoothness of the exponential map is a direct consequence of the smooth dependence on initial conditions $(x,v)$ for the solutions to the ODE system \ref{E:geodesic}. We will show how to get $C^1$-bounds here. The proof for higher derivatives are similar.

Let $| \cdot |$ and $\|\cdot \|$ be the norm of a vector with respect to $g_0$ and $g$ respectively. Since $\gamma$ is a geodesic, $\|\gamma'(t)\| \equiv \|\gamma'(0)\|=\|v\|$. Using \ref{E:A1} and $g_0 \sim_{c_1} g$, we have $| \gamma'(t)| \leq C |v|$. By Taylor's theorem, $g_0 \sim_{c_1} g$, \ref{E:A2} and \ref{E:geodesic}, we have the $C^0$-estimate:
\[ |\exp(x,v)-x-v|=| \gamma(1)-\gamma(0)-\gamma'(0)| \leq \max_{t \in [0,1]} |\frac{1}{2}\gamma''(t)| \leq C \max_{t \in [0,1]} |\gamma'(t)|^2 \leq C |v|^2.\]

For estimates on the derivatives, we differentiate the system \ref{E:geodesic}. For example, differentiating with respect to some $x^a$:
\beq
\label{E:A4}
 \left\{ \begin{array}{c} (\partial_{x^a} \gamma^k)'' =(\gamma^i)'(\gamma^j)'(\nabla \Gamma_{ij}^k \cdot \partial_{x^a} \gamma)+2\Gamma_{ij}^k (\gamma^i)' (\partial_{x^a} \gamma^j)', \\
\partial_{x^a} \gamma^k (0)=\delta^k_a, \; \partial_{x^a} (\gamma^k)'(0)=0. \end{array} \right. 
\eeq
Recall Kato's inequality that $|\alpha(t)|' \leq |\alpha'(t)|$ for any curve $\alpha(t)$ in $\mathbb{R}^n$, using \ref{E:A4} and \ref{E:A2}, we have 
$|\partial_{x^a} \gamma'|' \leq |\partial_{x^a} \gamma''| \leq C |v|^2 |\partial_{x^a} \gamma|+C |v| |\partial_{x^a} \gamma'|$. 
If we define the function $G:[0,1] \to \mathbb{R}$ by
\[ G(t) := \max_{s \in [0,t]} | \partial_{x^a} \gamma'(s)|, \]
then $G$ is a non-negative monotone increasing function hence differentiable a.e. and from \ref{E:A4}, we have the differential inequality
\[ G'(t) \leq C |v |^2 (1+G(t)) + C |v |G(t) \leq C|v|^2+ C |v| G(t) \]
with $G(0)=0$. Integrating the differential inequality gives
\[ G(t) \leq | v| (e^{C|v| t} -1) \leq C |v|^2,\]
provided $|v|$ is sufficiently small (but depending only on $c_1$). From this we have the pointwise estimate
\[ | \partial_{x^a} (\exp (x,v) - x -v)| \leq C|v|^2.\]

The estimate on the derivatives with respect to $v$ can be obtained similarly. We differentiate \ref{E:geodesic} with respect to some $v^a$
\beq
\label{E:A5}
 \left\{ \begin{array}{c} (\partial_{v^a} \gamma^k)'' =(\gamma^i)'(\gamma^j)'(\nabla \Gamma_{ij}^k \cdot \partial_{v^a} \gamma)+2\Gamma_{ij}^k (\gamma^i)' (\partial_{v^a} \gamma^j)', \\
\partial_{v^a} \gamma^k (0)=0, \; \partial_{v^a} (\gamma^k)'(0)=\delta^k_a. \end{array} \right. 
\eeq
Define $G(t): = \max_{s \in [0,t]} | \partial_{v^a} \gamma'(s)|$, we argue as before to obtain the differential inequality $G' \leq  C |v| G$
with initial condition $G(0)=1$, which implies that $G(t) \leq 1+C|v|$ provided that $|v|$ is sufficiently small (depending only on $c_1$). This implies the estimate
\[ | \partial_{v^a} (\exp (x,v) - x -v)| \leq C|v|.\]
Higher order derivative estimates can be obtained in a similar manner.
\end{proof}

As a corollary of the above exponential map estimates, one can prove the following lower bound on the injectivity radius.

\begin{corollary}
\label{C:inj-bound}
Under the same assumption as in \ref{P:exp-estimate}, then for all $x \in B^n_{1-2C^{-1}}$,
\[ \inj_x(B^n,g) \geq C^{-1} .\]
\end{corollary}
\begin{proof}
Since $D_v \exp (x,0)=id$ for all $x$, using the estimates in \ref{P:exp-estimate}, $D_v \exp(x,v)$ is a non-singular matrix for all $|v| \leq C^{-1}$. From this the assertion follows since we are looking at a local coordinate patch.
\end{proof}



\bibliographystyle{amsplain}
\bibliography{references-nicos}
\end{document}